\documentclass[12pt]{amsart}

\usepackage{amsmath,amsthm,amsfonts,amssymb,bm,graphicx, mathrsfs, fullpage}

\usepackage{comment}
\usepackage{enumitem}
\usepackage{tikz-cd}
\usepackage{adjustbox}
\usepackage{soul}
\usepackage{hyperref}
\usepackage[all,cmtip]{xy}
\usepackage{enumitem}
\usepackage{stmaryrd}

\hypersetup{colorlinks=true,citecolor=blue,linkcolor=blue,urlcolor=blue,
pdfstartview=FitH }

\newcommand{\bG}{\mathbf{G}}
\newcommand{\bH}{\mathbf{H}}
\newcommand{\bB}{\mathbf{B}}
\newcommand{\bPB}{\mathbf{PB}}
\newcommand{\bT}{\mathbf{T}}

\newcommand{\SO}{\textrm{SO}}

\providecommand{\C}{}
\renewcommand{\C}{{\mathbb C}}
\newcommand{\D}{{\mathbb D}}
\newcommand{\N}{{\mathbb N}}
\newcommand{\R}{{\mathbb R}}
\newcommand{\Z}{{\mathbb Z}}
\newcommand{\Q}{{\mathbb Q}}

\newcommand{\A}{{\mathbb A}}

\newcommand{\SL}{\operatorname{SL}}
\newcommand{\GL}{\operatorname{GL}}

\renewcommand{\Re}{\operatorname{Re}}

  \newcommand {\nf}{{\mathfrak n}}

   \newcommand {\pf}{{\mathfrak p}}
   
   \newcommand {\sk}{{\mathfrak s}}

\newcommand{\F}{{\mathbb F}}
\newcommand{\scD}{\mathscr D}
\newcommand{\scO}{\mathscr O}
\newcommand{\Mat}{\mathrm{Mat}_{2\times 2}}
\renewcommand{\restriction}{|}
\newcommand{\Nr}{\operatorname{Nr}}

\theoremstyle{definition}
\newtheorem{remark}{Remark}[section]

\newtheorem{defn}{Definition}

\newtheorem{example*}{Example}

\theoremstyle{plain}
\newtheorem{theorem}{Theorem}[section]
\newtheorem{conj}{Conjecture}
\newtheorem{cor}[theorem]{Corollary}
\newtheorem{lemma}[theorem]{Lemma}
\newtheorem{prop}[theorem]{Proposition}
\numberwithin{equation}{section}

\begin{document}

\author{Valentin Blomer}
\author{Farrell Brumley}
\author{Ilya Khayutin}
 
\address{Mathematisches Institut, Endenicher Allee 60, 53115 Bonn, Germany}
\email{blomer@math.uni-bonn.de}

\address{Sorbonne Universit\'e, Universit\'e Paris Cit\'e, CNRS, IMJ-PRG, F-75005 Paris, France}
\email{brumley@imj-prg.fr}
  
\title[The mixing conjecture]{The mixing conjecture under GRH}

\thanks{The first author was supported in part by Germany's Excellence Strategy grant EXC-2047/1 - 390685813 and ERC Advanced Grant 101054336. The second author is supported by the Institut Universitaire de France and ANR-FNS Grant  ANR-24-CE93-0016. The third author has been supported by National Science Foundation  Grant No.\  DMS-1946333, a Sloan Research Fellowship and an AMS Centennial Fellowship.}

\begin{abstract} 
We prove the mixing conjecture of Michel--Venkatesh for the class group action on Heegner points of large discriminant on compact arithmetic surfaces attached to maximal orders in rational quaternion algebras. The proof is conditional on the generalized Riemann hypothesis, and, when the division algebra is indefinite, we furthermore assume the Ramanujan conjecture. We establish the mixing conjecture for the discrete spectrum of the modular surface as well, under the same conditions. Our methods, which provide an effective rate, are based on the spectral theory of automorphic forms and their $L$-functions, together with sieve methods and multiplicative functions.
\end{abstract} 

\setcounter{tocdepth}{1}  

\maketitle

\section{Introduction}

A celebrated theorem of Duke \cite{Du} and Golubeva--Fomenko \cite{GF} states that the primitive integer points of norm $\sqrt{d}$ equidistribute when projected to the unit sphere $S^2$, as $d\to\infty$ along sequences avoiding local obstructions. More precisely, let $\Z_{\rm prim}^3=\{x\in\Z^3: \gcd(x_1, x_2, x_3) = 1\}$ denote the primitive points of the standard integral lattice in $\R^3$, and put 
\[
R_d = \{x \in \Z_{\rm prim}^3 : \, x_1^2 + x_2^2 + x_3^2 = d\}.
\]
Then $\D=\{d \in \Bbb{N} : d \not\equiv 0, 4, 7 \!\!\pmod 8\}$ is the set of locally admissible integers. The aforementioned authors, following a breakthrough of Iwaniec \cite{Iw} on the estimation of Fourier coefficients of half-integral weight Maass forms, showed that, as $d\in\D$ tends to infinity, the rescaled sets $d^{-1/2} R_d$ equidistribute on $S^2$ with respect to the normalized rotationally invariant measure $m$, with a power-saving rate of convergence. Using ergodic methods, Linnik \cite{Li2} had previously established  a similar result that required an auxiliary congruence condition on the set of eligible integers $d$. In fact, Linnik showed that his congruence condition could be removed under the assumption of the generalized Riemann hypothesis (GRH), and his method provided a logarithmic rate of convergence under the GRH assumption. The work of Duke and Golubeva--Fomenko can therefore be seen as rendering Linnik's result unconditional and strongly improving the rate of convergence. 

In recent years a growing body of work has been devoted to the topic of going beyond uniform distribution of integral points on the sphere and exploring their fine-scale spatial statistics, especially when measured by deviation, variance, and energy estimates \cite{EMV, BRS, HR, Shu}. 
Central to Linnik's approach is the action of the Picard group of the quadratic order of discriminant $-d$ on the set of primitive integral points $R_d$. This paper addresses an  ergodic theoretic property of the class group trajectories on these integral points, namely, the \textit{mixing conjecture} of Michel and Venkatesh \cite[Conjecture 2]{MV}, which we now describe. For convenience of exposition, we shall restrict ourselves to the set $\D^\flat$ of \textit{square-free} locally admissible integers.

\subsection{The mixing conjecture}\label{sec:mix-conj}

The set $R_d$ enjoys some natural symmetries coming from the finite rotation group $\SO_3(\Z)$. It will be more convenient to work modulo these symmetries, setting
\begin{equation}\label{eq:def-Rd}
\mathcal{R}_d= \Gamma 
\backslash R_d\quad\textrm{ and }\quad \mathcal{S}^2=\Gamma 
\backslash S^2,
\end{equation}
where $\Gamma \subseteq \SO_3(\Z)$ is a subgroup of $\SO_3(\Z)$ of index at most 2 as described in \cite[Section 3]{EMV}.
Let $\mu$ denote the pushforward measure of $m$ under the quotient map $S^2\rightarrow \mathcal{S}^2$. The equidistribution of $d^{-1/2}R_d$ on $(S^2,m)$ is then equivalent to that of $d^{-1/2}\mathcal{R}_d$ on $(\mathcal{S}^2,\mu)$.

A key structural observation, which undergirds both the classical cardinality estimate $|R_d| = d^{1/2 + o(1)}$ and Duke's equidistribution theorem, is the existence, for $d>3$, of a free action on $\mathcal{R}_d$ of the class group ${\rm Pic}(\scO_E)$ of the ring of integers $\scO_E$ of the imaginary quadratic field $E=\Q(\sqrt{-d})$. If $d\equiv 1,2\pmod{4}$, then $\mathcal{R}_d$ splits into two class group orbits, and when $d\equiv 3\pmod{8}$ the action is transitive.
This fact was discovered by Gau{\ss} \cite{Ga} and put into a robust algebraic framework by Venkov \cite{Ve}, using the theory of optimal embeddings of quadratic number fields into quaternion algebras. We may then view $\mathcal{R}_d$, equipped with the uniform probability measure, as a measure-preserving dynamical system under the action of ${\rm Pic}(\mathscr{O}_E)$.
 
In the seminal work \cite{Fu}, Furstenberg introduced a theory of arithmetic operations on dynamical systems, which relies on the following key concept. Given two Borel probability spaces $(X,\mu_X)$ and $(Y,\mu_Y)$ on which an abelian group $S$ acts by measure-preserving transformations, a \textit{joining} between them is an $S$-invariant Borel probability measure on the product space $X \times Y$, whose marginals are $\mu_X$ and $\mu_Y$. Joinings of non-isomorphic systems are related to questions of simultaneous equidistribution \cite{EL}. By contrast, a certain class of self-joinings of a single system is related to the ergodic theoretic notion of mixing. Indeed, given $(X,\mu_X)$ as above, and an element $s\in S$, the associated \textit{off-diagonal} self-joining $(X\times X,\mu_{X,s}^\Delta)$, supported on the graph of $s$, is defined by the rule $\mu_{X,s}^\Delta(A,B)=\mu_X(A\cap s^{-1} B)$. With $S$ assumed to be infinite, the transformation $s$ is mixing on $X$ precisely when the sequence of off-diagonals $\{\mu_{X,s^n}^\Delta\}_{n\in\N}$ converges to the trivial joining $\mu_X\otimes\mu_X$.

In the setting of the Picard group action on Linnik points on the sphere, for $d\in\D^\flat$ and $[\sk]\in {\rm Pic}(\scO_E)$, the set
\begin{equation}\label{eq:RdDelta}
\mathcal{R}_d^\Delta([\sk])=\{( [x] , [\sk].[x]) \mid [x]\in \mathcal{R}_d\},
\end{equation}
equipped with the diagonal action of the class group ${\rm Pic}(\scO_E)$, defines an off-diagonal self-joining of $\mathcal{R}_d$. A beautiful conjecture of Michel and Venkatesh \cite{MV} states that, although the acting group here is finite, the rescaled sequence $d^{-1/2}\mathcal{R}_d^\Delta([\sk])$ should converge to the trivial self-joining $(\mathcal{S}^2\times \mathcal{S}^2,\mu\otimes\mu)$, as the size of the group $|{\rm Pic}(\scO_E)|$, as well as the ``complexity'' of the ideal class $[\sk]$, go to infinity. To be more precise, let $q$ denote the smallest norm of an integral ideal representing $[\sk]$. The following is the mixing conjecture of the title.

\begin{conj}[Michel--Venkatesh]\label{MV-conj}
The set $d^{-1/2} \mathcal{R}_d^\Delta([\sk])$ equidistributes relative to the product measure $\mu\times\mu$ on $\mathcal{S}^2\times\mathcal{S}^2$ provided that $q\rightarrow \infty$ as $d \rightarrow \infty$ along $\D^\flat$.
\end{conj}

To better understand the statement of the mixing conjecture, it is helpful to restrict the action of the class group on $\mathcal{R}_d$ to certain monogenic subgroups, as in \cite{EMV}. Fix an odd prime $p$ and let $\D^\flat(p)$ be the subset of $d\in\D^\flat$ for which $p$ splits in $E=\Q(\sqrt{-d})$. After chosing a prime ideal $\mathfrak{p}$ lying over $p$, we may then consider the action of $\Z$ on $\mathcal{R}_d$ given by taking integral powers of the fixed ideal class $[\mathfrak{p}]\in {\rm Pic}(\scO_E)$. For every $d\in\D^\flat(p)$ and natural number $n\in\N$ we may consider the set of trajectories $\{ ([\mathfrak{p}]^i.[x] )_{ i=0,1,\ldots ,n} \mid  [x] \in \mathcal{R}_d\}$ in $\mathcal{R}_d$. Then $\mathcal{R}_d^\Delta([\mathfrak{p}]^n)$, as defined in \eqref{eq:RdDelta}, consists of all starting points and their corresponding endpoints. If these trajectories are sufficiently random, the former should retain very little information of the latter. In this set-up one  might say that the sequence $\mathcal{R}_d^\Delta([\mathfrak{p}]^n)$, $d \rightarrow \infty$, with the diagonal $[\mathfrak{p}]^\Z$ action is \textit{mixing}, if the rescaled set $d^{-1/2}\mathcal{R}_d^\Delta([\mathfrak{p}]^n)$ equidistributes in the product space $(\mathcal{S}^2\times\mathcal{S}^2,\mu\times\mu)$. A necessary condition for this to hold is that the length of the trajectories should go to infinity with $d\in\D^\flat(p)$. This length\footnote{Note that the endpoint $[\mathfrak{p}]^n.[x]$ depends on $n$ only through its class modulo the order of $[\mathfrak{p}]$ in ${\rm Pic}(\scO_E)$.} is given by the unique representative $n_{[\mathfrak{p}]}$ in $\{0,\ldots ,{\rm ord}([\mathfrak{p}])-1\}$ of the congruence class of $n$ modulo ${\rm ord}([\mathfrak{p}])$. In other words, we require $\log_p {\rm N}\mathfrak{q}\rightarrow\infty$, where $\mathfrak{q}=\mathfrak{p}^{n_{[\mathfrak{p}]}}$ is the integral ideal of smallest norm representing $[\mathfrak{p}^n]$.

Returning to the setting of Conjecture \ref{MV-conj}, a fundamental obstacle to mixing is the possible existence of low degree Hecke correspondences on which the sets $\mathcal{R}_d^\Delta([\sk])$ could accumulate. Interpreted adelically, the invariant measures defined on these Hecke correspondences define  off-diagonal self-joinings of $(\mathcal{S}^2,\mu)$. For example, when $[\sk]$ is the trivial class, $d^{-1/2}\mathcal{R}_d^\Delta([\sk])$ is contained in the diagonal copy of the sphere $\Delta \mathcal{S}^2\subset\mathcal{S}^2\times\mathcal{S}^2$. The condition on the minimal norm of an ideal representing $[\sk]$ in Conjecture \ref{MV-conj} is therefore a necessary condition for equidistribution in the product, as otherwise a subsequence of shifted diagonal sets will stay forever trapped in a fixed intermediate subvariety.

In this paper, we prove Conjecture \ref{MV-conj} under the assumption of the generalized Riemann hypothesis (GRH).

\begin{theorem}[Main Theorem, first version]\label{thm:sphere}
Assume GRH. Then Conjecture \ref{MV-conj} holds.
\end{theorem}

For the most general statement of our result, which includes a quantitative rate of convergence, see Theorem \ref{main-thm}.  Our methods pass through the spectral theory of automorphic forms, various period formulae and moments of $L$-functions, as well as sums of multiplicative functions along multivariate polynomial sequences, and information toward the Sato--Tate equidistribution law for Hecke eigenvalues. In particular, we exploit our two recent works \cite{BB,Kh}. We do not use any ergodic theoretic methods or input; for example, the relevant estimates taken from \cite{Kh} are purely analytic number theoretic. Nevertheless, the joinings theorem of Einsiedler and Lindenstrauss \cite{EL} inspired the structure of our argument.

\subsection{Previous results}\label{sec:prev-results}
Two important works have established partial progress toward Conjecture \ref{MV-conj}. We briefly review them here, in order to put Theorem \ref{thm:sphere} in greater context.

Firstly, the aforementioned paper \cite{EMV} of Ellenberg, Michel, and Venkatesh proved Conjecture \ref{MV-conj} if both of the following conditions hold:
\begin{enumerate}[leftmargin=4.2\parindent]
\smallskip
\item[\textit{(EMV1)}] $d\in\D^\flat(p)$, where $p>2$ is a fixed auxiliary prime;
\smallskip
\item[\textit{(EMV2)}] there is a fixed $\eta>0$ such that $q\leq d^{1/2-\eta}$ for all $q$ and $d$.
\smallskip
\end{enumerate}
The condition \textit{(EMV1)} comes from an application of Linnik's ergodic method. The condition \textit{(EMV2)} assures that the equidistribution of $\mathcal{R}_d^\Delta([\sk])$ in a low degree Hecke correspondence can be bootstrapped, via a spectral gap estimate, to equidistribution in $\mathcal{S}^2\times\mathcal{S}^2$.

More recently, in a novel combination of ergodic theory and analytic number theory \cite{Kh}, the last named author removed the second condition \textit{(EMV2)}, under the following assumptions:
\begin{enumerate}
\smallskip
\item[\textit{(Kh1)}] $d\in\D^\flat(p_1)\cap\D^\flat(p_2)$ where $p_1, p_2$ are auxiliary distinct odd primes;
\smallskip
\item[\textit{(Kh2)}] the absence of Siegel zeros for Dedekind zeta functions. 
\smallskip
\end{enumerate}
The double Linnik condition \textit{(Kh1)} allows for the use of the powerful joinings theorem for higher rank diagonalizable actions of Einsiedler and Lindenstrauss \cite{EL}. The second condition \textit{(Kh2)} allows for a successful treatment of the delicate range excluded from \textit{(EMV2)}, namely, when $q$ is close to $d^{1/2}$. This range is treated through an extension of Nair's classical estimates on sums of multiplicative functions along polynomial sequences. Such sums arise in \cite{Kh} as a geometric expansion of a  relative trace formula designed to capture the correlation of $\mathcal{R}_d^\Delta([\sk])$ with a fixed Hecke correspondence.

Let us compare these results with  Theorem \ref{thm:sphere}. The methods of \cite{EMV} could in principle be adapted to dispense with condition \textit{(EMV1)} under the assumption of GRH, just as Linnik did in the original equidistribution problem. On the other hand, their methods are incapable of removing condition \textit{(EMV2)}, even conditionally on GRH. The more apt comparison is rather to \cite{Kh}, which, like Theorem \ref{thm:sphere}, covers all ranges of $q$. Under the GRH assumption, which is stronger than \textit{(Kh2)}, our main theorem can be seen as removing the congruence conditions \textit{(Kh1)} of \cite{Kh}, while also providing an effective rate. It is worth noting here that the theorem in \cite{Kh} is ineffective even if one assumes GRH, because there is no known effective version of the joinings rigidity theorem of Einsiedler and Lindenstrauss \cite{EL}.   

\subsection{More general formulation}\label{sec:more-gen-form}
For expository purposes of the introduction, we have intentionally confined our discussion to the case of integral points on the sphere, but Duke's theorem and subsequent developments are best viewed as concrete instances of a more general problem. Indeed, from the point of view of algebraic groups, the theorem of Duke (as well as that of Duke and Schulze-Pillot \cite{DSP}) can be rephrased as the equidistribution of an adelic torus orbit of large fundamental discriminant inside the adelic quotient of an inner form $\bG$ of ${\rm PGL}_2$ associated with a definite quaternion algebra over $\Q$. From this perspective, there is a close connection to other celebrated equidistribution results (``Problems of Linnik type'', as coined in \cite{MV}), such as Duke's theorem \cite{Du} on Heegner points and closed geodesics on the modular curve, where one takes $\bG=\mathbf{PGL}_2$.

This adelic framework is therefore the more natural setting for Conjecture \ref{MV-conj}, as well as the previous work on it surveyed in \S \ref{sec:prev-results}. Our main theorem is in fact stated in this general language in Theorem \ref{main-thm}, after the necessary notation has been introduced. We point out, however, that there are two hypotheses that appear in Theorem \ref{main-thm} that are absent in Theorem \ref{thm:sphere}, which is specialized to the sphere, whose underlying quaternion algebra is the Hamiltonian quaternions $B^{(2,\infty)}$. Namely, when $\bG=\mathbf{PGL}_2$ we restrict to test functions in the discrete spectrum (as was done in \cite{BB}), and more generally when $\bG={\bf PB}^\times$ with ${\bf B}$ indefinite we assume the Ramanujan conjecture for $\mathbf{PGL}_2$.

\subsection{Mixing as ergodic theoretic Andr\'e--Oort}\label{sec:A-O}

We close the introduction with a brief comparison of the mixing conjecture with the Andr\'e--Oort conjecture. We shall describe this link in greater detail in \S \ref{sec:Hecke-explanation}. 

In the closely related setting of modular and Shimura curves the mixing conjecture is a strong form of the Andr\'e--Oort conjecture for the product of two modular (or Shimura) curves, where Zariski density is replaced by equidistribution. The first major progress towards the  Andr\'e--Oort conjecture in this setting (now a theorem, by Andr\'e \cite{A}) was made by Edixhoven \cite{Edix}, conditionally under the assumption of  GRH.  Our general result Theorem \ref{main-thm} can therefore be viewed as roughly analogous to that of Edixhoven.

\subsection*{Acknowledgement} We would like to express particular thanks to the referees whose careful and detailed report led to substantial improvements in the  presentation of this paper. 

\section{Overview of proof}\label{sec:pf-overview}

Before describing our approach to Theorem \ref{thm:sphere}, we state a more general result in Theorem \ref{thm:2nd-version} below, of which the mixing conjecture for the sphere is but a special case. We make an effort to minimize the notational overhead in this section, which we develop in ample detail in \S \ref{sec:prelim} and \S \ref{sec:adelic-mixing-main-theorem}. Moreover, in the course of describing our method, in \S \ref{sec:Hecke-explanation}, we shall again reformulate our main result in the special setting of Shimura curves, in the spirit of the measure theoretic version of the Andr\'e--Oort conjecture alluded to in \S \ref{sec:A-O}; this is Theorem \ref{thm:A-O-style}. In any case, the more general statement of our main theorem is given later, in Theorem \ref{main-thm}, after the requisite notation has been set up.

\subsection{More general version of main theorem}\label{sec:more-gen-version}

To motivate the more general version, note that the finite quotient of the sphere $\mathcal{S}^2=\Gamma\backslash S^2$ from \eqref{eq:def-Rd} can be realized as an adelic double quotient, in the following way. Let $B$ be the quaternion algebra over $\Q$ of Hamilton quaternions and denote by $R\subset B$ the Hurwitz quaternions; then $R$ is the unique maximal order in $B$,  up to $B^\times$-conjugacy. Setting $\bG = \bPB^{\times}$ to be the affine algebraic group defined over $\Q$ representing the projective unit group of $B$, i.e., $\bG(\Q)=\Q^\times\backslash B^\times$, we have $\mathcal{S}^2=\bG(\Q)\backslash\bG(\A)/K$. Here $K=K_\infty K_f$, $K_f$ is the closure in $\bG(\A_f)$ of $R^\times/{\pm 1}$, and $K_\infty\simeq\SO(2)$ is a maximal compact torus of $\bG(\R)\simeq\SO(3)$.

This description naturally generalizes to $\bG$ representing the projective group of units of an arbitrary quaternion algebra $B$ over $\Q$, along with a choice of a maximal order $R\subset B$. For $B$ indefinite, i.e., $\mathbf{B}(\Bbb{R}) = \Mat(\R)$, this recovers a wide class of Shimura curves, and for $B$ definite we obtain a finite union of (finite quotients of) ellipsoids. We denote any such space by $Y_K$ --- in this section we shall abbreviate this to $Y$ --- and call it a \textit{quaternionic variety}; see \S \ref{sec:Shimura} for more details. We will also denote by ${\rm Ram}_B$  the finite even set of rational places where $B$ ramifies and set $d_B\in\mathbb{N}$ to be the product of all primes in ${\rm Ram}_B$. In the split case we have $d_{\Mat}=1$.

Now, $Y$ admits a notion of ``special point", namely, one whose stabilizer (the Mumford--Tate group) is a torus defined over $\Q$. Such special points are partitioned into packets, each packet being a principal homogeneous space for the Picard group of an imaginary quadratic order $\scO$. One can then define the discriminant of a packet to be $D=|{\rm disc}(\scO)|$; it is \textit{admissible}, in the sense of Lemma \ref{lemma:adm-disc}. Since the level structure $K$ defining $Y$ arises from a maximal order, for every imaginary quadratic field $E$ there are $O_B(1)$ such packets depending on the splitting behavior of primes dividing ${\rm gcd}(D,d_B)$; see Proposition \ref{prop:number-of-packets}. 

We now take $D$ to be an admissible fundamental discriminant and fix such a packet --- a torsor under the class group ${\rm Pic}(\scO_E)$ --- and denote it by $H_D$. This notation is imprecise because, as we have just noted, there can in general be several packets associated to the same discriminant $D$, but we allow for it in the introduction for simplicity. The notation $H_\scD$, which we introduce in \eqref{def:Heegner-packet}, will distinguish these packets.

One can then define a \textit{joint Heegner packet} in $Y\times Y$ as a single orbit of ${\rm Pic}(\scO_E)$ under the diagonal action on a pair of special points from the same packet $(x,y)\in H_D\times H_D$. In this case, we let $[\sk]\in {\rm Pic}(\scO_E)$ be uniquely defined by $y=[\sk].x$ and write $H_D^\Delta([\sk])$ for this joint packet.

For example, when $Y=\mathcal{S}^2$, the set $\mathcal{R}_d$ from \S\ref{sec:mix-conj} is a union of at most two packets $H_D$, where $E=\Q(\sqrt{-d})$. In particular, when $d\equiv 3\pmod 8$ we have $\mathcal{R}_d=H_{\Q(\sqrt{-d})}$ and the set $\mathcal{R}_d^\Delta ([\sk])$ from \eqref{eq:RdDelta} is $H_{\Q(\sqrt{-d})}^\Delta([\sk])$.

Finally, just as in the discussion preceding Conjecture \ref{MV-conj}, we let $q$ denote the smallest norm of an integral ideal representing $[\sk]$.

\begin{theorem}[Main theorem, second version]\label{thm:2nd-version}
Let $B$ be a non-split quaternion algebra over $\Q$. Assume GRH and (when $B$ is indefinite) the Ramanujan conjecture for $\mathbf{PGL}_2$. For every imaginary quadratic field $E$ choose $[\sk] \in {\rm Pic}(\scO_E)$. Then $H_D^\Delta([\sk])$ equidistributes in $Y\times Y$ as $D=|{\rm disc}(E)|\rightarrow\infty$, with a quantitative rate of convergence,  provided that $q$ also tends to infinity with $D$.
\end{theorem}

\begin{remark}\label{rem:always-imag-quad} In the indefinite case, one can also formulate a mixing conjecture for joint packets of closed geodesics in $Y$ (or its unit tangent bundle) associated with the class group of a real quadratic field. Such a variant was described in the context of Duke's theorem in \S\ref{sec:more-gen-form}. We have restricted our attention in this paper to the case of imaginary quadratic fields for simplicity.
\end{remark}

\subsection{Special subvarieties and Heegner points higher level}\label{sec:Hecke-explanation}

As was mentioned after the statement of Conjecture \ref{MV-conj}, the fundamental obstacle to Theorem \ref{thm:2nd-version} is the presence of intermediate subvarieties in which $H_D^\Delta([\sk])$ could find itself trapped. In this section we shall identify these intermediate subvarieties, and interpret the condition that $q\rightarrow\infty$ as the geometric condition, clearly necessary, that $H_D^\Delta([\sk])$ escapes from them.

Now $H_D^\Delta([\sk])$ of course leaves the diagonal copy $Y^\Delta\subset Y\times Y$ the moment $[\sk]$ is non-trivial, and for a quantitative result we require $q\rightarrow \infty$. On the other hand, for every primitive integral ideal $\nf$ of $\scO_E$ representing $[\sk]$, $H_D^\Delta([\sk])$ lies in the graph of the Hecke correspondence $Y_0^\Delta(N)\subset Y\times Y$ of degree $\varphi(N)$, where $N$ is the norm of $\nf$. When $B$ is indefinite, so that $Y$ has the structure of a Shimura curve, such subvarieties are examples of \textit{weakly special} subvarieties (the others being pairs $(x,y)$ of special points, along with $Y\times \{y\}$, $\{x\}\times Y$, with $x$ and $y$ being special). For general $B$, it is known that $H_D^\Delta([\sk])$ lies in the set $H^\Delta(N,D)$ of Heegner points of level $N$ and discriminant $D$ on $Y^\Delta_0(N)$.

More can be said, in fact. Indeed, an integral ideal $\nf\subset\scO_E$ representing $[\sk]$ contains more information than its norm $N$ determining the Hecke correspondence $Y^\Delta_0(N)$. Let us assume momentarily that $N$ verifies the \textit{Heegner hypothesis} relative to $E$, so that all prime factors $p\mid N$ split in $E$: $p\scO_E=\pf^+\pf^-$, with $\pf^+\neq\pf^-$. Then $\nf$ is determined by its norm $N$ along with the choice of a prime ideal $\pf^+$ lying over every $p\mid N$. By the theory of orientations, this extra information corresponds to a choice of class group orbit on $H^\Delta(N,D)$:
\begin{equation}\label{eq:HE-decomp}
H^\Delta(N,D)=\bigsqcup_{\substack{\nf\subset\scO_E\\\Nr\nf=N}}H^\Delta(\nf,D).
\end{equation}
We can then assert the following more precise statement: for every primitive integral ideal $\nf$ representing the class $[\sk]$, whose norm $N$ verifies the Heegner hypothesis relative to $E$, the joint packet $H_D^\Delta([\sk])$ is equal to the Heegner packet $H^\Delta(\nf, D)$ of level $N$ and discriminant $D$ in $Y_0^\Delta(N)$. The parametrization \eqref{eq:HE-decomp} of Heegner points of fixed level and fundamental discriminant on quaternionic varieties is explained from the adelic point of view and without the simplifying assumption of the Heegner hypothesis in \S \ref{sec:oriented-packets}. Their relation with Hecke correspondences is described in \S\ref{sec521}.

\begin{example*}
We review the above facts in the classical setting of Heegner points on modular curves, where $B$ is the split algebra $\Mat(\Q)$.

In this case, the space $Y$ is the open modular curve $\SL_2(\Z)\backslash\mathbb{H}$, parametrizing isomorphism classes of elliptic curves $\mathscr{E}$ over $\C$. The \textit{level one Heegner points of discriminant $-D$} on $Y$ are those $\mathscr{E}$ having complex multiplication by $\scO_E$. More generally, for a positive integer $N$ we let $Y_0(N)$ be the space of triples $(\mathscr{E},\mathscr{E}',\phi)$, where $\mathscr{E},\mathscr{E}'\in Y$ and $\phi: \mathscr{E}\rightarrow \mathscr{E}'$ is a cyclic degree $N$ isogeny. The \textit{Heegner points of level $N$ and discriminant $-D$} on $Y_0(N)$ are those $(\mathscr{E},\mathscr{E}',\phi)\in Y_0(N)$ for which both $\mathscr{E}$ and $\mathscr{E}'$ have complex multiplication by $\scO_E$. Under the Heegner hypothesis, this set has cardinality $2^{\omega(N)}|\operatorname{Pic}(\scO_E)|$, where $\omega(N)$ denotes the number of (distinct) prime factors of $N$, and decomposes naturally into $2^{\omega(N)}$ disjoint sets $H(\nf,D)$, parametrized by the ideals $\nf\subset\scO_E$ of norm $N$, on which ${\rm Pic}(\scO_E)$ acts simply transitively. Specifically, $H(\nf,D)$ consists of those $(\mathscr{E},\mathscr{E}',\phi)\in H(N,D)$ such that $\ker\phi$ is annihilated by $\nf$. For a transcription of these facts into the classical language of classes of primitive binary quadratic forms, see \cite{GKZ}.

Note that in this case, given special points $\mathscr{E}$ and $\mathscr{E}'$ in $H_D$, the minimal norm $q$ of an integral ideal representing the class $[\sk]\in {\rm Pic}(\scO_E)$ taking $\mathscr{E}'$ to $\mathscr{E}$, as in the statement of Theorem \ref{thm:2nd-version}, is the smallest degree of a cyclic isogeny between them.

We now relate these notions to the preceding more general set-up. The graph of the Hecke correspondence at $N$ is the image $Y^\Delta_0(N)$ of the map $Y_0(N)\rightarrow Y\times Y$ which forgets $\phi$. The curves $Y_0(N)$ and $Y_0^\Delta(N)$ are birationally equivalent, but not necessarily isomorphic, as $\phi$ is not always uniquely determined by the pair $(\mathscr{E},\mathscr{E}')$; see Lemma \ref{lem:birational}. On the other hand, $H^\Delta(N,D)$ and $H^\Delta(\nf,D)$ are the isomorphic images in $Y\times Y$ of $H(N,D)$ and $H(\nf,D)$, respectively, under this same map (see Lemma \ref{lem:shifted-diag-torus} and \S \ref{sec521}), yielding the decomposition \eqref{eq:HE-decomp}.
\end{example*}

Returning to the general case, we may now reinterpret the condition on $q$ in Conjecture \ref{MV-conj} in terms of the various inclusions $H_D^\Delta([\sk])\subset Y_0^\Delta(N)$. Namely, that condition says that the minimal degree of a Hecke correspondence containing $H_D^\Delta([\sk])$ should tend to infinity with $D$. Equivalently, the joint Heegner packets $H_D^\Delta([\sk])$ ``escape'' the intermediate subvariety $Y^\Delta_0(N)$ as $D$ gets large, in the sense that each fixed Hecke correspondence $Y^\Delta_0(N)$ contains only finitely many $H_D^\Delta([\sk])$. This allows us to state the following version of our main theorem, where we impose $B$ indefinite in order to access the language of weakly special subvarieties.

\begin{theorem}[Main Theorem: third version]\label{thm:A-O-style}
Let $Y$ be a compact Shimura curve associated to a maximal order in an indefinite quaternion algebra. Let $(x_i,y_i)\in Y\times Y$ be a sequence of pairs of special points in $Y$, such that $x_i$ and $y_i$ are both defined over the Hilbert class field of the same imaginary quadratic field $E_i$. Assume the sequence $(x_i, y_i)$ is \textit{strict}, i.e., for every weakly special subvariety $Z \subset Y\times Y$ defined over $\Q$, we have $(x_i, y_i) \not\in Z$ for $i\gg_Z 1$.

Then, under the assumptions of GRH and the Ramanujan conjecture for $\mathbf{PGL}_2$, the Galois orbits of $(x_i,y_i)$ equidistribute in the limit in $Y\times Y$. Moreover, the rate of equidistribution is quantitative and effective in ${\rm disc}(E_i)$.
\end{theorem}

\subsection{The Weyl criterion}\label{sec:Weyl-crit}
In the remainder of this section, we describe our approach to Theorem \ref{thm:2nd-version}.
Weyl's equidistribution criterion roughly states that Theorem \ref{thm:2nd-version} follows from an understanding of the asymptotic size of the uniform sample
\begin{equation}\label{eq:joint-Heegner-period}
P_D^\Delta(\varphi;[\sk])=\frac{1}{|H_D|}\sum_{x\in H_D}\varphi(x,[\sk].x),
\end{equation}
for functions $\varphi$ lying in the $L^2$ spectrum of $Y\times Y$. For this we test against an $L^2$-orthonormal basis adapted to the decomposition of $L^2(Y\times Y)$ into tensor products of irreducible automorphic representations $\sigma\subset L^2([\bG(\A)])$. Indeed, we have a Hilbert space direct sum decomposition 
\[
L^2(Y\times Y)=\bigoplus_{\sigma_1,\sigma_2\subset L^2([\bG(\A)])}\sigma_1^K\otimes \sigma_2^K.
\]
The fact that $R$ is maximal implies that each $\sigma_i^K$ is a line, say generated by $f_i$. Then we take as a basis element $\varphi=f_1\otimes f_2$. As the case $f_1 = f_2 = \text{const.}$ is trivial, we may assume that $f_1$ and $f_2$ are not both constant. In this case, we wish to prove $P_D^\Delta(\varphi;[\sk])=o(1)$, as $q,D\rightarrow\infty$, under the hypotheses of Theorem \ref{thm:2nd-version}. Here and in the following all implied constants depend at most on the quaternion algebra $B$, and \emph{polynomially} on the archimedean parameters (cf.\ \S\ref{sec:reduction-to-sparse}) of $f_1$ and $f_2$. In turn, this should allow for a dependence on the test function $\varphi$ by means of an appropriate Sobolev norm. 

If one $f_i$ is cuspidal and the other is constant, then Duke's theorem \cite{Du} applies, without recourse to GRH, providing a strong upper bound on $P_D^\Delta(\varphi;s)$, which decays as a small power of $D$. (Strictly speaking, the setting of \cite{Du} is constrained to $B=\Mat(\Q)$ or the Hamilton quaternions. This restriction has been removed elsewhere in the literature. See \cite[\S 2.2]{MV} for a summary.) We may therefore assume that both $f_1, f_2$ are cuspidal. 

The remaining argument divides into two overlapping ranges:
\begin{enumerate}
\item \textit{the large $q$ range}, where $q\geq D^{1/6+\varepsilon}$, which we treat in Sections \ref{sec:large-q} and \ref{sec:diag-est}, conditionally on GRH and (in the indefinite case) the Ramanujan conjecture, using \cite{BB} and estimates on sparse sums of multiplicative functions;

\item \textit{the small $q$ range}, where $q\leq D^{1/2-\varepsilon}$ goes to infinity, which we treat in \S \ref{sec:small-q}, conditionally on the Lindel\"of hypothesis and (in the indefinite case) the Ramanujan conjecture, using a spectral expansion across an intermediate Hecke correspondence of degree $q^{1+o(1)}$.
\end{enumerate}
In fact, for $q \leq D^{\eta}$ and $\eta > 0$ sufficiently small, the
method of the small $q$ range is completely unconditional and depends
only on known subconvexity results. The above subdivision supposes the Lindel\"of hypothesis and the Ramanujan conjecture for small $q$ simply to ensure that the two ranges overlap.

\subsection{Large $q$ range}\label{sec:large-q-range}

In this case we apply the Parseval identity to obtain
\begin{equation}\label{eq:shifted-prod}
P_D^\Delta(\varphi;[\sk]) =\sum_{\chi\in {\rm Pic}(\scO_E)^\vee} \chi(\sk) W_D(f_1, \chi) \overline{W_D(f_2, \chi)},
\end{equation}
where 
\[
W_D(f_i, \chi)  =\frac{1}{|H_D|}\sum_{x\in H_D} f_i(x) \chi(x).
\]
Recall the notation $\bG={\bf PB}^\times$ from \S\ref{sec:more-gen-version}. Let $\sigma_j$ be the cuspidal automorphic representation of $\bG(\A)$ generated by $f_j$, as in \S\ref{sec:Weyl-crit}. If $\sigma_1 \neq \sigma_2$, we can forgo any cancellation in the sum over $\chi$ in \eqref{eq:shifted-prod}, apply absolute values and use the main result of \cite{BB} to bound $P_D^\Delta(\varphi;[\sk])$ by  $O_\varepsilon\left((\log D)^{-1/4+\varepsilon}\right)$, uniformly in $[\sk]$.   This requires the fundamental hypothesis of GRH, but applies for the entire range of $q$ relative to $D$. Henceforth we may therefore assume that $\sigma_1 = \sigma_2 =\sigma$, generated by $f$.

Let $\pi={\rm JL}(\sigma)$ be the Jacquet--Langlands correspondent of $\sigma$ on $\mathbf{PGL}_2$. By Waldspurger's theorem the expression \eqref{eq:shifted-prod} equals
\begin{equation}\label{3}
\frac{1}{|{\rm Pic}(\scO_E)|} \sum_{\chi \in {\rm Pic}(\scO_E)^\vee}\bar{\chi}([\sk])  L(1/2, \pi\times \chi),
\end{equation}
up to multiplication by a normalizing constant involving an $L$-values at 1 in the denominator,  which can be estimated easily on GRH. The use of GRH here stands in contrast to its use for the off-diagonal contributions in \cite{BB}, where it was used in a more fundamental way to control the average size of the \textit{central} value. By an approximate functional equation we are essentially left with bounding
\[
\sum_{\substack{(x, y) \not= (0, 0)\\ Q(x,y)\ll D}}   \frac{\lambda_\pi(Q(x, y))}{ Q(x, y)^{1/2}},
\]
where $\lambda_\pi(n)$ are the Hecke eigenvalues of $f$ at integers $n\in\N$ and $Q$ is the reduced positive definite binary quadratic form corresponding to the ideal class $[\sk]$. This reduction step is executed in Lemma \ref{lemma:conversion-to-sparse-sum}. We must therefore bound the average of a multiplicative function over the sparse sequence of values of a binary quadratic form of large discriminant. The key result is Proposition \ref{prop:Kh} below, complemented by \eqref{bound1}, where we estimate general sums of Hecke eigenvalues over values of quadratic forms.

It was probably first observed by Barban \cite{Ba} that the large sieve is extremely powerful in such situations, and this has been refined by many authors, e.g.,\  \cite{Wo,  NT, Kh}. In our case,   the version of \cite{Kh} is most suited for our purpose. Such an approach, however, requires positivity, so we simply apply absolute values to the coefficients and hope to win by the fact that $\lambda_\pi(n)$ satisfies a Sato--Tate law that makes it a little less than 1 on average. As in the analysis of the off-diagonal term  $\sigma_1 \neq \sigma_2$, a power saving error term as in Duke's theorem becomes impossible at this point; at best we can obtain a saving on a logarithmic scale. We remark, however, that we do not 
need the entire Sato--Tate law: some functorial information (e.g.,\ automorphy of the symmetric fourth power) suffices to get a logarithmic saving. This path was blazed first by Holowinsky \cite{Ho}, and similar ideas were also used by Lester--Radziwi{\l\l} \cite{LR}. 

In order to apply the sieve, the sequence $\{Q(x, y) : (x, y) \in \Bbb{Z}^2\}$ needs to be sufficiently well-behaved in arithmetic progressions, and this in turn requires $Q$ to be sufficiently ``well-rounded''. In other words, its minimum $q$ must not be too small. Quantitatively, by a classical result of van der Corput, the required lower bound is $q \geq D^{1/6}$; this could be reduced a little using more sophisticated exponential sum techniques. Conjecturally on the best possible error term in the Gau{\ss} circle problem, $q \geq D^{\varepsilon}$ suffices.

\subsection{Small $q$ range}\label{sec:bottom-q-sketch} 
Recall the joint Heegner period $P_D^\Delta(\varphi; [\sk])$ from \eqref{eq:joint-Heegner-period}. Then
\[
P_D^\Delta(f_1\otimes f_2; [\sk])=\frac{1}{|H_D|}\sum_{x\in H_D} f_1(x) f_2([\sk].x).
\]
Now, it is well-known that the action of ${\rm Pic}(\scO_E)$ on $H_D$ does not extend to the whole quaternionic variety $Y$, so one cannot, say, spectrally decompose $f_1(\cdot) f_2([\sk].\,\cdot)|_{Y^\Delta}$ across $L^2(Y^\Delta)$. However, for any primitive integral ideal $\mathfrak{n}$ representing $[\sk]$, the discussion in \S\ref{sec:Hecke-explanation} identifies $H_D^\Delta([\sk])$ with a Heegner packet $H^\Delta(\mathfrak{n},D)$ of level $N$ and discriminant $D$ appearing in the decomposition \eqref{eq:HE-decomp}. Thus we have the identity
\begin{equation}\label{2}
P_D^\Delta(f_1\otimes f_2; [\sk])=\frac{1}{|H_D|}\sum_{x\in H^\Delta(\mathfrak{n},D)} f_1(x) f_2(Nx).
\end{equation}
This last identity has the effect of replacing the action by $[\sk]$ with multiplication by $N$, at the price of adding an appropriate level structure.

To treat the right-hand side of \eqref{2}, we may spectrally expand $f_1(x) f_2(Nx)$ \textit{as a function of $x\in Y_0^\Delta(N)$.} We obtain
\begin{equation}\label{2a}
\frac{1}{|H_D|} \sum_{j\geq 0}  \langle f_1 f_2(N\cdot), u_j\rangle \sum_{x\in H^\Delta(\mathfrak{n},D)} u_j(x),
\end{equation}
where the sum runs over an orthonormal basis of eigenforms $\{u_j\}$ of level $N$ (including the constant function). The $j$-sum is rapidly decaying in the spectral parameter, so it is effectively of length $N$. A precise version of this spectral expansion (including Eisenstein series if $B = \text{Mat}_{2 \times 2}(\Bbb{Q})$) is given in Proposition \ref{lemma-spectral-decomp} below.

If $f_1 = f_2 = f$, then it is easy to see that the constant function contributes essentially $\lambda_f(N) N^{-1/2}$, where $\lambda_f(N)$ is the Hecke eigenvalue of $f$ at $N$. By non-trivial bounds towards the Ramanujan conjecture, this tends to zero as soon as $N$ tends to infinity. 

Both the triple product and the Heegner period are related to (canonical) square-roots of central $L$-values. At this point we have no chance to exploit any cancellation in the spectral sum. Applying absolute values will therefore inevitably lose roughly a factor $N^{1/2}$, but at least under the assumptions of the Lindel\"of hypothesis the $L$-functions can be estimated optimally. We therefore expect a bound of the shape $D^{-1/4} N^{1/2} (DN)^{\varepsilon}$, which is admissible provided that $N\leq D^{1/2 - \varepsilon}$. Note that in \eqref{2a} we have the flexibility to take any ideal $\mathfrak{n}$ in the same class as $\mathfrak{s}$. Taking the representative with minimal norm $q$ as in the statement of Theorem \ref{thm:2nd-version}, we obtain the weakest possible condition.

Needless to say, this is only an optimistic back-of-an-envelope computation, which omits several technical details, such as local factors and oldforms in period formulae. 

\begin{remark}
The strategy just described should be compared to that of \cite{EMV}, which uses equidistribution in stages over the tower $H^\Delta_E ([\sk])\subset Y^\Delta_0(N)\subset Y \times Y$. Our method is roughly based on the same principle, but uses a spectral expansion across $Y^\Delta_0(N)$, an idea which can already be found in Michel--Venkatesh \cite[Remark 2.1]{MV1}. See also \S\ref{gelfand}. 
\end{remark}

\begin{remark}
That the spectral argument fails for very large $q$ should not be surprising, as the right-hand side of \eqref{2} considers the distribution of $|H_D|\approx D^{1/2}$ points on a surface of volume $\approx N$, so the natural limit should indeed be $N\leq D^{1/2 + o(1)}$. In terms of the reduced positive definite binary quadratic form $Q$ determined by $[\sk]$, the methods for both ranges are nicely complementary: the spectral approach is successful if the ellipse described by $Q$ is sufficiently slim, while the sieve approach is powerful if $Q$ is sufficiently round. In addition, at least under the assumption of reasonable conjectures (GRH and Gau{\ss} circle problem), there is a huge overlap, but neither of the two methods covers the entire range.
\end{remark}

\subsection{Gelfand formations}\label{gelfand}

This division into two quantitative ranges described in Sections \ref{sec:large-q-range} and \ref{sec:bottom-q-sketch} has a representation theoretical counterpart.

Recall that a pair of algebraic groups $\mathcal{H}\subset\mathcal{G}$ over $\Q$ is called a \textit{strong Gelfand pair} if for every place $v$ of $\Q$ one has the local multiplicity one property $\dim {\rm Hom}_{\mathcal{H}(\Q_v)}(\pi_v,\sigma_v)\leq 1$ for all irreducible admissible representations $\pi_v$ and $\sigma_v$ of $\mathcal{G}(\Q_v)$ and $\mathcal{H}(\Q_v)$, respectively. When this fails to hold for a given $\mathcal{H}\subset\mathcal{G}$ one can hope to construct the next best thing: a \textit{strong Gelfand tower}, where an intermediate subgroup $\mathcal{L}$ is found to make each inclusion $\mathcal{H}\subset\mathcal{L}\subset\mathcal{G}$ a strong Gelfand pair. When this can be done in two different ways, the resulting schematic construction is called a \textit{strong Gelfand formation}. In the theory of automorphic forms, the process of spectrally expanding the $\mathcal{H}$-period of an automorphic form on $\mathcal{G}$ across the intermediate subgroups of a Gelfand formation often leads to powerful identities between seemingly unrelated families of $L$-functions. This idea was beautifully expounded in \cite{Rez}.

Let $\bT$ be a subgroup of $\bG$ isomorphic to the projective torus ${\rm Res}_{E/\Q}\mathbb{G}_m/\mathbb{G}_m$, determined by the imaginary quadratic extension $E$. When the above principle is applied to the case of interest of this paper, we obtain the strong Gelfand formation
\[
\xymatrix  @R=1.3pc{
& \bG\times \bG\ar@{-}[ld]  \ar@{-}[rd] \\
\bT\!\times\! \bT \!\!\ar@{-}[rd] &\quad &\!\!\Delta\bG \ar@{-}[ld] \\
 & \Delta \bT
}
\]
where $\Delta\bG$ is the diagonal embedding of $\bG$ into $\bG\times\bG$, and similarly for $\Delta\bT$. Our methods can be viewed as applying the above strong Gelfand formation to the automorphic form $f_1(g_1)f_2(g_2s)$ on $(\bG\times\bG)(\A)$. Going back to the quantitative division of ranges, the case of small $q$ is handled by the strong Gelfand tower $\Delta\bT\subset\Delta\bG\subset\bG\times\bG$, while the delicate large $q$ range, as well as the off-diagonal analysis imported from \cite{BB}, is handled by the strong Gelfand tower $\Delta\bT\subset\bT\times\bT\subset\bG\times\bG$. 

\section{Quaternionic varieties and Heegner points}\label{sec:prelim}

We now begin preparing the ground to state a more general version of the mixing conjecture than those presented in Theorems \ref{thm:sphere}, \ref{thm:2nd-version}, and \ref{thm:A-O-style}. Over the course of the next two sections, our aim is to state Conjecture \ref{conj:adelic-mixing}, which is the version that we shall actually prove, conditionally and at full level, in Theorem \ref{main-thm}. Before doing so, we must spend a substantial amount of time setting up the relevant notation. 

Throughout, $\A$ will denote the adele ring of $\Q$ and $\A_f$ the finite adeles. We refer to \cite[Section 2]{BB} for a detailed dictionary between the adelic and classical language of the objects considered here. 

\subsection{Adelic quotients}
For any linear algebraic group $\bH$ defined over $\mathbb{Q}$ we will write $[\,\cdot\,]\colon\bH(\A)\to \bH(\Q)\backslash\bH(\A)$ for the quotient map. If $U_f<\bH(\A_f)$ is a compact open subgroup and $U=U_\infty\cdot U_f <\bH(\A)$ for a compact subgroup $U_\infty<\bH(\R)$, then we will denote by $[\,\cdot\,]_U\colon \bH(\Q)\backslash\bH(\A)\to \bH(\Q)\backslash \bH(\A) \slash U$ the quotient map between the two. We will also use the notation $[\,\cdot\,]_U$ for the surjective quotient map from $\bH(\A)$ to $\bH(\Q)\backslash \bH(\A) \slash U$. This should cause no confusion because the maps are compatible. In particular, we have
\[
[\bH(\A)]=\bH(\Q)\backslash\bH(\A),\qquad
[\bH(\A)]_U=\bH(\Q)\backslash\bH(\A)\slash U.
\]

\subsection{Quaternionic varieties}\label{sec:Shimura}
Let $B$ be a quaternion algebra over $\Q$, with ramification set ${\rm Ram}_B$ and discriminant $d_B=\displaystyle\prod_{p\in{\rm Ram}_B} p$. For every rational place $v$ define $B_v=B\otimes_\Q \Q_v$. Fix a maximal order $R\subset B$. Denote by $R_p\subset B_p$ the completion of $R$ in $B_p$. For every prime $p\in {\rm Ram}_B$ the order $R_p$ is the unique maximal order in $B_p$. For  primes $p\notin {\rm Ram}_B$ we fix an algebra isomorphism $\phi_p:B_p\xrightarrow{\sim} \Mat(\Q_p)$, such that $\phi_p(R_p)=\Mat(\Z_p)$. 
Define $B_f=\prod_p' B_p$, where the restricted direct product over primes is taken with respect to the system $(R_p)_p$. Let $\widehat{R}=\prod_p R_p\subset B_f$; then $R=B\cap \widehat{R}$. 

Set $\bB^\times$ to be the algebraic group representing the group of units of $B$, and let $\bG$ represent of projective group of units of $B$, i.e.,
\begin{alignat*}{2}
\bB^\times(A)=\operatorname{GL}_1(B\otimes_\Q A)=\left(B \otimes_\Q A\right)^\times, \quad && \bG(A)=\operatorname{PGL}_1(B\otimes_\Q A)=A^\times \backslash\left(B \otimes_\Q A\right)^\times,
\end{alignat*}
for any $\Q$-algebra $A$. Then $\bB^\times$ and $\bG$ are affine algebraic groups defined over $\Q$, and $\bG\simeq \mathbb{G}_m\backslash \bB^\times$. Set $\widetilde{K}_f=\widehat{R}^\times\subset\bB^\times(\A_f)$. Let $K_f$ be the image of $\widetilde{K}_f$ in $\bG(\A_f)$ under the  quotient homomorphism $\bB^\times(\A_f)\rightarrow \bG(\A_f)$. Let $K_\infty$ be a maximal compact torus of $\bG(\R)$. Let $U_f$ be a finite index subgroup of $K_f$ and write $U=U_f K_\infty$. Recall that $[\bG(\A)]=\bG(\Q)\backslash\bG(\A)$ and set
\begin{equation}\label{def:Y}
Y_U=[\bG(\A)]_U= \bG(\Q)\backslash \bG(\A)/U,
\end{equation}
a possibly disconnected space\footnote{if $B$ is definite, or $U_f$ is not maximal.}. If $\mathbf{G}$ is indefinite then $Y_U$ is a complex Shimura curve with level structure given by $U_f$. 

Write $K_f=\prod_p K_p$, where $K_p$ is the image of $R_p^\times$ in $\bG(\Q_p)$ under the quotient map. Then $K_p$ is a maximal compact subgroup of $\bG(\Q_p)$ for every $p\nmid d_B$ and is an index $2$ normal subgroup of the compact group $\bG(\Q_p)$ for $p \mid d_B$. Indeed, in the latter case, $B_p$ is the unique quaternion division algebra over $\Q_p$, its ramification index is $2$, and $K_p$ is the subgroup of elements who have a representative with even valuation. We say that a compact subgroup $U_f=\prod_p U_p<\bG(\A_f)$ is \emph{almost maximal} if $U_p$ is a maximal compact subgroup for all $p\not\in{\rm Ram}_B$ and $U_p$ is the normal subgroup of elements with even valuation if $p\in{\rm Ram}_B$. The almost maximal compact subgroups are exactly the subgroups of $\bG(\A_f)$ that arise as projectivizations of units in the completion of a maximal order.
When $U_f=K_f$ is almost maximal, we write $K=K_fK_\infty$ and consider $Y_K=[\bG(\A)]_K$. 

\textbf{Example:} Consider the quaternion algebra $B=B^{(2,\infty)}$ that is ramified exactly at $2$ and $\infty$; these are the Hamilton quaternions over $\Bbb{Q}$ with $d_B = 2$. A maximal order $R$ is the set of Hurwitz quaternions. The group $\bG(\Bbb{R})$ is well-known to be isomorphic to ${\rm SU}(2)/\{\pm 1\} \cong {\rm SO}(3)$ via $(x + iy + jz + k w) \mapsto (\begin{smallmatrix} x + \sqrt{-1} y & z + \sqrt{-1} w\\ -z + \sqrt{-1} w & x - \sqrt{-1} y \end{smallmatrix})$. If we choose $U_f = K_f$ to be image of $\widehat{R}^{\times} \subseteq \bB^{\times}(\A_f) \rightarrow \bG(\Bbb{A}_f)$ and $K_{\infty} = {\rm SO}(2) \subseteq {\rm SO}(3)$, for instance embedded in the upper left corner, we recover $Y_U$ as the finite quotient of the sphere $\mathcal{S}^2$ from the introduction. 

When $B = \text{Mat}_{2 \times 2}$, then $d_B = 1$, $\bG = \mathbf{PGL}_2$. In this case, almost maximal is the same as maximal, and $Y=\mathbf{PGL}_2(\Q)\backslash\mathbf{PGL}_2(\A)/\mathbf{PGL}_2(\widehat{\Z}){\rm PSO}(2)\simeq {\rm PSL}_2(\Z)\backslash\mathbb{H}$.
\medskip

For later purposes, we fix local Haar measures ${\rm d}g_v$ on $\bG(\Q_v)$ for each place $v$ in the following way. If $v\in {\rm Ram}_B$, we let ${\rm d}g_v$ be the unique probability Haar measure on $\bG(\Q_v)$. For primes $p\nmid d_B$ we normalize ${\rm d}g_p$ to give the maximal compact subgroup $K_p$ volume 1. If $B$ is indefinite we let ${\rm d}g_\infty$ denote the choice of Haar measure for which ${\rm d}g_\infty={\rm d}x\, {\rm d}^\times y\, {\rm d}k$ in the standard coordinates for the Iwasawa decomposition $\bG(\R)=\mathbf{N}(\R)\mathbf{A}(\R)K_\infty$. Then $\prod_v {\rm d}g_v$, which induces the product measure on open sets of the form $\bG(\Q_S)\times\prod_{p\notin S}K_p$, where $S\supset {\rm Ram}_B$, defines a Haar measure on $\bG(\A)$. We continue to write $\prod_v {\rm d}g_v$ for the measure on $[\bG(\A)]$ given by the quotient with the counting measure on $\bG(\Q)$. Let $C_\bG>0$ denote its total volume and put $m_\bG=C_\bG^{-1}\prod_v {\rm d}g_v$. Then $m_\bG$ is the unique $\bG(\A)$-invariant Borel probability measure on $[\bG(\A)]$. Let $\mu$ denote the push-forward of $m_\bG$ to $Y_U$. 
We shall sometimes refer to $\mu$ as the \textit{uniform probability measure}. Then, for $f_1,f_2\in L^2(Y_U,\mu)$, the inner product
\begin{equation}\label{eq:inner-prod}
\langle f_1,f_2\rangle=\int_{Y_U} f_1(y)\overline{f_2(y)}\, {\rm d} \mu(y)=\int_{[\bG(\A)]} f_1(g)\overline{f_2(g)}\, {\rm d} m_\bG(g)
\end{equation}
is the restriction of the natural inner product on $L^2([\bG(\A)],m_\bG)$. In other words, $L^2(Y_U,\mu)=L^2([\bG(\A)],m_\bG)^U$.

\subsection{Heegner packets}\label{sec:special}
A \emph{homogeneous toral datum} is a pair $\scD=(\iota,g)$ with $\iota\colon E\to B$ a ring embedding of a quadratic field $E/\mathbb{Q}$ in $B$, and an element $g\in \bG(\A)$. For a fixed quadratic field $E$, the existence of some $\iota$ is equivalent, by the Albert--Brauer--Hasse--Noether theorem, to asking that every rational place $v$ lying in ${\rm Ram}_B$ be inert or ramified in $E$. The map $\iota$ induces an embedding of the algebraic torus ${\rm Res}_{E/\Q}\mathbb{G}_m$ into $\bB^\times$ and of the projective torus $\mathbb{G}_m\backslash{\rm Res}_{E/\Q}\mathbb{G}_m$ into $\bG$, whose respective images we denote by $\widetilde{\bT}$ and $\bT$.  Following \cite[\S 4]{ELMV2} we call the subset $[\bT(\A)g]\subset [\bG(\A)]$ a \emph{homogeneous toral set}.

In accordance with Remark \ref{rem:always-imag-quad}, we shall henceforth assume that $E$ is imaginary quadratic. Then the group of real points of the projective torus $\bT(\R)$ is conjugate to the fixed compact maximal torus $K_\infty\simeq {\rm PSO}(2)$ of $\bG(\R)$. We will restrict in the following to the case that $\bT(\R)=g_\infty K_\infty g_\infty^{-1}$, and we will call $\scD$ and $[\bT(\A)g]$ in this case a \emph{homogeneous Heegner datum} and a \emph{homogeneous Heegner set} respectively. Let $U= U_f K_\infty<K$ be a finite index subgroup as above. Following \cite[\S 5]{ELMV2} we shall call the projection
\begin{equation}\label{def:Heegner-packet}
H_\scD=[\bT(\A)g]_{U}\subset [\bG(\A)]_U
\end{equation}
a \emph{Heegner packet}. We suppress the dependence of $U$ in the notation because it is assumed to be fixed, while the Heegner packets will vary.

\begin{remark}\label{rem31}
In classical language, if $Y_U$ is a modular curve with $\Gamma_0(N)$ level structure and all primes dividing $N$ split in $E$, then the points of $H_\scD$ are Heegner points of $Y_U$. More generally, if $B$ is indefinite then the points of $H_\scD$ are special points of $Y_U$, and the theory of complex multiplication establishes that $H_\scD$ is a single Galois orbit of a special point on $Y_U$. For more details, see \cite[\S 2.3]{BD}. 

When $B=B^{(2,\infty)}$ is the Hamiltonian quaternion algebra, then $E=\Q(\sqrt{-d})$ with $d$ squarefree embeds in $\bB(\Bbb{Q})$ if and only if $d \equiv 1, 2$ (mod 4) or $d \equiv 3$ (mod 8). If we  choose $g_f=e$, the set $H_\scD$ recovers $\mathcal{R}_d$ as defined in \eqref{eq:def-Rd}. 
\end{remark}

Set $U_\scD=\bT(\A)\cap gUg^{-1}$, this is a compact subgroup of the torus $\bT(\A)$. 
If we write $U=U_f K_\infty$ with $U_f<K_f$ of finite index, then $U_\scD=U_{\scD,f}\cdot \bT(\R)$, where $U_{\scD,f}=\bT(\A_f)\cap g_f U_f g_f^{-1}$ is a compact open subgroup of $\bT(\A_f)$.

\begin{lemma}\label{lem:bij-HeegnerSetPacket}
The map
\begin{equation}\label{eq:Heegner_set2packet}
\bT(\Q)\backslash \bT(\A)/U_\scD=[\bT(\A)]_{U_\scD}\to [\bT(\A)g]_{U}=H_\scD
\end{equation}
induced by $\bT(\A)\hookrightarrow\bG(\A)$, $t\mapsto t g$, is a bijection.
\end{lemma}
\begin{proof} The surjectivity of \eqref{eq:Heegner_set2packet} is clear. To check injectivity it is sufficient to show that if $\gamma t_1 g u=t_2 g$ for some $t_1,t_2\in \bT(\A)$, $\gamma\in\bG(\Q)$ and $u\in U$ then $\gamma\in \bT(\Q)$ and $u\in g^{-1}\bT(\A)g$. We can rewrite  the archimedean component as $\gamma = t_{2,\infty} (g_\infty k_\infty^{-1} g_\infty^{-1})t_{1,\infty}^{-1}\in \bT(\R)$, where $k_\infty\in K_\infty$ is the archimedean component of $u$. Hence $\gamma\in \bT(\R)\cap\bG(\Q)=\bT(\Q)$ as required. It is   straightforward to verify $u=g^{-1} t_1^{-1} \gamma^{-1} t_2 g\in g^{-1}\bT(\A) g$ as required.
\end{proof}
The left-hand side of \eqref{eq:Heegner_set2packet} is a finite set, because algebraic tori have finite class number \cite[VI, Theorem 1]{La}. We deduce from the lemma that the Heegner packet $H_\scD$ has finite cardinality.

\subsection{Local and global toral measure normalizations}\label{sec:periodic-measure}
The homogeneous toral set $[\bT(\A)g]$ carries an action of the locally compact group $g^{-1}\bT(\A) g$. Moreover, it is a principal homogeneous space for this action. Because $\bT$ is anisotropic over $\Q$, the homogeneous space $[\bT(\A)]$ carries a unique Borel \emph{probability} measure ${\rm d}t$ that is invariant under the $\bT(\A)$-action. Pushing forward by the map of right multiplication by $g$ we get a unique Borel probability measure on $[\bT(\A)g]$ invariant under $g^{-1}\bT(\A) g$, which we again denote by ${\rm d} t$. This is the \emph{periodic measure} supported on the homogeneous toral set. 

For later purposes, especially the estimation of local toric integrals in Section \ref{sec:main-estimate}, it will be necessary to fix local measures ${\rm d}t_v$ on $\bT(\Q_v)$ for every place $v$, in such as way that ${\rm d}t=\otimes_v {\rm d}t_v$. To do so, we first assign normalizations to the Haar measures of $\Q_v^\times$ and $E_v^\times$, form the product measure over all $v$, and then push forward the quotient of these to the projective torus via the embedding $\iota: \A_E^\times/\A^\times\xrightarrow{\,\sim\,}\bT(\A)$.

In more detail, let $F$ be a number field (to be taken to be either $\Q$ or $E$ in what follows). We let $e_\Q: \Q\backslash\A\rightarrow \C$ be the unique additive character whose restriction to $\R$ is $e^{2\pi i x}$. Denote by ${\rm d}_Fx$ the unique self-dual Haar measure on $\A_F$ relative to the additive character $e_F=e_\Q\circ {\rm tr}_{F/\Q}:F\backslash\A_F\rightarrow\C$. Then ${\rm d}_Fx$ assigns $F\backslash \A_F$ volume 1. At any place $v$ of $\Q$ let ${\rm d}_{F_v}x$ be the additive Haar measure  on $F_v=\prod_{w\mid v}F_w$ induced by ${\rm d}_Fx$; when $v=p$, the ring of integers $\scO_{F_p}$ of $F_p$ is of volume $|\operatorname{Disc}(\scO_{F_p})|_p^{1/2}$. We then normalize the local multiplicative Haar measure on $F_v^\times$ by putting ${\rm d}^\times_{F_p} y_p=\left(\prod_{w\mid p}\zeta_{F,w}(1)\right) {\rm d} y_{E,p}/|y_p|_p$, when $v=p$, so that $\scO_{F_p}^\times$ is given volume $|\operatorname{Disc}(\scO_{F_p})|_p^{1/2}$, and ${\rm d}^\times_{F_v} y_v={\rm d} y_{F,v}/|y_v|_v$, when $v\mid\infty$. Let $\otimes_v {\rm d}_{F_v}^\times y_v$ denote the unique Haar measure on $\A_F^\times$ inducing the product measure on open sets of the form $F_S^\times \times\prod_{p\notin S}\scO_{F_p}^\times$. We write ${\rm d}^\times_F y$ for the measure on $\A^1_F$ given by the quotient of $\otimes_v {\rm d}_{F_v}^\times y_v$ by ${\rm d}t/t$ on $\R_+^\times$. Since ${\rm d}^\times_F y$ assigns $F^\times\backslash \A_F^1$ volume 
\begin{equation}\label{rho}
   \rho_F={\rm Res}_{s=1}\zeta_F(s),
\end{equation}   
    the measure $\mu_F=\rho_F^{-1}{\rm d}^\times_F y$ is a probability Haar measure on $F^\times\backslash\A_F^1$.

With this notation in place, we let ${\rm d}t$ denote the pushforward of $\mu_E/\mu_\Q$ on $[\bT(\A)]$ under the isomorphism $(\A_E^1/E^\times)/(\A^1/\Q^\times)=(\A_E^\times/E^\times)/(\A^\times/\Q^\times) \xrightarrow{\,\sim\,}[\bT(\A)]$ induced by $\iota$.

We now pass to measures on the Heegner packet $H_\scD=[\bT(\A)g]_U$. The latter does not carry a torus action, but it does admit an induced action by the abelian group $[\bT(\A)]_{U_\scD}$. Since  the periodic measure ${\rm d}t$ on $[\bT(\A)g]$ is $g^{-1}\bT(\A) g$-invariant, the push-forward of ${\rm d}t$ under the quotient map $[\bT(\A)g]\to [\bT(\A)g]_U$ is itself $[\bT(\A)]_{U_\scD}$-invariant. As $[\bT(\A)]_{U_\scD}$ acts simply transitively, the push-forward is the uniform probability measure.

\subsection{Intersection of Heegner packets}
Different choices of the Heegner datum $\scD=(\iota, g)$ may give rise to the same homogeneous Heegner set $[\bT(\A)g]$. For example, if $\gamma \in \bG(\Q)$ then $[(\gamma \bT \gamma^{-1})(\A)(\gamma g)]=[\bT(\A)g]$. We now show that if two Heegner packets intersect then they coincide, and they necessarily arise from this construction.

\begin{lemma}\label{lemma:disjoint-packets}
Let $\bT_1,\bT_2<\bG$ be two maximal tori, defined and anisotropic over $\Q$. Fix $g_1,g_2\in\bG(\A)$ and write $g_i=(g_{i,v})_v$ for $i=1,2$. Assume also $g_{i,\infty}^{-1}\bT_i(\R)g_{i,\infty}=K_\infty$ for $i=1,2$. Let $U<K$ be a finite index subgroup. If $[\bT_1(\A)g_1]_U\cap [\bT_2(\A)g_2]_U\neq \emptyset$ then $\bT_2=\gamma \bT_1 \gamma^{-1}$ for some $\gamma\in \bG(\Q)$ and $g_2\in \bT_2(\A) \gamma g_1 U$. In particular, $[\bT_1(\A)g_1]_U=[\bT_2(\A)g_2]_U$.
\end{lemma}
\begin{remark}
In the case $\bG$ is indefinite part of this lemma is evident from  the point of view of Shimura varieties and the theory of complex multiplication; see \cite[\S 2.3]{BD}. Indeed, the set $H_\scD$ can be identified with a Galois orbit of a special point on the Shimura variety $\bG(\Q) \backslash( \C- \R )\times \bG(\A_f) \slash U_f$. Hence, if two Galois orbits intersect, then they coincide. We provide a complete group-theoretic proof of the lemma.
\end{remark}
\begin{proof}
Let $t_i\in\bT_i(\A)$ and $\gamma\in\bG(\Q)$ satisfy $\gamma t_1 g_1 U = t_2 g_2 U$. The archimedean component of this expression can be written as $\gamma g_{1,\infty} (g_{1,\infty}^{-1} t_{1,\infty} g_{1,\infty}) \in g_{2,\infty }(g_{2,\infty}^{-1} t_{2,\infty} g_{2,\infty}) K_\infty$. The assumption $g_{i,\infty}^{-1}\bT_i(\R)g_{i,\infty}=K_\infty$ now implies $\gamma g_{1,\infty} \in g_{2,\infty} K_\infty$. Thus $\gamma \bT_1(\R) \gamma^{-1}= \gamma g_{1,\infty} K_\infty g_{1,\infty}^{-1} \gamma^{-1}=g_{2,\infty} K_\infty g_{2,\infty}^{-1}=\bT_2(\R)$. 

We need to show that $\gamma$ conjugates not only the real points of the tori, but also the algebraic torus varieties defined over $\Q$. We use the  fact that  the image of $\bT_i(\Q)$ in $\bT_i(\R)$ is the intersection $\bT_i(\R)\cap\bG(\Q)$. This implies $\gamma \bT_1(\Q) \gamma^{-1} =\gamma \bT_1(\R) \gamma^{-1} \cap \bG(\Q)=\bT_2(\R)\cap \bG(\Q)=\bT_2(\Q)$. The rational points of a torus defined over $\Q$ are Zariski dense in the torus. This implies that the closed subschemes $\gamma \bT_1 \gamma^{-1}$ and $\bT_2$ coincide.

We rewrite the expression $\gamma t_1 g_1U =  t_2 g_2 U$ in the form  $t_2^{-1} (\gamma t_1 \gamma^{-1} )\gamma g_1 U=g_2 U$. Thus $g_2\in \bT_2(\A) \gamma g_1 U$, as claimed. 
\end{proof}

\subsection{The order and discriminant of a Heegner packet}\label{sec:disc}
In \cite[\S 4.2]{ELMV2} a notion of discriminant is defined for a homogeneous toral set $[\bT(\A)g]$, generalizing the discriminant of a Heegner point. We review a definition adapted to homogeneous Heegner sets on unit groups of quaternion algebras.

For each finite place $p$ of $\mathbb{Q}$ we consider $\scO_{\scD,p}=\iota(E\otimes \Q_p)\cap g_p R_p g_p^{-1}$. This is an order in the quadratic \'etale algebra $\iota(E\otimes \Q_p)\simeq E\otimes \Q_p$. 

If $B$ ramifies at $p$, then $B_p$ contains a unique maximal order, so that $R_p=g_p R_p g_p^{-1}$. Moreover, $R_p$ is the set of all integral elements of $B_p$, and $\scO_{\scD,p}$ is the set of all integral elements of $\iota(E\otimes \Q_p)$. Hence $\scO_{\scD,p}$ is necessarily the unique maximal order of the field $\iota(E\otimes \Q_p)$ for all $p\in {\rm Ram}_B$.

We define the local discriminant $D_p$ of the homogeneous Heegner set $[\bT(\A)g]$ to be the discriminant of the order $\scO_{\scD,p}$. The global discriminant is defined as $D=\prod_p D_p$. Equivalently, $R^g=\cap_p g_p R_p g_p^{-1}\subset B$ is an order in $B$ (the intersection is well-defined because $g_p R_p g_p^{-1}=R_p$ for all but finitely many primes $p$), and $\scO_\scD=\iota(E)\cap R^g$ is an order in the field $\iota(E)\simeq E$. Then $D=|\operatorname{disc}(\scO_\scD)|$. 

Finally, we observe that the discriminant $D$ and the global quadratic order $\scO_\scD$ are well-defined invariants of a Heegner packet $[\bT(\A)g]_U$ for $U<K$. Indeed, $\scO_{\scD,p}$ does not change if we replace $g_p$ by any element of $g_p K_p$, hence the local orders, the global order, and the discriminant depend only on the class of $g$ in $\bG(\A)\slash K$. 

\begin{remark}
Assuming $g_f=e$, the embedding $\iota\colon E\to B$ is optimal if and only if $\scO_\scD=\scO_E$, if and only if the discriminant $D$ is fundamental.
\end{remark}

\begin{remark}\label{rem:almost-max-pic}
If $U=K$ is almost maximal, there is a canonical identification
\begin{equation}\label{eq:iso-Pic}
{\rm Pic}(\scO_\scD)\simeq [\bT(\A)]_{K_\scD}=\bT(\Q)\backslash \bT(\A_f)/K_{\scD,f}.
\end{equation}
In particular, the measure on $H_\scD$ from \S\ref{sec:periodic-measure} is ${\rm Pic}(\scO_\scD)$-invariant.  This makes explicit the action of the Picard group ${\rm Pic}(\scO_\scD)$ on the Heegner packet $H_\scD$. We caution, however, that when $U$ is not almost maximal, it is not in general true that $[\bT(\A)]_{U_\scD}$ is the Picard group of an order in $E$. 
\end{remark}

\section{The adelic mixing conjecture and the main theorem}\label{sec:adelic-mixing-main-theorem}
In this section, we formulate an adelic version of the mixing conjecture, in the presence of a fixed level structure. We then state the main result of this paper, Theorem \ref{main-thm}, in its most general and quantitative form.

\subsection{The classical shifting element}\label{sec:shift}
Let $\scO_\scD$ be the global order defined in \S\ref{sec:disc}. Consistently with Conjecture \ref{MV-conj} and Theorem \ref{thm:2nd-version}, we make the following definition.

\begin{defn}
For $[\sk]\in {\rm Pic}(\scO_\scD)$ let
\begin{equation}\label{eq:defn-q}
q:=\min_{\substack{\mathfrak{n}\subset\scO_{\scD}\\ [\mathfrak{n}]=[\sk]}}\Nr\mathfrak{n}.
\end{equation}
Thus $q$ is the minimal norm of an integral ideal representing the class $[\sk]$.
\end{defn}

\begin{remark}
By Minkowski's theorem, one has $q\leq (2/\pi) \sqrt{|{\rm disc}(\scO_{\scD})|}$.
\end{remark}

In \S\ref{sec:large-q-range} we considered the reduced primitive binary integral quadratic form $Q$ of discriminant $|{\rm disc}(\scO_E)|$ associated with the class $[\sk]\in {\rm Pic}(\scO_E)$. This correspondence holds more generally, as we now recall; see e.g.\ \cite[Section 7]{Cox}.  Let $\mathfrak{n}$ be an invertible ideal in the global order $\mathscr{O}_\scD$, and fix an oriented $\Z$-basis $[\alpha,\beta]$ for $\mathfrak{n}$. Then $f_{\mathfrak{n},[\alpha,\beta]}=\Nr(\alpha x + \beta y)/\Nr\mathfrak{n}$ defines an element in $F_D$, the set of primitive binary integral quadratic forms of discriminant $D$. Any other choice of oriented $\Z$-basis yields a form in the same equivalence class under the action of $\SL_2(\Z)$. The resulting map onto $\SL_2(\Z)\backslash F_D$ descends to a bijective correspondence $[\mathfrak{n}]\leftrightarrow \bar{f}_{[\mathfrak{n}]}:=\SL_2(\Z).f_{\mathfrak{n},[\alpha,\beta]}$ from the Picard group ${\rm Pic}(\scO_{\scD})$. Let ${\rm Red}_D\subset F_D$ denote the reduced forms, which constitute a set of representatives for the $\SL_2(\Z)$-action. We obtain a bijection
\begin{equation}\label{eq:quad-form-assignment}
{\rm Pic}(\scO_{\scD})\xrightarrow{\sim} {\rm Red}_D,\qquad [\mathfrak{n}]\mapsto Q_{[\mathfrak{n}]},
\end{equation}
where $Q_{[\mathfrak{n}]}$ is the unique reduced form representing the class $\bar{f}_{[\mathfrak{n}]}$.

\begin{remark}\label{rem:Q-and-q}
An integer $n$, coprime to the conductor of $\scO_\scD$, is represented by a class in $\SL_2(\Z)\backslash F_D$ precisely when it is represented as the norm of an invertible ideal in $\scO_\scD$. We deduce that the integer $q$, as defined in \eqref{eq:defn-q}, is the smallest non-zero integer represented by $Q_{[\sk]}$. The previous discussion shows in particular
\begin{equation}\label{orth-char}
  \#\{[\mathfrak{n}] = [\mathfrak{s}]  : {\rm Nr}\, \mathfrak{n} = n\}  = \frac{1}{w} \#\{(x, y) \in \Bbb{Z}^2 : Q_{[\mathfrak{s}]}(x, y) = n\}
\end{equation}
for any ideal $\mathfrak{s}$, where $w$ is the number of roots of unity in $\scO_\scD$.
\end{remark}

\subsection{The adelic shifting element}\label{sec:adelic-shift}

We are working our way towards the formulation of an adelic version of the mixing conjecture, for which we shall instead fix a shifting element $s\in\bT(\A_f)$. 

We begin by explicating the relation between elements of $\bT(\A_f)$ and the group $\mathcal{I}(\scO_\scD)$ of all invertible proper fractional $\scO_\scD$-ideals. The map $(E\otimes \A_f)^\times \slash \widehat{\scO}_\scD^\times \to \mathcal{I}(\scO_\scD)$ given by $(\lambda_p)_p\mapsto \bigcap_p \lambda_p \scO_{\scD,p}$ is a bijection, because the ring $\scO_\scD$ is Gorenstein and every element of $\mathcal{I}(\scO_\scD)$ is everywhere locally principal. This induces a bijection
\begin{equation*}
\A_f^\times \backslash (E\otimes \A_f)^\times \slash \widehat{\scO}_\scD^\times\xrightarrow{s\mapsto \cap_p s_p \scO_{\scD,p}}\mathcal{I}(\Z) \backslash \mathcal{I}(\scO_\scD) \to \Q^\times \backslash \mathcal{I}(\scO_\scD),    
\end{equation*}
where the last equality holds because $\Q$ has class number $1$. We can now use the embedding $\iota$ to define a bijection
\begin{equation}\label{eq:s-to-ideal}
\bT(\A_f) \slash {K_{\scD,f}} \xrightarrow{\sim} \Q^\times \backslash \mathcal{I}(\scO_\scD).    
\end{equation}

Recall that an integral ideal $I$ in $\scO_\scD$ is said to be \textit{primitive} if every integral ideal in its homothety class $\mathbb{Q^\times} I$ is an integral multiple of $I$. In particular, $I$ is not divisible by any prime ideal in $E$ which is inert over $\Q$.  Every homothety class of ideals has a unique primitive integral representative, because it has a unique integral ideal of minimal norm. We then define the \emph{primitive integral ideal associated with $s\in\bT(\A_f)$} to be the unique primitive integral ideal $\sk\subset \scO_\scD$ representing the image of $s$ under the map \eqref{eq:s-to-ideal}.
 
\begin{remark}
The map \eqref{eq:s-to-ideal} induces the isomorphism \eqref{eq:iso-Pic}, a fact which justifies the choice of notation $[\sk]$ for the image of $s$ under the map \eqref{eq:iso-Pic}.
\end{remark}

\subsection{Joint Heegner packets}\label{sec:joint-Heegner-packet}
We shall now work with two copies of $\bG$. Fix a homogeneous toral datum $\scD$ and let $\bT<\bG$ be the associated torus. We write $\Delta\bT$ for the image of the diagonal embedding of $\bT$ into the product $\bG\times\bG$. 

For $s\in \bT(\A_f)$ we consider the \textit{joint homogeneous toral set}
\[
[\Delta\bT(\A)(g,sg)]\subset [(\bG\times\bG)(\A)].
\]
Note that the shifts $s$ and $t_\Q s$ give rise to the same joint homogeneous toral set for all $t_\Q\in\bT(\Q)$. Recall the periodic measure ${\rm d}t$ on $[\bT(\A)]$ from \S\ref{sec:periodic-measure}. The push-forward of ${\rm d}t$ under the map
\begin{equation}\label{eq:tg-tsg}
[\bT(\A)]\rightarrow [(\bG\times \bG)(\A)], \quad [t]\to [(tg,tsg)],
\end{equation}
defines a $(g,sg)^{-1}\Delta\bT(\A)(g,sg)$-invariant Borel probability measure on $[(\bG\times \bG)(\A)]$, supported on $[\Delta\bT(\A)(g,sg)]$. Evaluating it on an integrable function $\varphi\colon [(\bG\times\bG)(\A)]\to \C$ yields the period integral
\begin{equation}\label{defn:PEdelta}
P^\Delta_\scD(\varphi;s)=\int_{[\bT(\A)]} \varphi(tg,tsg) \, {\rm d} t.
\end{equation}

We now add level structure and furthermore assume that $\scD$ is a homogeneous Heegner datum, as defined in \S\ref{sec:special}. Let $U_f\subset K_f$ be of finite index and write $U=U_fK_\infty$ as in \S\ref{sec:Shimura}. We continue to denote by $[\,\cdot\,]_{U\times U}$ the canonical quotient map $(\bG\times\bG)(\A)\rightarrow [(\bG\times\bG)(\A)]_{U\times U}=Y_U\times Y_U$. We then consider the following \emph{joint Heegner packet}
\[
H^\Delta_\scD(s)=[\Delta\bT(\A)(g,sg)]_{U\times U}\subset Y_U\times Y_U.
\]
For $\scD$ fixed, the joint packet depends only on the class $[s]$ of $s$ in the double quotient $[\bT(\A)]_{U_\scD}=\bT(\Q)\backslash \bT(\A_f)/U_{\scD,f}$; if $U=K$ is almost maximal, this is the class of $s$ in the Picard group ${\rm Pic}(\scO_\scD)$. 

\begin{remark}
If $B$ is indefinite then $H^\Delta_\scD(s)$ is a single Galois orbit of a special point of the form $(x,[s].x)\in Y_U\times Y_U$ for any $x\in H_\scD$.
When $B=B^{(2,\infty)}$ is the Hamiltonian quaternions, $E=\Q(\sqrt{-d})$ and $g_f=e$, the shifted set $\mathcal{R}_d^\Delta(s)$ from \eqref{eq:RdDelta} breaks up into a disjoint union of at most two such joint packets.
\end{remark}

The map \eqref{eq:tg-tsg} induces a map
\begin{equation}\label{eq:Heegner_set2jointpacket}
[\bT(\A)]_{U_\scD}\to [\Delta\bT(\A)(g, s g)]_{U\times U}=H^\Delta_\scD(s).
\end{equation}
This is well-defined because $\bT$ is abelian and $s U_\scD s^{-1}=U_\scD$.

\begin{lemma}\label{lem:shifted-diag-torus}
The map \eqref{eq:Heegner_set2jointpacket} is a bijection.
\end{lemma}
\begin{proof}
It is evidently surjective. The injectivity follows from the observation that the composition of  \eqref{eq:Heegner_set2jointpacket} with the projection map onto the first coordinate is the map $[\bT(\A)]_{U_\scD}\to[\bT(\A)g]_U$, which is injective by Lemma \ref{lem:bij-HeegnerSetPacket}.
\end{proof}

By virtue of Lemmata \ref{lem:bij-HeegnerSetPacket} and \ref{lem:shifted-diag-torus} the map  
\begin{equation}\label{eq:alphaU}
H_\scD\to H^\Delta_\scD(s),\qquad [tg]\mapsto [(tg,tsg)].
\end{equation}
is a $[\bT(\A)]_{U_\scD}$-equivariant bijection. Using the commutativity of $\bT(\A)$ to write $tsg=sx$, where $x=tg$, it follows that $H^\Delta_\scD(s)=\{(x,sx): x\in H_\scD\}$, and the uniform measure on $H_\scD$ projects to the same on $H^\Delta_\scD(s)$. Since, from \S\ref{sec:periodic-measure}, the former derives from the periodic measure on $[\bT(\A)]$, we obtain, for an integrable  function $\varphi\colon Y_U\times Y_U\to \C$,
\begin{equation}\label{eq:defn-PEdelta}
P_\scD^\Delta(\varphi;s)=\frac{1}{|H_\scD|}\sum_{x\in H_\scD}\varphi(x,sx).
\end{equation}
We call $P_\scD^\Delta(\varphi;s)$ the \textit{joint Heegner period}. This recovers the expression \eqref{eq:joint-Heegner-period}, when the shifting element $g\in\bG(\A)$ in the homogeneous Heegner datum $\scD=(\iota,g)$ is trivial at finite places.

\subsection{The conjecture and our main result}\label{subsec:mixing-conj}
We are now ready to present the mixing conjecture of Michel and Venkatesh \cite{MV}. It asserts minimal conditions under which joint Heegner packets should tend to the product of the uniform measures on $Y_U\times Y_U$.

\begin{conj}[The adelic mixing conjecture at level $U$]\label{conj:adelic-mixing}
Let $\varphi\in C_c(Y_U\times Y_U)$ and fix a sequence of homogeneous Heegner data $\{\scD_i\}_i$ and shifts $s_i\in\bT_i(\A_f)$. Denote by $D_i$ the discriminant of $\scO_{\scD_i}$ and let $q_i$ be the minimal norm of an integral ideal representing the class of $s_i$ in $[\bT_i(\A_f)]_{K_{\scD_i}}\simeq \operatorname{Pic}(\scO_{\scD_i})$. Assume $|D_i|\to \infty$ and $q_i\to \infty$ as $i \to \infty$. Then
\begin{equation*}
P^\Delta_{\scD_i}(\varphi;s_i)\rightarrow \int_{Y_U\times Y_U} \varphi \, {\rm d} (\mu\times \mu)
\end{equation*}
as $i\rightarrow\infty$.\end{conj}

We prove the above conjecture under a variety of assumptions, among which is a maximality condition on the level structure and (when $B=\Mat(\Q)$) a discreteness condition on the space of eligible test functions. The precise result is as follows.

\begin{theorem}[Main Theorem:\ detailed version]\label{main-thm}
Assume $U=K$ is almost maximal. Let $\scD=(\iota,g)$ be a homogeneous Heegner datum with a fundamental discriminant $D$, and let $s\in\bT(\A_f)$.
Assume the generalized Riemann hypothesis and, whenever $B$ is indefinite, the Ramanujan conjecture for $\mathbf{PGL}_2\slash\Q$. For $\varphi\in L^2_{\rm disc}(Y_K\times Y_K,\mu\times\mu)$ it holds that 
\[
P_\scD^\Delta(\varphi;s)=\int_{Y_K\times Y_K} \varphi \, {\rm d} (\mu\times \mu) +O_{\varepsilon,\varphi, B}(h_{\varepsilon}(q,D)),
\]
where
\[
h_{\varepsilon}(q,D)=\begin{cases}
q^{-1/2 +\varepsilon} + q^{1/2} D^{-1/4 + \varepsilon}, & 1\leq q\leq D^{1/2-\varepsilon};\\
(\log D)^{-1/18+\varepsilon}, & q\geq D^{1/2-\varepsilon}.
\end{cases}
\]
The error rate $h_{\varepsilon}(q,D)$ does not depend on the choice of the almost maximal compact subgroup $K$. In particular, when $B$ is a division algebra, if $\{\scD_i\}_i$ is a sequence of Heegner data with fundamental discriminants $D_i\to\infty$, and $s_i$ is a sequence of elements in $\bT_i(\A_f)$ with $q_i\rightarrow\infty$, then $H_{\scD_i}^\Delta(s_i)\subset Y_K\times Y_K$ equidistributes according to the product measure $\mu\times\mu$.
\end{theorem}

One can check that the dependence on $\varphi$ in the error term is through an appropriate Sobolev norm. Since $B$ is taken to be fixed throughout this paper, we will usually drop the subscript $B$ from asymptotic inequalities of the form $\ll_B$, as in the error term in Theorem \ref{main-thm}. The specialization to the set-up described in Remark \ref{rem31} recovers Theorem \ref{thm:sphere}. 

\subsection{Reductions of the main theorem}\label{sec:reductions}
To prove Theorem \ref{main-thm}, it suffices to show 
\begin{equation}\label{eq:joint-period-bound}
P_\scD^\Delta(\varphi;s)\ll_{\varepsilon,\varphi} h_{\varepsilon}(q,D ),
\end{equation}
for $\varphi\in L^2_{\rm cusp}(Y_K\times Y_K,\mu\times\mu)$, since the case when one of the components of $\varphi$ is the constant function is covered by Duke's theorem.  See the discussion in \S\ref{sec:Weyl-crit}. We now give two further reductions, having to do with the Heegner datum $\scD$ and the shifting element $s$.

\subsubsection{Reduction to Heegner data with archimedean translation}\label{reduction-to-arch-trans}
Theorem \ref{main-thm} can be reduced to the case when the translation $g$ in the Heegner datum $\scD=(\iota,g)$ is purely archimedean. This follows from the fact that $\bG$ has a finite class number.

Indeed, let $\xi_1,\ldots,\xi_h\in \bG(\A_f)$ be a full set of representatives for the double quotient $\bG(\Q) \backslash \bG(\A_f) \slash K_f$. Write $g=\delta \widetilde{g}_{\infty} \xi_{j} k$, with $\delta\in\bG(\Q)$, $\widetilde{g}_{\infty}\in\bG(\R)$ and $k\in K_f$. Consider the modified Heegner datum $\widetilde{\scD}=(\delta^{-1} \iota \delta, \widetilde{g}_\infty)$; then the fact that $\varphi$ is $K\times K$-invariant implies
\begin{equation*}
P_\scD^\Delta(\varphi;s)=P_{\widetilde{\scD}}^\Delta(R_{(\xi_j,\xi_j)}\varphi;\delta^{-1}s\delta)\;,
\end{equation*}
where $R_{(\xi_j,\xi_j)}(\varphi)(h_1,h_2)=\varphi(h_1\xi_j,h_2\xi_j)$ is $\xi_j K \xi_j^{-1}$-invariant. The latter group is also an almost maximal compact subgroup, and, moreover,
$\widetilde{\scD}$ has a purely archimedean translation. We can thus deduce Theorem \ref{main-thm} for a Heegner packet with an adelic translation from the same statement for packets with purely archimedean translations by replacing $O_{\varepsilon,\varphi}(h_{\varepsilon}(q,D))$ by
\begin{equation*}
\max_{1\leq j \leq h}O_{\varepsilon,R_{(\xi_j,\xi_j)}\varphi}(h_{\varepsilon}(q,D)).
\end{equation*}

\subsubsection{Reduction to shifts coprime to \texorpdfstring{$d_B$}{discriminant of B}}\label{sec442}
Assume henceforth the translation $g=g_\infty$ is purely archimedean. Another simple reduction of Theorem \ref{main-thm} is to the case when $s_p$ is trivial for all $p\mid d_B$. This is equivalent to reducing to the case that $\gcd(N,d_B)=1$, where $N$ is the norm of the primitive integral ideal $\sk$ associated with $s$, as in \S\ref{sec:adelic-shift}. Write $s=s_0 s_{d_B}$, where $s_0$ is supported away from $d_B$.  For each prime $p\mid d_B$, fix $\sigma_p\in \bG(\Q_p)- K_p$, and for $l\mid d_B$ define $\sigma_l=\prod_{p\mid l} \sigma_p$. Let $l_s\mid d_B$ be the product of the primes $p\mid d_B$ where $s_p\not\in K_p$.
Then because $\varphi$ is $K\times K$-invariant and $K_p$ is a normal subgroup of index $2$ in $\bG(\Q_p)$ for all $p\mid d_B$, we deduce
\begin{equation*}
P_\scD^\Delta(\varphi;s)=P_\scD^\Delta(R_{(e,s_{d_B})})\varphi;s_0)=P_\scD^\Delta(R_{(e,\sigma_{l_s})}\varphi;s_0).
\end{equation*}
Because all primes dividing $d_B$ are either inert and of order $1$ in the class group, or ramified and of order at most $2$ in the class group, the minimal norm of an integral ideal representing $s_0$ is at most $q/l_s$. Hence we can deduce Theorem \ref{main-thm} for a general shift $s$ from the same statement for shifts supported away from $d_B$  by replacing $O_{\varepsilon,\varphi}(h_{\varepsilon}(q,D))$ by
\begin{equation*}
\max_{l\mid d_B} O_{\varepsilon,R_{(e,\sigma_l)}\varphi}(h_{\varepsilon}(q/l,D)).
\end{equation*}

\section{Statement of the large $q$ estimate and the off-diagonal case}\label{sec:large-q}

The goal of this section is to prove the following result, which establishes \eqref{eq:joint-period-bound} in the most delicate range, when $q$ is close to $D^{1/2}$ . 
\begin{prop}\label{prop2}
For $j=1,2$, let $\sigma_j$  be cuspidal automorphic representations of $\bG(\A)$ with $\sigma_j^K\neq \{0\}$, and let $f_j$ be $L^2$ normalized vectors in $\sigma_j^K$. Let $\scD=(\iota,g)$ be a homogeneous Heegner datum, where $\iota\colon E\to B$ is a ring embedding of an imaginary quadratic field $E/\mathbb{Q}$ of discriminant $-D$, and $g_f=e$. Let $s\in\bT(\A_f)$ and define $q$ as in \eqref{eq:defn-q}.
Assume the generalized Riemann hypothesis.
\begin{enumerate}
\item\label{item-BB} If $\sigma_1\neq \sigma_2$ then $P_\scD^\Delta(f_1\otimes f_2;s)\ll_\varepsilon (\log D)^{-1/4+\varepsilon}$, uniformly in $s$;
\item\label{item-BBK} If $\sigma_1=\sigma_2$, then we can take $f_1=f_2=f$. If $B$ is indefinite, assume the Ramanujan conjecture for $\mathbf{PGL}_2/\Q$. Let $\varepsilon > 0$ and assume that $q \geq D^{1/6 + \varepsilon}$. Then
\[
P_\scD^\Delta(f\otimes f;s) \ll_\varepsilon (\log D)^{-1/18+\varepsilon}.
\]
\end{enumerate}
\end{prop}

The proof of Proposition \ref{prop2} roughly follows the outline of \S\ref{sec:large-q-range}. 

We begin by recalling the definition of the Heegner packet $H_\scD=[\bT(\A)g]_K$ from \eqref{def:Heegner-packet}. The discussion in \S\ref{sec:periodic-measure} defines a $[\bT(\A)]_{K_\scD}$-invariant probability measure ${\rm d} t$ on $H_\scD$, which was shown to be the uniform probability measure. Since the translating element $g$ is purely archimedean (and $U=K$ is almost maximal) the global order $\scO_\scD$ from \S\ref{sec:disc} is the maximal order $\scO_E$ in $E$. We may then identify any class group character $\chi\in {\rm Pic}(\scO_E)^\vee$ with a character of $[\bT(\A)]_{K_\scD}$ via the isomorphism \eqref{eq:iso-Pic}. Composing with the isomorphism of Lemma \ref{lem:bij-HeegnerSetPacket} we may also view $\chi$ as a function on $H_\scD$. We define the \textit{$\chi$-twisted Heegner period} of $f_j$ as
\[
W_\scD(f_j, \chi)=\int_{H_\scD} f_j (t) \chi(t)\, {\rm d} t=\frac{1}{|H_\scD|}\sum_{t\in H_\scD} f_j(t)\chi(t).
\]

We may decompose $L^2(H_\scD, {\rm d} t)$ into a direct sum of $\chi$-isotypic components, where $\chi$ varies over ${\rm Pic}(\scO_E)^\vee$. We apply the Parseval identity to the joint Heegner period in \eqref{eq:defn-PEdelta} to obtain
\begin{equation}\label{eq:Parseval-again}
P_\scD^\Delta(f_1\otimes f_2;[\sk]) =\sum_{\chi\in {\rm Pic}(\scO_E)^\vee} \chi(\sk) W_\scD(f_1, \chi) \overline{W_\scD(f_2, \chi)}.
\end{equation}
This is the starting point of the argument. We break down the analysis into two cases, according to whether $\sigma_1$ and $\sigma_2$ are distinct or not.

We may quickly dispense with the first part of Proposition \ref{prop2}. Indeed, Part \eqref{item-BB}  follows by applying absolute values 
\[
|P_\scD^\Delta(f_1\otimes f_2;[\sk])| \leq \sum_{\chi\in {\rm Pic}(\scO_E)^\vee} |W_\scD(f_1, \chi)W_\scD(f_2, \chi)|
\]
and citing the central result in \cite{BB}. 

\section{The diagonal estimate: sparse sums of Hecke eigenvalues}\label{sec:diag-est}

The proof of Part \eqref{item-BBK} of Proposition \ref{prop2} occupies the rest of this section.

\subsection{Bounds at $s=1$ for general $L$-functions}

We shall need the following standard estimates on automorphic $L$-functions at the point $s=1$.

\begin{lemma}\label{tit} Let $L(s)$ be a holomorphic $L$-function of fixed degree $d$ and analytic conductor $C_L$ in the extended Selberg class, not necessarily primitive, with Dirichlet coefficients $\lambda(n)$. Assume GRH for $L(s)$ and assume that the local roots $\alpha(p, j)$, $1 \leq j \leq d$, satisfy $|\alpha(p, j)| \leq p^{1/2 - \rho}$ for some $\rho > 0$. Then
\begin{equation}\label{log-series}
  \sum_{p \leq x} \frac{\lambda(p)}{p} = \log L(1) + O_{\varepsilon}(1), \quad x \geq (\log C_L)^{2+\varepsilon}.
\end{equation}
Moreover, for fixed $\alpha > 0$ we have 
\begin{equation}\label{titch1}
\begin{split}
&  L(1) \gg (\log\log C_L)^{-\alpha}, \quad \text{if } \lambda(p) \geq -\alpha \text{ for all } p,\\
 & L(1) \ll (\log\log C_L)^{\alpha}, \quad \text{if } \lambda(p) \leq \alpha \text{ for all } p.\\
\end{split}  
\end{equation}
\end{lemma}

\begin{proof}
See \cite[Lemma 5]{BB}.
\end{proof}

\subsection{Reduction to a sparse coefficient sum}\label{sec:reduction-to-sparse}
We now assume that $\sigma_1=\sigma_2=\sigma$. Then the Jacquet--Langlands correspondent $\pi=\sigma^{\rm JL}$ of $\sigma$ is a cuspidal automorphic representation of $\mathbf{GL}_2(\A)$, having trivial central character. We record some information about $\pi$ and introduce some related notation. 

Since $\sigma$ has non-zero invariants under the almost maximal subgroup $K_f$, we deduce that the conductor of $\pi$ is the reduced discriminant $d_B$ of $B$. Moreover since $\sigma_\infty$ has non-zero invariants under $K_\infty$, the archimedean component $\pi_\infty$ of $\pi$, as an irreducible unitary representation of $G=\mathbf{PGL}_2(\R)$, is either a principal series representation of the form ${\rm Ind}_P^G({\rm sgn}^\epsilon |\cdot |^{\nu_\pi},{\rm sgn}^\epsilon |\cdot |^{-\nu_\pi})$, where $P<G$ is the standard Borel subgroup, $\nu_\pi\in i\R\cup (-1/2,1/2)$, and $\epsilon\in \{0,1\}$ (when $B$ is indefinite), or the discrete series representation of weight $k \geq 2$ (when $B$ is definite). In the latter case, we put $\nu_\pi=k$, and refer to $\nu_\pi$ in either case as the ``spectral parameter'' of $\pi=\sigma^{\rm JL}$. 

When $B$ is indefinite, so that $\pi_\infty$ is a principal series, we let $\vartheta_\pi=|{\rm Re}\, \nu_\pi|$ measure the failure of $\pi$ to satisfy the Selberg eigenvalue conjecture. By \cite{KS} we have $\vartheta_\pi\in [0,7/64]$. We extend this definition to the case of $B$ definite (in which case $\pi_\infty$ is the weight $k$ discrete series representation) by putting $\vartheta_\pi=0$.

Finally, we write the Dirichlet series and Euler product expansion of the standard $L$-function for $\pi$ as
\[
L(s,\pi)=\sum_{n\geq 1}\lambda_\pi(n)n^{-s}=\prod_p \prod_{i=1,2} \Big(1-\frac{\alpha_\pi(p,i)}{p^s}\Big)^{-1} \qquad ({\rm Re}\, s>1).
\]
Recall from the statement of Proposition \ref{prop2} that we are assuming the Ramanujan conjecture for $\pi$ at finite places when $B$ is indefinite (for $B$ definite, this is Deligne's theorem). For all primes $p$ the local roots satisfy $|\alpha_\pi(p,i)|\leq 1$ (in fact $|\alpha_\pi(p,i)|=1$ when $p\nmid d_B$).

With the above notation in place, we may now state and prove the following proposition. In it, we use Waldspurger's formula and the approximate functional equation to reduce Part \eqref{item-BBK} of Proposition \ref{prop2} to estimates on the sparse coefficient sum
\[
S_\pi(Y,\mathcal{Q}) :=  \sum_{  Q(x, y) \leq Y}|\lambda_\pi(\mathcal{Q}(x, y) )|,
\]
for a primitive integral positive definite binary quadratic form $\mathcal{Q}$ of discriminant $-D<0$.

\begin{lemma}\label{lemma:conversion-to-sparse-sum}
We have
\begin{equation}\label{sothat}
P_\scD^\Delta(f\otimes f;s)\ll \max_{\substack{  {\rm Nr}(\mathfrak{d}) \mid d_B}}\log\log D  \sum_{Y = 2^{\nu}} \sum_{m} \frac{1}{m} \frac{S_\pi(Y,Q_{[\sk\mathfrak{d}]})}{Y^{1/2}} \Big(1 + \frac{Ym^2}{D}\Big)^{-A} \Big(  \frac{Ym^2}{D}\Big)^{-\vartheta_{\pi } -  \varepsilon}
 \end{equation}
for any fixed $\varepsilon, A > 0$, where, for an ideal class $[\mathfrak{n}]\in {\rm Pic}(\scO_E)$, the quadratic form $Q_{[\mathfrak{n}]}$ is defined in \eqref{eq:quad-form-assignment}.
\end{lemma}

Here and in the following a summation condition of the form $Y=2^{\nu}$ denotes a sum over $\nu \in \Bbb{N}$, and $Y$ is  shorthand notation for $2^{\nu}$. The notation reflects a dyadic partition of unity. 

\begin{proof}
Applying the Parseval relation \eqref{eq:Parseval-again} with $f_1=f_2=f$ yields
\begin{equation}\label{eq:Parseval-diag}
P_\scD^\Delta(f\otimes f;[\sk]) =\sum_{\chi\in {\rm Pic}(\scO_E)^\vee} \chi(\sk) |W_\scD(f, \chi)|^2.
\end{equation}
We may assume that ${\rm Hom}_{\bT(\A)}(\sigma,\chi)\neq 0$ for otherwise $W_\scD(f;\chi)=0$. Under this non-vanishing hypothesis, we may then express $|W_\scD(f;\chi)|^2$ in terms of $L$-values, using an exact form of Waldspurger's theorem, such as that given in \cite[Theorem 1.1]{FMP} and explicated further in \cite[Appendix A]{BB}. Indeed, we have
\begin{equation}\label{eq:FMP-Walds}
|W_\scD(f;\chi)|^2=C_{\bG}C_{\rm Ram}(\pi,\chi)\frac{1}{\sqrt{D}}\frac{1}{L(1,\chi_E)^2} \frac{L(1/2, \pi \times \chi)}{L(1,\pi,{\rm Ad})}F(\pi_\infty),
\end{equation}
where $\chi_E$ is the quadratic character associated to the field extension $E/\Bbb{Q}$, the constant $C_\bG>0$ depends only on $\bG$, $C_{\rm Ram}(\pi,\chi)>0$ depends on $\chi$ and on the local ramified components of $\pi$, and (since $\chi_\infty=1$ in our case) the function $F$ depends only on the spectral parameters of $\pi_\infty$. More precisely, we have
\[
C_{\text{Ram}}(\pi, \chi)=C'_{\text{Ram}}(\pi, D) \prod_{p \mid (D, d_B)} \frac{1}{L_{p}(1/2, \pi \times \chi)}
\]
where $C'_{\text{Ram}}(\pi, D)$ is independent of $\chi$ and bounded by $O_{\pi}(1)$, uniformly in $D$. We recall from \cite[(6.3) -- (6.4)]{BB} that
\[
L_p(s, \pi \times \chi) =  \prod_{i , k= 1, 2}  \Big( 1 - \frac{\alpha(p, i) \xi_{\chi}(p, k)  }{p^s}\Big)^{-1},
\]
where
\[
\{\xi_{\chi}(p, 1), \xi_{\chi}(p, 2)\} = \begin{cases} \{\chi(\mathfrak{p}), \bar{\chi}(\mathfrak{p})\}, & (p) = \mathfrak{p} \bar{\mathfrak{p}}, \mathfrak{p} \not= \bar{\mathfrak{p}},\\\{\chi(\mathfrak{p}), 0\}, & (p) = \mathfrak{p}^2,\\ \{-1, 1\}, & \chi_E(p) = -1.\end{cases}
\]

We now insert \eqref{eq:FMP-Walds} into \eqref{eq:Parseval-diag}. Identifying $\sqrt{D}L(1,\chi_E)$ with $(2\pi/w_E)|{\rm Pic}(\scO_E)|$ by means of the class number formula, where $w_E$ is the number of roots of unity of $\mathscr{O}_E^\times$, we obtain
\[
P_\scD^\Delta(f\otimes f;s) \ll  \frac{1}{|{\rm Pic}(\scO_E)|} \sum_{\substack{  {\rm Nr}(\mathfrak{d}) \mid d_B}} \Big|\sum_{\chi\in {\rm Pic}(\scO_E)^\vee } \chi(\sk\mathfrak{d}) \frac{L(1/2,\pi\times \chi)}{L(1, \chi_E) L(1,\pi, \text{Ad})}\Big|.
\]
From \eqref{titch1}, applied to $L(1, \chi_E)$, we have 
\[
P_\scD^\Delta(f\otimes f;s) \ll  \frac{\log\log D}{|{\rm Pic}(\scO_E)|}\sum_{\substack{  {\rm Nr}(\mathfrak{d}) \mid d_B}} \Big|\sum_{\chi\in {\rm Pic}(\scO_E)^\vee }  \chi(\sk\mathfrak{d}) L(1/2,  \pi\times \chi)\Big|.
\]
The root number of $L(s,\pi\times\chi)$ is independent of $\chi$ and is in $\{\pm 1\}$; we may assume that this root number is 1, since otherwise all twisted central values vanish, which then implies the vanishing of the period $P_\scD^\Delta(f\otimes f;s)$. We open the $L$-function using the approximate functional equation \cite[Theorem 5.3]{IK} 
\begin{equation}\label{approx}
L(1/2, \pi \times \chi)  = \sum_{m, n} \frac{\chi_E(m)}{m} \frac{\lambda_\pi(n) }{n^{1/2}} \sum_{{\rm Nr}(\mathfrak{a}) = n} \chi(\mathfrak{a})W\Big( \frac{nm^2}{D}\Big).
\end{equation}
Here, $W$ depends only on $\pi$ and satisfies $W(x) \ll_{\varepsilon, A, \pi} x^{-\vartheta_{\pi} - \varepsilon}(1 + x)^{-A}$ for every $\varepsilon, A > 0$ with polynomial dependence on the archimedean type of $\pi$; cf.\ \cite[Proposition 5.4]{IK}.  We insert \eqref{approx} and use orthogonality of characters along with \eqref{orth-char} to obtain 
\[
P_\scD^\Delta(f\otimes f;s)\ll \max_{\substack{  {\rm Nr}(\mathfrak{d}) \mid d_B}}\log\log D \sum_{m}\frac{1}{m} \sum_{(x, y) \not= (0, 0)}  \frac{|\lambda_\pi(Q_{[\sk\mathfrak{d}]}(x, y) )|}{Q(x, y)^{1/2}} \Big| W\Big( \frac{Q_{[\sk\mathfrak{d}]}(x, y)}{D/m^2}\Big) \Big|,
\]
Inserting the stated estimate on $W$ and applying a dyadic decomposition, we obtain \eqref{sothat}.\end{proof}

The following sections give estimates on the $S_\pi(Y,Q_{[\sk\mathfrak{d}]})$ appearing in Lemma \ref{lemma:conversion-to-sparse-sum}, using a combination of sieve methods with some distributional results on the coefficients $\lambda_\pi(n)$ stemming from known instances of functoriality of symmetric power lifts of $\pi$.

\subsection{Basic estimate}\label{sec:basic-est}

Let $\mathcal{Q}=Q_{[\sk\mathfrak{d}]}$, where $\mathfrak{d}\mid d_B$. In the following, we establish two bounds for $S_\pi(Y, \mathcal{Q})$, each expressed in terms of the truncation parameter $Y$ as well as two quantities attached to the quadratic form $\mathcal{Q}$: the absolute value $D$ of the discriminant and the smallest non-zero integer represented by $\mathcal{Q}$. Since $\mathfrak{d}\mid d_B$, and $d_B$ is fixed, the minimal value of any such $\mathcal{Q}$ differs from the minimal value $q$ of $Q=Q_{[\sk]}$ (see Remark \ref{rem:Q-and-q}) by a bounded scalar multiple.
 
By \cite[Lemma 3.1]{BG} together with the Ramanujan bound we have the bound
\begin{equation}\label{bound1}  
S_\pi(Y,\mathcal{Q}) =  \sum_{  Q(x, y) \leq Y}|\lambda_\pi(\mathcal{Q}(x, y) )|
 \ll Y^{\varepsilon}\Big( \frac{Y}{D^{1/2}} + \frac{Y^{1/2}}{q^{1/2}}\Big),
\end{equation}
which is of use for $Y$ small relative to $D$. Indeed, when $Y \leq D^{1-10\varepsilon}$, for a small value of $\varepsilon>0$, we may insert \eqref{bound1} into \eqref{sothat} to obtain
\[
P_\scD^\Delta(f\otimes f;s) \ll_{\varepsilon} D^{\varepsilon} \sum_{\substack{Y = 2^{\nu} \\ Y \leq D^{1-10\varepsilon}}}   \Big( \frac{Y}{D}\Big)^{-\theta_{\pi} - \varepsilon} \Big(\frac{Y^{1/2  }}{D^{1/2  }} + \frac{1}{q^{1/2}}\Big)  \ll_{\varepsilon} D^{-\varepsilon},
\]
provided that $q \geq D^{10\varepsilon}$.

In the more delicate range when $Y$ is close to $D$ we shall use \cite[Theorem 9.7]{Kh} to show the following bound. 

\begin{prop}\label{prop:Kh} Suppose that 
\begin{equation}\label{sieve}
 q \geq \frac{D^{5/6}}{Y^{2/3}} (DY)^{\varepsilon}.
\end{equation}
Then, assuming GRH and the Ramanujan conjecture for $\mathbf{PGL}_2$, we have
\begin{equation}\label{bound2}
S_\pi(Y, \mathcal{Q}) =  \sum_{  Q(x, y) \leq Y}|\lambda_\pi(\mathcal{Q}(x, y) )|
\ll_\varepsilon \frac{Y}{D^{1/2}} (\log D)^{-1/18+\varepsilon}.
\end{equation}
\end{prop}

For the portion $Y \geq D^{1-10\varepsilon}$ we need to assume
 $q \geq D^{1/6 + 50 \varepsilon}$ 
in order to satisfy \eqref{sieve}. In this case we may insert \eqref{bound2} into \eqref{sothat} to obtain
\begin{displaymath}
\begin{split}
P_\scD^\Delta(f\otimes f;s) & \ll_\varepsilon (\log D)^{-1/18+\varepsilon} \sum_{\substack{Y = 2^{\nu} \\ Y \geq D^{1-10\varepsilon}}} \sum_m \frac{1}{m} \frac{Y^{1/2}}{D^{1/2}} \Big(1 + \frac{Ym^2}{D}\Big)^{-A} \Big(  \frac{Ym^2}{D}\Big)^{-\theta_{\pi } -  \varepsilon}\\
 & \ll_{\varepsilon} (\log D)^{-1/18+\varepsilon} \sum_{\substack{Y = 2^{\nu} \\ Y \geq D^{1-10\varepsilon}}} \Big(  \frac{Y}{D}\Big)^{1/2 - \theta_{\pi} - \varepsilon} \Big(1 + \frac{Y }{D}\Big)^{-A}   \ll_{\varepsilon} (\log D)^{-1/18+\varepsilon}. 
 \end{split}
 \end{displaymath}
The proof of Proposition \ref{prop2} is therefore reduced to that of Proposition \ref{prop:Kh}. 

\subsection{A sieving argument}
The crucial ingredient to Proposition \ref{prop:Kh} is the following estimate, due to the third author \cite{Kh}, and building on the work of Shiu and Nair, which converts a sum of a multiplicative function $f$ on values of an integral binary quadratic form $\mathcal{Q}$ into short sums and products involving the density function 
\[
\rho_{\mathcal{Q}}(n) = \#\{(x, y) \in (\Bbb{Z}/n\Bbb{Z})^2 \mid \mathcal{Q}(x, y) \equiv 0 \, (\text{mod } n)\}
\]
and the multiplicative function $f$ itself.

\begin{lemma}\label{lemma:Kh} Suppose the conditions \eqref{sieve}. Then, under the Ramanujan conjecture, we have
\[
S_\pi(Y,\mathcal{Q}) \ll \frac{Y}{D^{1/2}} \prod_{ 3 \leq p \leq D^{100}} \Big(1 - \frac{\rho_{Q}(p)}{p^2}\Big) \sum_{\substack{a \leq D^{100} \\ a\textrm{ is } \square-\mathrm{free}}} \frac{|\lambda_\pi(a)| \rho_{\mathcal{Q}}(a)}{a^2}.
\]
\end{lemma}

\begin{proof} We  apply \cite[Theorem 9.7]{Kh} with $f(n) = |\lambda_\pi(n)|$,  $\mathscr{E} = \{(x, y) \in \Bbb{R}^2 \mid  \mathcal{Q}(x, y) \leq Y\}$, $A(\mathscr{E}) =2\pi Y/\sqrt{D}$, and
\[
R_{\text{max}} \leq  \frac{A(\mathscr{E})^{1/2} D^{3/4} }{q^{3/2}} \ll \frac{(YD)^{1/2}}{q^{3/2}}, \quad \theta_l = 2/3 + \varepsilon
\]
by \cite[Lemma 10.4]{Kh} and  \cite[Remark 9.5]{Kh}. Then the existence of constants $A, B$ as in \cite[Definition 9.6]{Kh} and $C_l$ as in \cite[Definition 9.4]{Kh} is guaranteed. The existence of $\eta > 0$ as in \cite[Theorem 9.7]{Kh} follows from the first condition in \eqref{sieve}. It is easy to see (cf.\ \cite[Lemma B.1]{Kh}) that
\begin{equation}\label{den}
\begin{split}
\rho_{\mathcal{Q}}(p) = p(1 + \chi_E(p)) - \chi_E(p) \,\, (p \text{ odd}), \quad \rho_{\mathcal{Q}}(p^k) \leq (k+1) p^k\,\, (k\geq 2).
\end{split}
\end{equation}
In order to apply \cite[Theorem 9.7]{Kh}, we need to check the existence of $r, C, X, \delta$ as defined in \cite[Theorem 9.7]{Kh}. 
By \eqref{den} we  can choose any $r < 1$ and some suitable $C$. Since $q \leq D^{1/2}$ by the Minkowski bound, \eqref{sieve} implies automatically $Y \geq D^{1/2}$. Hence 
we may choose $X = D^{100}$ and $\delta = 600$. Now \cite[Theorem 9.7]{Kh} yields the stated bound, except with the $a$-sum not restricted to square-free numbers. The claim follows from the decomposition
\begin{equation*}
\sum_{a \leq D^{100}} \frac{|\lambda_\pi(a)| \rho_{\mathcal{Q}}(a)}{a^2} \leq
\sum_{\substack{a \leq D^{100} \\ a\textrm{ is } \square-\mathrm{free}}} \frac{|\lambda_\pi(a)| \rho_{\mathcal{Q}}(a)}{a^2} 
\sum_{\substack{a =1 \\ a\textrm{ is powerful}}}^\infty \frac{|\lambda_\pi(a)| \rho_{\mathcal{Q}}(a)}{a^2}
\end{equation*}
using the multiplicativity of the summnds and from the inequality
\begin{equation*}
\sum_{ a\textrm{  powerful}}  \frac{|\lambda_\pi(a)| \rho_{\mathcal{Q}}(a)}{a^2} \ll_\varepsilon
\sum_{  a\textrm{ powerful}} \frac{1}{a^{1-\varepsilon}}  \ll_{\varepsilon} 1
\end{equation*}
for $\varepsilon < 1/2$ that follows from the Ramanujan bound and \eqref{den}.
\end{proof}

\subsection{Bounding short sums using $L$-functions}\label{sec:short-sums}
We now estimate the product and the sum appearing in Lemma \ref{lemma:Kh}, under the assumption of GRH. 

The following result uses the $L$-function estimates of Lemma \ref{tit} along with the existence of the first few symmetric functorial lifts from $\mathbf{GL}_2$.

\begin{lemma}\label{lem:Kh-sum} Under GRH and the Ramanujan conjecture for $\mathbf{PGL}_2$ we have 
\[
\sum_{\substack{a \leq D^{100} \\ a\textrm{ is } \square-\mathrm{free}}} \frac{|\lambda_\pi(a)| \rho_{\mathcal{Q}}(a)}{a^2}  \ll_\varepsilon (\log D)^{17/18 + \varepsilon}.
\]
\end{lemma}

\begin{proof} 
We obtain using \eqref{den}
\begin{displaymath}
\begin{split}
& \sum_{\substack{a \leq D^{100} \\ a\textrm{ is } \square-\mathrm{free}}} \frac{|\lambda_\pi(a)| \rho_{\mathcal{Q}}(a)}{a^2}  \leq
\prod_{  \substack{p \leq D^{100} \\ p \nmid N}}\Big( 1+ \frac{|\lambda_\pi(p)|(1 + \chi_E(p))}{p}\Big).
\end{split}
\end{displaymath}
The Ramanujan bound and the Hecke relations imply $|\lambda_\pi(p)| \leq \frac{1}{2} + \frac{1}{2}\lambda_\pi(p^2) - \frac{1}{18}\lambda_\pi(p^2)^2$, so that
\[
\prod_{\substack{  p \leq D^{100} \\ p \nmid N}}\Big( 1+ \frac{|\lambda_\pi(p)|(1 + \chi_E(p))}{p}\Big) \leq \exp\Big(  \sum_{\substack{ p \leq D^{100}\\ p \nmid N} }  \frac{\frac{1}{2} + \frac{1}{2}\lambda_\pi(p^2) - \frac{1}{18}\lambda_\pi(p^2)^2 }{p} (1+ \chi_E(p))\Big).
\]
Finally, using \eqref{log-series} and letting $\theta_E = \textbf{1} \boxplus \chi_E$, we have 
\[
\sum_{\substack{ p \leq D^{100}\\ p \nmid N} }   \frac{\lambda_\pi(p^2)(1+ \chi_E(p) )  }{p} =   \log L(1, \text{sym}^2 \pi \times \theta_E) + O_\pi(1).
\]
Since $\text{sym}^2\pi \times \text{sym}^2\pi = \text{sym}^4\pi \boxplus\text{sym}^2\pi \boxplus\textbf{1}$, we obtain similarly
\[
\sum_{\substack{ p \leq D^{100}\\ p \nmid N} }   \frac{\lambda_\pi(p^2)^2(1+ \chi_E(p) )  }{p} = \log\log D +  \log  L(1, \text{sym}^4 \pi \times \theta_E) + \log  L(1, \text{sym}^2 \pi \times \theta_E) + O_\pi(1).
\]
Combining this with \eqref{titch1}, we obtain the stated bound.
\end{proof}

We bound the product over the local densities appearing in Lemma \ref{lemma:Kh} using the following result.

\begin{lemma}\label{lem:Kh-prod} Under GRH we have $\prod_{ 3 \leq p \leq D^{100}} \big(1 - \frac{\rho_{\mathcal{Q}}(p)}{p^2}\big) \ll_\varepsilon \log D^{-1+\varepsilon}$.
\end{lemma}

\begin{proof} 
By a similar, but simpler, reasoning to Lemma \ref{lem:Kh-sum}, we have 
\begin{displaymath}
\begin{split}
 \prod_{ 3 \leq p \leq D^{100}} \Big(1 - \frac{\rho_{\mathcal{Q}}(p)}{p^2}\Big) &=\prod_{ 3 \leq p \leq D^{100}} \Big(1 - \frac{p(1 + \chi_E(p)) - \chi_E(p)}{p^2}\Big)\\
 & \ll  \prod_{  p \leq D^{100}} \Big(1 - \frac{1 }{p}\Big) \prod_{   p \leq D^{100}} \Big(1 - \frac{\chi_E(p)}{p}\Big) \\
 &  \ll  \frac{1}{\log D} \exp\Big(-\sum_{p \leq D^{100}} \frac{\chi_E(p)}{p}\Big) \ll \frac{1}{L(1, \chi_E)\log D},
\end{split}
\end{displaymath}
by \eqref{den} and \eqref{log-series}. An appeal to \eqref{titch1} completes the proof. \end{proof}

Inserting Lemmata \ref{lem:Kh-sum} and \ref{lem:Kh-prod} into Lemma \ref{lemma:Kh} completes the proof of Proposition \ref{prop:Kh}.

\section{Oriented and aligned Heegner packets}\label{sec:oriented-packets}

The goal of Sections \ref{sec:oriented-packets} and \ref{sec:HHS} is to provide the necessary background on Heegner points and Hecke correspondences for the proof of Proposition \ref{prop-smallq}. The latter proves Theorem \ref{main-thm} for the small $q$ range as described in \S\ref{sec:bottom-q-sketch}. Specifically, in this section we clarify the parametrization of Heegner packets of \S\ref{sec:Hecke-explanation} --- without the Heegner condition and for general quaternion algebras --- thereby providing a more adelic perspective on the classical results of \cite{GKZ}.
 
\subsection{Fibering over Heegner packets of level 1}\label{sec:variation}
We begin by introducing the set of Heegner packets with level structure, and describing how they can be partitioned according to their level 1 projections.

To fix an embedding $\iota\colon E\to B$ is equivalent to fixing a rational torus $\bT<\bG$ with splitting field $E$. Hence a Heegner packet is parameterized by the datum $(\bT,g)$, where $\bT$ is the projective group of units of an embedding of an imaginary quadratic field $\iota\colon E\to B$ and $g\in\bG(\A)$ satisfies $g_\infty^{-1} \bT(\R) g_\infty=K_\infty$. Equivalently, $\bT$ is a maximal torus of $\bG$, defined over $\Q$ and anisotropic over $\R$.

Now fix $\bT$ and let $U=U_fK_\infty$, where $U_f<K_f$ is of finite index. We denote by $H_E(U)$ the collection of Heegner packets of the form $[\bT(\A)g]_U$ (due to Lemma \ref{lem:all-field-embeddings-conjugate} the collection $H_E(U)$ depends only on the field $E$, and does not depend on the particular embedding $\iota$.). Then $H_E(U)$ is parametrized by the class of $g$ in $\bT(\A) \backslash \bG(\A) \slash U$. Since the class of $g_\infty$ in $\bT(\R) \backslash \bG(\R) \slash K_\infty$ is fixed by the requirement $g_\infty^{-1} \bT(\R) g_\infty=K_\infty$, the Heegner packets in $H_E(U)$ are parametrized by the double quotient $H_E(U_f)=\bT(\A_f) \backslash \bG(\A_f) \slash U_f$. In particular, the set of Heegner packets of \emph{level $1$} is described by $H_E(K_f)=\bT(\A_f)\backslash \bG(\A_f) \slash K_f$. We shall provide an explicit description of the level $1$ packets in \S\ref{sec:level-one-packets}.

With $U_f<K_\infty$ of finite index, as before, we have
\begin{equation}\label{eq:Heegner-packets-double-quotient}
H_E(U_f)={\prod}'_p \bT(\Q_p) \backslash \bG(\Q_p) \slash U_p,
\end{equation}
the factorization into local components on the right-hand side being a restricted product, so that at all but finitely many primes the class is the trivial class of the identity. An immediate observation is that the parameter space \eqref{eq:Heegner-packets-double-quotient} is countable. 

We can now partition \eqref{eq:Heegner-packets-double-quotient} according to the level $1$ classes. We claim that
\begin{equation}\label{eq:Heegner-level-decomposition}
H_E(U_f)=\bigsqcup_{[g_f]\in H_E(K_f)} \bT(\A_f) \cdot  g_f \cdot  K_{\scD,f} \backslash K_f \slash U_f\;,
\end{equation}
where we recall from \S\ref{sec:special} that $K_{\scD,f}=\bT(\A_f) \cap g_f K_f g_f^{-1}$. Indeed each element of the union corresponds to a level $1$ Heegner packet, and all the Heegner packets in $[\bG(\A)]_U$ projecting to a fixed Heegner packet $[\bT(\A)g]_K$ are described by classes in $K_{\scD,f} \backslash K_f \slash U_f$. Notice that this double quotient is finite; hence there are finitely many Heegner packets in $[\bG(\A)]_U$ lying above a fixed level $1$ Heegner packet.

Because the associated global order $\mathscr{O}_\scD$ of a Heegner packet depends only on the level $1$ projection, we see that all the Heegner packets corresponding to a fixed union element in \eqref{eq:Heegner-level-decomposition} have the same associated order, and in particular the same discriminant $D$.

\subsection{Heegner Packets of level one}\label{sec:level-one-packets}

Let $K=K_f K_\infty$, where $K_f$ is almost maximal. In this paragraph we find which integers are discriminants of Heegner packets in $Y_K$ and how many packets correspond to a fixed discriminant. 

\begin{lemma}\label{lemma:adm-disc}
Fix $D\in\mathbb{N}$ which is not a perfect square. If there exists a packet of discriminant $D$ in $Y_K$ then $D=D_0 f^2$, with $D_0$ a fundamental discriminant, $(\frac{-D}{p})\neq 1$ for all $p\mid d_B$, and $\gcd(f,d_B)=1$. In this case we say $D$ is admissible. 
\end{lemma}
Note that admissibility is a congruence condition modulo $8 d_B^2$. Later we shall establish that if $D$ is admissible then a packet of discriminant $D$ always exists.

\begin{proof}
Observe first that $D$ determines the splitting field of the torus $E=\mathbb{Q}(\sqrt{-D})$. Assume a packet of discriminant $D$ exists; then $D$ is the discriminant of an order in $E$ and $D=D_0 f^2$, with $D_0=|\operatorname{disc}(E)|$ and $f$ the conductor of the order. Moreover, Section \ref{sec:disc} implies that the $p$-adic valuation of $D$ is at most $1$ for all $p\mid d_B$; hence $\gcd(f,d_B)=1$. The Albert--Brauer--Hasse--Noether theorem implies that an embedding $E\hookrightarrow B$ exists if and only if all primes dividing $p$ are not split in $E$, thus $\left(\frac{-D}{p}\right)\neq 1$ for all $p\mid d_B$, as claimed.
\end{proof}

We deduce that the collection of Heegner packets in $Y_K$ can be organized as follows:
\[
\bigsqcup_{\substack{E \textrm{ imaginary quadratic}\\ {\rm disc}(E)=-D_0 \textrm{ admissible}}} H_E(K).
\]
We now fix an admissible fundamental discriminant $D_0$ and let $E$ be the imaginary quadratic field with ${\rm disc}(E)=-D_0$. Then we may decompose
\[
H_E(K)=\bigsqcup_{\substack{f\geq 1\\ (f,d_B)=1}} H(K,D_0f^2),
\]
the union being disjoint by Lemma \ref{lemma:disjoint-packets}. We now compute the cardinality of $H(K,D_0f^2)$.

\begin{lemma}\label{lem:all-field-embeddings-conjugate}
Let $E/\mathbb{Q}$ be a quadratic field extension. Then all the embeddings of $E$ into $B$ are conjugate to each other by an element of $\bG(\Q)$. Similarly, if $E_p/\Q_p$ is a quadratic \'etale algebra then all embeddings of $E_p$ into $B_p$ are conjugate by an element of $\bG(\Q_p)$.
\end{lemma}
\begin{proof}
We verify the global claim for $E/\Q$. The proof of the local case follows mutatis mutandis; alternatively it can be deduced from the global case. Assume an embedding $E\hookrightarrow B$ exists, as otherwise the claim is trivially true.
Let $\mathbf{V}$ be the affine algebraic variety defined over $\Q$ representing the embeddings of $E$ into $B$, i.e.,
\[
\mathbf{V}(A)=\left\{\iota\colon E\otimes A \hookrightarrow B\otimes A \mid \iota \textrm{ is a ring embedding} \right\}
\]
for all $\Q$-algebras $A$. The group $\bB^\times$ acts on $\mathbf{V}$ and the action is scheme theoretically transitive, because over an algebraically closed field all split quadratic \'etale algebras are conjugate to the $2\times 2$ matrix algebra. In particular, $\bB^\times(\bar{\Q})$ acts transitively on $\mathbf{V}(\bar{\Q})$, the latter being the set of embeddings of the split \'etale algebra $\bar{\Q}^2$ into $\Mat(\bar{\Q})$. Fix an embedding $\iota_0\colon E\hookrightarrow B$ and let $\widetilde{\bT}<\bB^\times$ be the associated torus, it is the stabilizer of $\iota_0$ for the $\bB^\times$-action. All points in $\mathbf{V}(\Q)$ are in the $\bB^\times(\bar{\Q)}$-orbit of $\iota_0$, and we want to show they constitute a single $\bB^\times(\Q)$-orbit. This can be checked via Galois cohomology; that is, we need to show that $\ker\left[H^1(\widetilde{\bT},{\rm Gal}(\bar{\Q}/\Q))\to H^1(\bB^\times,{\rm Gal}(\bar{\Q}/\Q)) \right]$ is trivial. This kernel is trivial because $\widetilde{\bT}\simeq \operatorname{Res}_{E/\Q} \mathbb{G}_m$ and $H^1(\operatorname{Res}_{E/\Q} \mathbb{G}_m,{\rm Gal}(\bar{\Q}/\Q))$ is trivial by Hilbert's Satz 90.
\end{proof}

\begin{prop}\label{prop:number-of-packets}
For $D$ admissible, let $\omega$ be the number of primes $p\mid d_B$ such that $\left(\frac{-D}{p}\right)=-1$. Then there are exactly $2^\omega$ Heegner packets of discriminant $D$ in $Y_K$. For $p\mid d_B$ the group $K_p$ is an index two normal subgroup of $\bG(\Q_p)$ and the group 
\[
\mathcal{W}=\prod_{p\mid d_B} \bG(\Q_p)\slash K_p\simeq \left(\Z\slash 2 \Z\right)^{\omega(d_B)}
\]
acts on $Y_K$. In fact, $\mathcal{W}$ acts transitively on the $2^\omega$ packets of discriminant $D$. If one restricts to the product of $\bG(\Q_p)\slash K_p$ over primes ramified in $B$ and inert in $E$, then this action is simply transitive.
\end{prop}

Notice that in the case that $D$ is coprime to the discriminant of $B$, this statement is essentially contained in \cite[Lemma 2.5]{BD}, albeit in a different language.

\begin{remark}
The failure of uniqueness here arises because an almost maximal compact subgroup $K_f<\bG(\A_f)$ is not maximal at ramified primes. We can recover uniqueness of packets of discriminant $D$ if we replace $K_f$ by a maximal compact open subgroup of $\bG(\A_f)$, i.e.,\ replace $K_p$ by $\bG(\Q_p)$ for all primes $p\mid d_B$. The price for working with a maximal compact subgroup is not only the loss of some minor information from passing from $Y_K$ to the quotient $\mathcal{W}\backslash Y_K$, but also that the Picard group no longer acts freely on a packet. All the primes of $E$ above the primes dividing $d_B$ will act trivially.
\end{remark}

\begin{proof}
We first establish that a packet with discriminant $D$ does exists. Assume $D=D_0f^2$ is admissible, and let $\operatorname{disc}(E)=-D_0$. For $p\mid d_B$ the condition $(\frac{-D_0}{p})= (\frac{-D}{p})\neq 1$ implies that  $p$ does not split in $E$, and the Albert--Brauer--Hasse--Noether theorem implies that a ring embedding $\iota_0\colon E\hookrightarrow B$ exists. Denote by $E_p=E\otimes \Q_p$ for all primes $p$, and
let $\bT<\bG$ be the associated torus. Fix $g_\infty \in \bG(\R)$ such that $g_\infty^{-1}\bT(\R) g_\infty=K_\infty$.

Write $D=\prod_p D_p$. The discussion in \S \ref{sec:disc} implies that $\iota(E_p)\cap R_p=\scO_{E_p}$ for all $p\mid d_B$.  Note also that $|\operatorname{disc}(\iota(E_p)\cap R_p)|_p=1=|D_p|_p$ for almost all primes $p$. Hence the local discriminant of $[\bT(\A)g_\infty]_K$ coincides with $D_p$ for all primes $p\mid d_B$ and for almost all other primes.  To show a packet of discriminant $D$ exists it suffices to find for each exceptional prime $p$ an element $g_p\in \bG(\Q_p)$ such that $g_p^{-1}\iota(E_p)g_p\cap R_p$ has discriminant $D_p$. We can then set $g_p=e$ for all non exceptional primes, and the packet $[\bT(\A)g]_K$ with $g=(g_\infty,g_2,g_3,\ldots)\in \bG(\A)$ has discriminant $D$.

Fix an exceptional prime $p$. Lemma \ref{lem:all-field-embeddings-conjugate} implies that it is enough to find an arbitrary embedding  $\jmath\colon E_p\hookrightarrow \Mat(\Q_p)$ satisfying $|\operatorname{disc}(\jmath(E_p)\cap \Mat(\Z_p)|_p=|D_p|_p$. Let $\scO\subset E_p$ be an order of discriminant $D_p$; we claim an embedding $\jmath\colon E_p\hookrightarrow\Mat(\Q_p)$ exists with $\jmath(E_p)\cap R_p=\jmath(\scO)$. To construct such an embedding take a $\Z_p$-basis $\lambda_1,\lambda_2$ of $\scO$, and define $\jmath$ using $x\cdot \lambda_i=\sum_j \jmath(x)_{i,j}\lambda_j$, i.e.,\ $\jmath(x)$ is the matrix for the multiplication by $x$ operator in the basis $\lambda_1,\lambda_2$. Then $\jmath^{-1}(R_p)=\{x\in E_p \mid x \scO\subset \scO\}=\scO$. Thus we have proved that every admissible $D$ is the discriminant of a Heegner packet in $Y_K$ as claimed.

We continue to count the number of distinct packets with a fixed admissible discriminant $D$. Because the discriminant fixes the field $E$, Lemma \ref{lem:all-field-embeddings-conjugate} implies that all homogeneous Heegner sets of discriminant $D$ can be written as $[\bT(\A)g]$ for some $g\in \bG(\A)$ and $\bT$  fixed. The archimedean double coset $\bT(\R) g_\infty K_\infty$ is determined uniquely by the requirement $g_\infty^{-1} \bT(\R) g_\infty=K_\infty$. Because the discriminant $D$ is a product of local invariants $D_p$, the discussion in \S\ref{sec:variation} reduces the claim to the local question how many classes of $g_p$ in $\bT(\Q_p)\backslash \bG(\Q_p) \slash K_p$ have discriminant $p$, for all primes $p$. That is, we need to find all $K_p$-conjugacy classes of embeddings $\iota_p\colon E_p\hookrightarrow B_p$) such that $\iota_p(E_p)\cap R_p$ has discriminant $D_p$. 

If $B_p$ is split, then we can consider equivalently embeddings of $E_p$ into $\Mat(\Q_p)$. The proof of \cite[Corollary 4.4]{ELMV1}, adapted to $\Q_p$, implies that $\mathbf{GL}_2(\Z_p)$-conjugacy classes of embedding $\iota_p\colon E_p\hookrightarrow B_p$ are in bijection with $E_p^\times$-homothety classes of $\Z_p$ lattices $L\subset E_p$. This bijection satisfies $\iota_p^{-1}(R_p)=\scO(L)=\left\{x \in E_p \mid xL\subset L\right\}$, and $D_p^{-1}=|\operatorname{disc} \scO(L)|_p$. Because $E_p$ is a quadratic $\Q_p$ \'etale-algebra all its orders are monogenic and of the form $\mathbb{Z}_p+f\scO_{E_p}$ with $f\in\mathbb{Z}_p$ the conductor. The discriminant is $|f|_p^{-1}\operatorname{disc}(\scO_{E_p})$ and it uniquely determines the local order. Hence, in order to show that all embedding with a fixed discriminant $D_p$ are in the same $\mathbf{GL}_2(\Z_p)$-conjugacy class, it suffices to prove that all $\Z_p$-lattices $L\subset E_p$ satisfying $\scO(L)=\scO_p$, for some fixed order $\scO_p\subset E_p$, are in the same $E_p^\times$-homothety class. This is true, because $\scO_P$ is monogenic, and every proper fractional $\scO_p$-ideal is principal, cf.\  \cite[Proposition 2.1]{ELMV3}.

If $B_p$ is a division algebra then it has a unique maximal order --- the subring of integral elements. Thus $K_p$ is a normal subgroup of $\bG(\Q_p)$, and $\bG(\Q_p)\slash K_p\simeq \Z\slash 2\Z$ because the ramification degree of $B_p$ is $2$. It is the subgroup of elements that have a representative with even valuation. If $E_p\slash \Q_p$ is ramified, then $\iota_p(E_p)$ necessarily contains a uniformizer for $B_p$ and $\bT(\Q_p)\backslash \bG(\Q_p) \slash K_p$ is trivial. Otherwise if $E_p$ is inert then $\bT(\Q_p)\subset K_p$ and $\bT(\Q_p)\backslash \bG(\Q_p) \slash K_p=\bG(\Q_p) \slash K_p$ contains exactly two elements. These correspond to two embedding $\bT(\Q_p)\hookrightarrow B_p$ that are not $K_p$ conjugate, yet they have the same discriminant (all embeddings have the same discriminant at a ramified place). Nevertheless the right action of $\bG(\Q_p) \slash K_p$ on $Y_K$ does switch the Heegner packets corresponding to these embeddings.
\end{proof}

\subsection{Alignment}\label{subsec:alignment}

We now pass to a more general level structure $U$, and seek to understand the structure of the level $U$ Heegner packets lying over a given level $1$ Heegner packet $[\bT(\A)g]_K$ in the decomposition \eqref{eq:Heegner-level-decomposition}. In particular, we would like to ensure conditions for which we can \textit{lift} a level 1 Heegner packet to a level $U$ packet of the same size. We address this question in this and the next subsection.

We begin by observing that the discriminant of the Heegner packet $[\bT(\A)g]_U$, as defined in \S\ref{sec:disc}, does not necessarily coincide with other possible definitions of a discriminant if $U$ is not an almost maximal compact subgroup. For example, if $U_\scD=\bT(\A)\cap gUg^{-1}$ is a proper subgroup of $K_\scD=\bT(\A) \cap g K g^{-1}$, then the Heegner packet $[\bT(\A) g]_U$, which is a principal homogeneous space for $\bT(\Q)\backslash \bT(\A)\slash U_\scD$, is a bigger set than $[\bT(\A)g]_K$, which is itself a principal homogeneous space for $\bT(\Q)\backslash \bT(\A)\slash K_\scD$. It is even possible for $U_\scD$ to be a projectivized group of units of some proper suborder of $\scO_\scD$.

\begin{remark}
Note that \cite{ELMV2} defines a notion of \textit{volume} of a homogeneous toral set, for which they prove an analytic class number formula in terms of the discriminant. This volume is comparable to the size of the Heegner packet $[\bT(\A)g]_K$, but the size of the Heegner packet $[\bT(\A)g]_U$ for $U\lneq K$ is typically larger.
\end{remark}

To remedy this problem, we make the following definition.

\begin{defn}\label{defn:aligned}
We say that $[\bT(\A)g]_U$ is an \emph{aligned} Heegner packet if
\[
\bT(\A)\cap gUg^{-1}=\bT(\A)\cap gKg^{-1}.
\]
Similarly, $[\bT(\A)g]_U$ is \textit{aligned at a prime $p$} if $\bT(\Q_p) \cap g_p U_p g_p^{-1}=\bT(\Q_p) \cap g_p K_p g_p^{-1} $. 
\end{defn}

We make a few immediate observations:
\begin{enumerate}
\item a packet $[\bT(\A)g]_U$ is aligned if, and only if, it is aligned at all primes $p$;
\item in view of the motivating discussion, as well as \eqref{eq:iso-Pic}, a packet $[\bT(\A)g]_U$ is aligned precisely when it is a principal homogeneous space for $
{\rm Pic}(\scO_\scD)$;
\item for an aligned Heegner packet $[\bT(\A)g]_U$, the natural projection $[\bT(\A)g]_U\rightarrow [\bT(\A)g]_K=H_\scD$ is a bijection.
\end{enumerate}
The last property shows that aligned packets can be viewed as trivial lifts of level 1 packets.

With $U<K$, $D=D_0f^2$ admissible as in Lemma \ref{lemma:adm-disc}, and $E$ satisfying ${\rm disc}(E)=-D_0$, it then makes sense to define $H(U,D)$ as the subset of $H_E(U)$ consisting of \textit{aligned} Heegner packets of level $U$. This will then capture the classical definition of the set of Heegner points of level $N$ and discriminant $D$ when $B=\Mat(\Q)$ is the split algebra and $U_f=K_0(N)$.

\subsection{Lifting level 1 Heegner packets}

We give ourselves now a $g\in\bG(\A)$ whose class in $\bT(\A)\backslash\bG(\A)/K$ defines a level 1 Heegner packet. We will be interested in lifting $[\bT(\A)g]_K$ to a level $U$ Heegner packet, where
\begin{equation}\label{eq:U-cap-conj}
U=K\cap g^{-1}s gK g^{-1}s^{-1} g
\end{equation}
for some $s\in\bT(\A_f)$. To do so, we consider, as in \S\ref{sec:variation}, the set $H_E(U)$ of Heegner packets associated with $E$ and of level structure $U$; we wish to determine which packets in the decomposition \eqref{eq:Heegner-level-decomposition} are aligned. This requires that we compute the double quotient
\[
K_{\scD,f} \backslash K_f \slash U_f=\prod_p K_{\scD,p} \backslash K_p \slash U_p\qquad (K_{\scD,p}=\bT(\Q_p) \cap g_p K_p g_p^{-1}),
\]
and in view of the above factorization, it suffices to do so locally.

Note that the trivial class in $K_{\scD,f} \backslash K_f \slash U_f$ already gives an aligned level $U$ Heegner packet. Indeed, since $\bT(\A)$ is abelian, we have $\bT(\A)\cap gKg^{-1}=\bT(\A)\cap sgKg^{-1}s^{-1}$ for all $s\in\bT(\A)$. This implies that $[\bT(\A)g]_U$ is aligned when $U$ is of the form \eqref{eq:U-cap-conj}. It remains to examine the packets associated to the non-trivial classes.

In what follows, we shall assume, for simplicity, that the global order $\scO_\scD$ from \S\ref{sec:disc} is the maximal order in $E$. We shall find that there are both aligned and non-aligned packets in $H_E(U)$ lying over $[\bT(\A)g]_K$. The exact statement is as follows.

\begin{lemma}\label{lemma:alignment}
Let $\scD=(\iota, g)$ be a homogeneous Heegner datum for which $\scO_\scD=\scO_E$. Let $U=\prod_p U_p$ be as in \eqref{eq:U-cap-conj} for some $s=\prod_p s_p\in\bT(\A_f)$. Let $p$ be a prime.
\begin{enumerate}
\item If $p$ ramifies in $B$ or is inert in $E$, then $K_{\scD,p} \backslash K_p \slash U_p$ is a singleton, and the packet $[\bT(\A)g]_U$ is aligned at $p$;

\item If $p$ splits in $B$ and $E$, and $s_p\not\in K_{\scD}$, then $[K_p\colon U_p]=(p+1)p^{r-1}$ for some $r>0$. There are $2r-1$ lifts of the Heegner packet $[\bT(\A)g]_K$ to a level $U$ Heegner packet. Of these, exactly two correspond to aligned packets, namely, $[\bT(\A) g]_U$ and $[\bT(\A) g w_p]_U$, where $w_p \in K_p\smallsetminus g_p^{-1}\bT(\Q_p)g_p$ normalizes $g_p^{-1}\bT(\Q_p)g_p$.
\end{enumerate}
\end{lemma}

\begin{proof}
If $p$ ramifies in $B$, then $K_p$ is a normal subgroup so that $U_p=K_p$. If $p$ is inert then $g_p^{-1}\bT(\Q_p)g_p\subset K_p$ and $g_p^{-1}s_pg_p\in K_p$, then again $U_p=K_p$ by \eqref{eq:U-cap-conj}. In both cases, $K_{\scD,p}\backslash K_p \slash U_p$ is a singleton, proving the claim.

Now suppose that $p$ splits in both $B$ and $E$. Then we can fix  an isomorphism $\tilde{\bG}(\Q_p)\xrightarrow{\sim} \mathbf{GL}_2(\Q_p)$ sending $g_p^{-1}R_pg_p$ to $\Mat(\Z_p)$ and sending $g_p^{-1}\bT(\Q_p)g_p$ to the diagonal subgroup. Then\footnote{The last bijection can be established by letting $\mathbf{GL}_2(\Z_p)$ act on $\mathbb{P}^1\left(\Z/ p^r \Z\right)$ via the reduction modulo $p^r$ map. Then $K_0(p^r)$ is the stabilizer of $[1,0]$.} 
$K_p \slash U_p \simeq \mathbf{GL}_2(\Z_p) \slash K_0(p^r)\simeq  \mathbb{P}^1\left(\Z/ p^r \Z\right)$, where $K_0(p^r)$ is the subgroup of matrices in $\mathbf{GL}_2(\Z_p)$ congruent to upper diagonal matrices mod $p^r$ for a fixed $r$ depending on $s_p$. Moreover, $r>0$ because $s_p\not\in K_{\scD}$. Thus $K_{\scD,p}\backslash K_p \slash U_p$ can be identified with
\[
\operatorname{diag}\left(\F_p^\times,\F_p^\times \right)\backslash \mathbb{P}^1\left(\Z/ p^r \Z\right)\simeq \{[1,1]\}\sqcup \left\{[p^i,1] \colon 0<i\leq r \right\}\sqcup \left\{[1,p^i] \colon 0<i\leq r \right\}.
\]
Of these $2r-1$ classes, exactly two are aligned, namely, those corresponding to the vectors $[0,1]=[p^r,1]\in \mathbb{P}^1(\Z/p^r \Z)$ and $[1,0]=[1,p^r]\in \mathbb{P}^1(\Z/p^r \Z)$. These are represented by the identity element $e\in K_p$ and by an element $w_p \in K_p\smallsetminus g_p^{-1}\bT(\Q_p)g_p$ that normalizes $g_p^{-1}\bT(\Q_p)g_p$.
\end{proof}

We deduce that the aligned Heegner packets correspond to classical Heegner points of discriminant $-D$. The other classes are not aligned; although they project onto $[\bT(\A)g]_K$ they carry an action by a bigger ray class group. The reader is invited to compare the second part of Lemma \ref{lemma:alignment} with \cite[Proposition 2.5]{EMV}.

\subsection{Change of orientation}\label{sec:orientation}

Fix a homogeneous Heegner set $[\bT(\A)g]$. For each prime $p$ there exists an element $w_p\in g_p^{-1}\left(N_\bG(\bT)(\Q_p)- \bT(\Q_p)\right)g_p\cap K_p$.  This can be seen, for example, from the explicit adapted coordinate systems in \cite[\S5]{Kh}. Hence, for every finite place $p$ we have the operation of change of orientation at $p$ given by $[\bT(\A)g]\mapsto [\bT(\A)gw_p]$. Because the Weyl group of $\bT(\Q_p)$ is isomorphic to $\Z\slash 2 \Z$, the change of orientation at $p$ is an involution.
Changes of orientation at different places are commutative transformations. Because $w_p$ normalizes the torus $g^{-1}\bT(\A)g$, two homogeneous Heegner sets that differ by a change of orientation are invariant under the \emph{same} torus subgroup of $\bG(\A)$.

Let now $U<K$ be of finite index and consider the change of orientation of a Heegner packet $[\bT(\A)g]_U$. If $U=K$, then the change of orientation is a trivial operation on Heegner packets, because $w_p\in K_p$ for all primes $p$. Moreover, the change of orientation at $p$ is a non-trivial transformation of $[\bT(\A)g]_U$ only if the class of $w_p$ in $K_{\scD,p}\backslash K_p \slash U_p$ is not the trivial identity class.

We conclude from the above discussion that there are finitely many different Heegner packets one can obtain from $[\bT(\A)g]_U$ by change of orientation. We may organize them as follows. Let $N_\scD$ denote the positive square-free integer given by the product of all primes $p$ which are split in $B$ and $E$ and for which the class of the change of orientation element $w_p$ is non-trivial in the double quotient $K_{\scD,p}\backslash K_p \slash U_p$. For any divisor $M\mid N_\scD$ let $w_M$ denote the element $\prod_{p\mid M}w_p$ in $\prod_{p\mid M}K_p$.

With the notation and terminology we have set up over the course of this section, we can now state the following conclusion.

\begin{cor}
Fix $g\in\bG(\A)$. Let $D=D_0f^2$ be admissible, as in Lemma \ref{lemma:adm-disc}, and let the imaginary quadratic field $E$ satisfy ${\rm disc}(E)=-D_0$. Let $U<K$ be of the form \eqref{eq:U-cap-conj}, for some $s\in\bT(\A_f)$.  Recall the set $H(U,D)$ from \S\ref{subsec:alignment} consisting of aligned Heegner packets of level $U$ and discriminant $D$. Then
\[
H(U,D)=\bigsqcup_{M\mid N_\scD} [\bT(\A)gw_M]_U.
\]
\end{cor}

The above result nearly captures (and extends) the classical statement \eqref{eq:HE-decomp}, except for the fact that the Heegner points there were defined on the modular correspondence $Y^\Delta_0(N)$ rather than $Y_0(N)$. This aspect will be explored in the next section.

\section{Homogeneous Hecke sets and Hecke correspondences}\label{sec:HHS}
Denote by $\Delta\bG$ the image of the diagonal embedding $\bG\hookrightarrow\bG\times\bG$. Fix now $h_f\in\bG(\A_f)\subset\bG(\A)$ and consider the set $[\Delta\bG(\A)(e,h_f)]\subset[(\bG\times \bG) (\A)]$. Following \cite{Kh} we call it a \emph{homogeneous Hecke set}. This set is a homogeneous space for $(e,h_f)^{-1}\Delta\bG(\A)(e,h_f)$ and carries a unique Borel probability measure invariant under this action.

\subsection{Level structures}\label{sec:Hecke-corresp--level}
Let $U<K$ be a finite index subgroup. Define
\begin{equation}\label{def:Eichler-congruence-sbgps}
U(h_f)=U\cap h_f U h_f^{-1} < \bG(\A)\quad\textrm{ and }\quad U_f(h_f)=U_f\cap h_f U_f h_f^{-1} < \bG(\A_f).
\end{equation}
Note that $U(h_f)=K_\infty U_f(h_f)$.   If $U=K$ is almost  maximal then $U$ is the projectivized completion of the group $R^\times$ of units of $R$. We argue that in this case $U(h_f)$ is the projectivized completion of the group of units of an Eichler order. In particular, if $\bG$ is split and $U=K$, then $U(h_f)$ is conjugate to $K_0(q)$ --- the completion of $\Gamma_0(q)$ --- for some $q\in\mathbb{N}$.
Set $R^{h_f}=h_f \widehat{R} h_f^{-1}\cap B$; then $R^{h_f}$ is also a maximal order, because it is everywhere locally maximal. Moreover, $R^{h_f}$ is a maximal order in the same genus as $R$. Let $R^\sharp=R\cap R^{h_f}$. It is an Eichler order, i.e.,\ the intersection of two maximal orders. Then $U(h_f)$ is the projectivized completion of ${R^\sharp}^\times$. 

Consider the commutative diagram
\medskip
\begin{center}
\begin{tikzcd}[row sep=0.8cm, column sep=1.5cm]
&{[\bG(\A)]}            \arrow[d] \arrow[r,"\beta"] &{[\Delta\bG(\A)(e,h_f)]} \arrow[d]\\
&{[\bG(\A)]}_{U_f(h_f)} \arrow[d] \arrow[r,"\beta_{U_f}"] &{[\Delta\bG(\A)(e,h_f)]_{U_f\times U_f}} \arrow[d]\\
&{[\bG(\A)]}_{U(h_f)}             \arrow[r,"\beta_U"] &{[\Delta\bG(\A)(e,h_f)]}_{U\times U}
\end{tikzcd}
\end{center}
\medskip
The vertical arrows are the surjective quotient maps and the horizontal arrows $\beta, \beta_{U_f}, \beta_U$ are induced by the map $\Delta\bG(\A)\hookrightarrow(\bG\times \bG)(\A)$, defined by $g\mapsto (g,g h_f)$.

\begin{lemma}\label{lem:birational}
The following properties of the above diagram hold:
\begin{enumerate}
\item The maps $\beta,\beta_{U_f}$ are homeomorphisms;
\item\label{non-inj} The map $\beta_U$ is surjective but not necessarily injective. Specifically, $\beta_U([g])=\beta_U([g'])$ if, and only if, $[g]$ and $[g']$  belong to Heegner packets $[\bT(\A)g]_{U(h_f)}$ and $[\bT(\A)g']_{U(h_f)}$ respectively, such that $g'=gu$, with $u\in U\cap h_f U h_f^{-1} \cdot g^{-1}\bT(\Q)g $.
\end{enumerate}
\end{lemma}
Note that there are countable many points in $Y_{U(h_f)}$ that belong to some Heegner packet. The map $\beta_U$ is a bijection onto its image when restricted to $Y_{U(h_f)}- \bigcup_{\scD} H_\scD$, and the union runs over all Heegner data.

\begin{example*}
As an example of non-injectivity in \eqref{non-inj}, consider the case when $\bG=\mathbf{PGL}_2$ is split and $U=K$ is maximal. Fix a prime $p$ and take $h_f=h_p\in \bG(\Q_p)$ with Cartan decomposition $h_p\in K_p  \left(\begin{smallmatrix} p^i & 0 \\ 0 & 1
\end{smallmatrix}\right) K_p$. Then $\beta_K$ is the map $Y_0(p^i)\to Y_0(1)\times Y_0(1)$ induced by the Hecke correspondence $T_{p^i}$. To construct Heegner points where the map $\beta_K$ is not injective, take an imaginary quadratic field $E/\mathbb{Q}$, where $p=\mathfrak{p}\bar{\mathfrak{p}}$ is split, and the ideal class of $\mathfrak{p}^{2i}$ is trivial. 

Let $\iota\colon E \to B$ be an optimal embedding, i.e.,\ $E\cap \iota^{-1}(R)=\scO_E$. The optimality assumption implies that  $\bT(\Q_p)$ is conjugate to the diagonal subgroup of $\mathbf{PGL}_2(\Q_p)$ by an element of $K_p$, hence there is a decomposition $h_p=k_1 s_p k_2$ with $k_1,k_2\in K_p$ and $s_p\in\bT(\Q_p)$.
Fix $g_\infty$ so that $g_\infty^{-1} \bT(\R) g_\infty =K_\infty$ and let $g_f=k_1^{-1}$. This defines a Heegner datum $\scD=(\iota, g)$.

Let $w_p\in K_p$ be the orientation-switching element from \S\ref{sec:orientation}. Then $\beta_K\left([\bT(\A)g]_{U(h_f)}\right)=\beta_K\left([\bT(\A)g w_p]_{U(h_f)}\right)$, i.e.,\ we have two packets of Heegner points with different orientations mapping to a single joint packet in $Y_U\times Y_U$. That these packets have the same image is a direct computation that uses an element $\gamma\in \bT(\Q)$ satisfying $\gamma\in s_p^2 K_\scD$. Such an element exists because $\mathfrak{p}^{2i}$ is equivalent to a principal ideal. 
\end{example*}

\begin{proof}
The top map $\beta$ is obviously a continuous bijection, hence each horizontal map is surjective.
For the injectivity of $\beta_{U_f}$, it suffices to show that if $(\gamma_1 g, \gamma_2 g)=(g' u_f,g' h_f u_f' h_f^{-1})$ then $\gamma_1=\gamma_2$. This follows by looking at the archimedean component $(\gamma_1 g_\infty,\gamma_2 g_\infty)=(g'_\infty,g'_\infty)$.  The maps $\beta$, $\beta_{U_f}$ also have continuous inverses. This follows by the inverse function theorem for $\beta_{U_f}$ because it is a smooth diffeomorphism between two real manifolds with an everywhere non-vanishing Jacobian. The same follows for $\beta$ by considering $\beta_{U_f}$ for arbitrarily small open subgroups $U_f<K_f$.

To verify part \eqref{non-inj}, let $[g],[g']\in[\bG(\A)]_{U(h_f)}$ be distinct elements mapping to the same point. Then we can choose representatives $g$, $g'$, such that there are $e,\gamma\in\bG(\Q)$, $e\neq \gamma$, and $u,u'\in U$, satisfying $g=g'u^{-1}$, $\gamma g=g'h_f u' h_f^{-1}$. Equivalently, $g^{-1} \gamma g=u h_f u' h_f^{-1}$. At the archimedean place we deduce $g_\infty^{-1} \gamma g_\infty\in K_\infty$. In particular, $\gamma$ is regular semisimple and the centralizer $Z_\bG(\gamma)<\bG$ is a maximal torus defined over $\Q$ and satisfies $g_\infty^{-1}Z_\bG(\gamma) g_\infty=K_\infty$. The condition $g^{-1} \gamma g=u h_f u' h_f^{-1}$ is satisfied even if we replace $g$ by any element of $Z_\bG(\gamma)(\A)g$. Hence the whole Heegner packet $[Z_\bG(\gamma)(\A) g]_{U(h_f)}$ belongs to the non-injective locus. In the other direction, assume that there are distinct $e,\gamma\in \bT(\Q)$ with $\gamma=g u h_f u' h_f^{-1} g^{-1}$, $u,u'\in U$. Then it is a direct computation to check that $\beta_U([t g])=\beta_U([t g u])$ for all $t\in\bT(\A)$.
\end{proof}

We complete this subsection by observing that if $h_f,h'_f\in\bG(\A_f)$ with $h'_f=u_1 h_f u_2$ for some $u_1,u_2\in U_f$, then 
\begin{equation}\label{eq:same-Hecke-corr}
[\Delta\bG(\A)(e,h_f)(\A)]_{U\times U}=[\Delta\bG(\A)(e,h'_f)(\A)]_{U\times U}
\end{equation}
and the map $g\mapsto g u_1^{-1}$ induces an isomorphism $[\bG(\A)]_{U(h_f)}\to [\bG(\A)]_{U(h'_f)}$. 

\subsection{Intermediate Hecke sets}\label{sec:int-Hecke-sets}
Fix a Heegner datum $\scD$ and a shift $s\in\bT(\A_f)$. We have a commutative diagram
\medskip
\begin{center}
	\begin{tikzcd}
		&{[\bT(\A)g]_{U(s)}}\arrow[r]\arrow[d, "\eqref{eq:alphaU}"]  &{[\bG(\A)]_{U(s)}}\arrow[d,"\beta_U"] \\
		&{H^\Delta_\scD(s)} \arrow[r] &{[\Delta\bG(\A)(e,s)]}_{U\times U}.
	\end{tikzcd}
\end{center}
The leftmost map \eqref{eq:alphaU} is a measure-preserving homeomorphism. This diagram is fundamental to our approach in the small $q$ range, as it allows us to represent the joint Heegner period, defined in \S\ref{sec:joint-Heegner-packet}, as a Heegner period, defined in \S\ref{sec:special}, with a different level structure. Specifically, for $\varphi \in C(Y_U\times Y_U)$ we have
\begin{equation}\label{eq:joint-period-to-reg}
P^\Delta_{\scD}(\varphi;s)=P_\scD(\beta_U \circ \varphi\restriction_{{[\Delta\bG(\A)(e,s)]}_{U\times U}} )\;.
\end{equation}

It is important to note that although $t_\Q s$ and $s$ give rise to the same joint Heegner packet for all $t_\Q\in\bT(\Q)$, the intermediate Hecke sets ${[\Delta\bG(\A)(e,s)]}_{U\times U}$ and ${[\Delta\bG(\A)(e,t_\Q s)]}_{U\times U}$ are distinct and so are $U(s)$ and $U(t_\Q s)$. The latter groups are in general not conjugate and represent different level structures. If $U=K$ is almost maximal, let $\mathfrak{s}, \mathfrak{s'}\subset\scO_\scD$ be the primitive integral ideals associated to $s$ and  $s'=t_\Q s$ (see \S \ref{sec:adelic-shift}). Then $K(s)$ and $K(s')$ are the projectivized group of units of an Eichler order of level $\rm N\mathfrak{s}$ and $\rm N\mathfrak{s'}$ respectively. If moreover $\bG=\mathbf{PGL}_2$ is split then ${[\mathbf{PGL}_2(\A)]_{K(s)}}\simeq Y_0(\rm N \mathfrak{s})$ and ${[\mathbf{PGL}_2(\A)]_{K(s')}}\simeq Y_0(\rm N \mathfrak{s'})$.

This observation explains why we need to consider the minimal norm of any integral ideal representing $[s_i]\in\operatorname{Pic} (\scO_{\scD_i})$. Indeed, every such ideal gives rise to an intermediate Hecke set; and if $q_i\not\to \infty$ then we can find a subsequence of Heegner packets which are all contained in the same fixed intermediate Hecke set.

\subsection{Shifts with the same norm}\label{sec521}
In this subsection we assume $U=K$ is almost maximal. Our goal is to analyze how joint Heegner packets with shifts by integral ideals of the same norm $q$ embed in the graph of Hecke correspondence of level $q$. Fix a Heegner datum $\scD$ and let $\bT<\bG$ be the associated torus. Assume $\scO_\scD=\scO_E$ is maximal. Let $s,s'\in\bT(\A_f)$ both correspond to primitive integral ideals $\mathfrak{s}$, $\mathfrak{s}'$ of the same norm $N\in\mathbb{N}$, and assume that $\mathfrak{s}$ and $\mathfrak{s}'$ are primitive integral representatives in their respective ideal classes. We can then write $\mathfrak{s}=\gcd(\mathfrak{s},\mathfrak{s}') \prod_{i=1}^n \mathfrak{p}_i^{k_i}$ and $\mathfrak{s}'=\gcd(\mathfrak{s},\mathfrak{s}') \prod_{i=1}^n \overline{\mathfrak{p}}_i^{k_i}$, where each prime $\mathfrak{p}_i$  sits above a split prime $p_i$. Let $s_0\in\bT(\A_f)$ represent $\gcd(\mathfrak{s},\mathfrak{s}')$ and similarly let $s_i$ represent $\mathfrak{p}_i^{k_i}$. Then 
\[
s K_{\scD,f}= s_0 \prod_{i=1}^n s_i \cdot K_{\scD,f},\quad\textrm{ and }\quad s' K_{\scD,f}= s_0 \prod_{i=1}^nw_{p_i} s_i w_{p_i}^{-1} \cdot K_{\scD,f}.
\]
This holds because the Galois action on $\bT(\Q_p)$ coincides with the action of the Weyl group. Denote $w_s=\prod_{1\leq i \leq n} w_{p_i}\in \bG(\A_f)$, then $s'\in w_s s w_s^{-1} U_{\scD,f}$. Expanding on the diagram from \S\ref{sec:int-Hecke-sets} above, we have the following commutative diagram
\medskip
\begin{center}
		\begin{tikzcd}[row sep=1.3cm, column sep=0.2cm]
		&{[\bT(\A)g]_{K(s)}}\arrow[r]\arrow[d, "\eqref{eq:alphaU}"]  &{[\bG(\A)]_{K(s)}}\arrow[d,"\beta_K"] &[1.75cm] {[\bG(\A)]_{K(s')}}\arrow[l,"{[g w_s]}\mapsfrom{[g]}"] \arrow[d,"\beta_K'"]& {[\bT(\A)g]_{K(s')}}\arrow[l] \arrow[d, "\eqref{eq:alphaU}"]\\
		&{H^\Delta_\scD(s)} \arrow[r] &{[\Delta\bG(\A)(e,s)]}_{K\times K}  & {[\Delta\bG(\A)(e,s')]_{K\times K}} \arrow[l,,"{=}"]& 
		{H^\Delta_\scD(s')}, \arrow[l]
	\end{tikzcd}
\end{center}
\medskip
where the bottom middle map $[(g,gs')]\mapsto [(g w_s, g w_s s)]$ is the equality map, since
\[
[(g w_s, g w_s s)] \equiv [(g,g w_s s w_s^{-1})] \equiv [(g, g s)] \bmod K\times K.
\]
In particular, we have two embeddings of $[\bT(\A)g ]_{K_\scD}$ into $[\bG(\A)]_{K(s)}$, one from the packet shifted by $s$ and another from the packet shifted by $s'$. Chasing the maps in the diagram above it is easy to see that these two embeddings are Heegner packets obtained from each other by a change of orientation exactly at the places $\{p_i\}_i$. In this way, fixing a Heegner datum $\scD$ and a norm $N$, we obtain $2^k$ total Heegner packets in $[\bG(\A)]_{K(s)}$, corresponding to orientation change of the original datum $\scD$, where $k$ is the number of primes $p\mid N$ that split in $E$.

\subsection{Hecke correspondences}

Corresponding to the finite index subgroup $U_f\subset K_f$, and an element $h_f\in \bG(\A_f)$, we have two degeneracy maps
\[
[\bG(\A)]_U \xleftarrow{\;\; \delta_1\;\; }  [\bG(\A)]_{U(h_f)} \xrightarrow{\;\; \delta_2\;\; } [\bG(\A)]_U
\]
defined as follows. For an element $g\in\bG(\A)$ let us write $[g]_{U(h_f)}$ and $[g]_U$ for the class $\bG(\Q) g U(h_f)\in [\bG(\A)]_{U(h_f)}$ and $\bG(\Q) g U\in [\bG(\A)]_U$, respectively. Then we put
\begin{align*}
&\delta_1: [\bG(\A)]_{U(h_f)}\rightarrow [\bG(\A)]_U,  \qquad [g]_{U(h_f)}\mapsto [g]_U,\\
&\delta_2: [\bG(\A)]_{U(h_f)}\rightarrow [\bG(\A)]_U,\qquad [g]_{U(h_f)}\mapsto [gh_f]_U,
\end{align*}
the latter being well-defined thanks to the inclusion $h_f^{-1}U(h_f)h_f\subset U$. Both $\delta_1$ and $\delta_2$ are degree $\deg(h_f)=[U_f(h_f): U_f]$ coverings of $[\bG(\A)]_U$. 

\begin{remark}\label{rem:deg}
When $h_f=s\in\bT(\A_f)$ and $U=K$ is almost maximal, the index $\deg(s)$ is easily calculated; cf.\ \cite[Lemma 7.6]{Kh}; in particular, if $\sk\in\mathcal{I}(\mathscr{O}_\scD)$ denotes the unique primitive integral ideal representing the image of $s$ in $\Q^\times\backslash\mathcal{I}(\mathscr{O}_\scD)$ under the map \eqref{eq:s-to-ideal}, then $\deg(s)=\prod_{p^n\| N} p^{n-1}(p+1)=N\prod_{p\mid N}(1+p^{-1})$, where $N=\Nr\mathfrak{s}$.
\end{remark}

The Hecke correspondence at $h_f$ is given by ${\rm Div}([\bG(\A)]_U)\rightarrow {\rm Div}([\bG(\A)]_U)$, $x\mapsto \delta_2(\delta_1^{-1}(x))$. Its graph is defined to be
\begin{equation}\label{eq:first-def-Hecke-corr}
Y_U^\Delta(h_f)=\left\{(x,y)\in [(\bG\times\bG)(\A)]_{U\times U}: y\in \delta_2(\delta_1^{-1}(x))\right\}.
\end{equation}
\begin{lemma}\label{lemma:Hecke-graph}
We have $Y_U^\Delta(h_f)=[\Delta\bG(\A)(e,h_f)]_{U\times U}$.
\end{lemma}
\begin{proof}
By definition, $Y_U^\Delta(h_f)$ is the image of the map
\[
[\bG(\A)]_{U(h_f)}\rightarrow [(\bG\times \bG)(\A)]_{U\times U}, \qquad x\mapsto (\delta_1(x),\delta_2(x)).
\]
This is precisely the definition of $\beta_U$ from \S \ref{sec:Hecke-corresp--level}, as
\[
\beta_U([g]_{U(h_f)})=[(g,gh_f)]_{U\times U}=([g]_U,[gh_f]_U)=(\delta_1([g]_{U(h_f)}),\delta_2([g]_{U(h_f)})).
\]
We conclude from part \eqref{non-inj} of Lemma \ref{lem:birational}.
\end{proof}
Then the associated Hecke operator
\begin{equation}\label{defn:Hecke-operator}
T_{h_f}: L^2([\bG(\A)]_U,\mu)\rightarrow L^2([\bG(\A)]_U,\mu),\quad (T_{h_f} \varphi)(x)=\frac{1}{{\rm deg}(h_f)}\sum_{y\in \delta_2(\delta_1^{-1}(x))} \varphi(y)
\end{equation}
is self-adjoint relative to the inner product in \eqref{eq:inner-prod}, with spectrum contained in $[-1,1]$.

\begin{remark}\label{rem:Hecke}
Let us recall the relation between the Hecke operator in \eqref{defn:Hecke-operator} and the classical Hecke operator. To put them on the same footing, for a non-zero rational $y\in\Q^\times$ coprime to ${\rm Ram}_B$, we define $a(y)\in \bG(\A_f)$ by $a(y)_p=e$ for $p\mid d_B$ and
\[
\phi_p(a(y)_p)=\begin{pmatrix} y & \\ & 1\end{pmatrix},
\]
for $p\nmid d_B$, where $\phi_p\colon\bG(\Q_p)\xrightarrow{\,\sim\,} \mathbf{PGL}_2(\Q_p)$ is the isomorphism defined in \S\ref{sec:Shimura}. For an integer $N\in\N$, prime to $d_B$, we wish to relate $T_{a(1/N)}$ to the normalized classical Hecke operator $T_N$.

By definition, when $N$ is square-free, $T_N=N^{-1/2}[K_fa(1/N)K_f]$ where $[K_fa(1/N)K_f]$ is correspondence the on the quaternionic variety $[\bG(\A)]_U=Y_U$ induced by the double coset $K_fa(1/N)K_f\subset \bG(\A_f)$. Thus, in the square-free case, the two operators agree up to normalization, namely ${\rm deg}(a(1/N)) T_{a(1/N)}=N^{1/2}T_N$. In general, we have
\[
N^{1/2}
T_{N} =   \sum_{M^2 \mid N} {\rm deg}(a(M^2/N)) T_{a(M^2/N)}.
\]
and hence by M\"obius inversion
 \[
T_{a(1/N)} = \frac{N^{1/2}}{{\rm deg}(a(1/N))} \sum_{M^2 \mid N}  \frac{\mu(M)}{M} T_{N/M^2}.
\]
In particular, if $\lambda(N)$ is an eigenvalue for $T_{a(1/N)}$ on $L^2(Y_U, \mu)$, the Ramanujan conjecture for the classical Hecke operator $T_N$ translates to the statement that $\lambda(N)\ll_\varepsilon N^{-1/2+\varepsilon}$.
\end{remark}

\section{The spectral expansion}\label{sec:small-q}

Recall the joint Heegner period $P_\scD^\Delta(\varphi;s)$ from \eqref{eq:defn-PEdelta}, as well as the context of Theorem \ref{main-thm}, in which we take $U=K$ almost maximal and write
\[
Y_K=[\bG(\A)]_K,\qquad Y_K\times Y_K=[(\bG\times \bG)(\A)]_{K\times K}.
\]
In this section we make an explicit choice of orthonormal basis of $L^2(Y_{K_0(N)})$ and use \eqref{eq:joint-period-to-reg} to spectrally expand $P_\scD^\Delta(\varphi;s)$ across the Hecke correspondence $Y_{K_0(N)}^\Delta$, introduced in \eqref{eq:first-def-Hecke-corr}. We record this result in Proposition \ref{lemma-spectral-decomp}. 

\subsection{The level $N$ Hecke correspondence}\label{sec:level-N-Hecke-cor}

We begin by observing the equality between $Y_K^\Delta(s)$ and another Hecke correspondence, which will be more suitable for computations. The finite index subgroup $K(s)_f=K_f\cap sK_f s^{-1}$ of $K_f$ only depends on the image of $s$ in $\Q^\times\backslash\mathcal{I}(\mathscr{O}_\scD)$ under the map \eqref{eq:s-to-ideal}. As in \S\ref{sec:adelic-shift}, let $\sk\in\mathcal{I}(\mathscr{O}_\scD)$ denote the unique primitive integral ideal representing this image. We let $N=\Nr\sk$ denote the norm of $\sk$. Unlike in the definition \eqref{eq:defn-q}, we shall not assume that $\sk$ is of minimal norm in its ideal class. Note that $N$ is divisible only by primes that are split or ramified in $E$, since $\sk$ was assumed primitive. Recall furthermore the reduction steps in \S\ref{sec:reductions}, from which it follows that $N$ and $d_B$ are relatively prime.

We denote by $Y_0^\Delta(N)$ the Hecke correspondence $Y_K^\Delta(a(1/N))$ defined in \eqref{eq:first-def-Hecke-corr}.  

\begin{lemma}\label{lemma:conversion-to-Hecke-congruence}
We have $Y^\Delta_0(N)=Y^\Delta_K(s)$.
\end{lemma}

\begin{proof}
Recall from Lemma \ref{lemma:Hecke-graph} that
\[
Y^\Delta_0(N)=[\Delta\bG(\A)(e,a(1/N))]_{K\times K},\qquad Y^\Delta_K(s)=[\Delta\bG(\A)(e,s)]_{K\times K}.
\]
From \eqref{eq:same-Hecke-corr} it suffices to show that there is $k_N, k_N'\in \prod_{p\mid N}K_p$ such that $a(1/N)=k_N s k_N'$. We do this place by place, and so we factorize $s=\prod_p s_p$ and $N=\prod_{p\mid N}p^k$.

Let $p\mid N$ be prime. If $p$ splits in $E$ then, as noted in the example following Lemma \ref{lem:birational}, we may find $k_p\in K_p$ such that $k_p^{-1}\bT(\Q_p)k_p$ is the diagonal subgroup of $\mathbf{PGL}_2(\Q_p)$. It follows that $s_p=k_p a(p^{-k}) k_p^{-1}$. When $p$ is ramified in $E$, then $s_p$ is conjugate over $K_p$ to $\left(\begin{smallmatrix} 0 & p^{-k} \\ 1 & 0\end{smallmatrix}\right)$, and the Cartan coset of the latter element is given by $a(p^{-k})$; we deduce that there are $k_p,k_p'\in K_p$ such that $s_p=k_p a(p^{-k})k_p'$. Note finally, as was done in \S \ref{subsec:alignment}, that when $p$ is insert in $E$, then $\bT(\Q_p)\subset K_p$; in this case there is nothing to prove.
\end{proof}

\subsection{Level $K_0(N)$ toric periods}\label{sec:level-N-toric-periods}
For an integer $N\geq 1$ coprime with $d_B$, we let
\begin{equation}\label{def:K0N}
\begin{aligned}
K_0(N)_f&=K_f\cap a(1/N)K_fa(1/N)^{-1}<\bG(\A_f),\\
K_0(N)&=K\cap a(1/N)Ka(1/N)^{-1}=K_0(N)_fK_\infty<\bG(\A).
\end{aligned}
\end{equation}
When $\bG=\mathbf{PGL}_2$ the definitions \eqref{def:K0N} recover the classical Hecke congruence subgroup of $\mathbf{PGL}_2(\A_f)$; as we are borrowing the notation from that $B=\Mat(\Q)$ setting, we caution the reader to recall the dependence of $K_0(N)$ on $\bG$. Such subgroups fall into the form described in \eqref{def:Eichler-congruence-sbgps}, with $h_f=a(1/N)$.

Recall from the reduction step \S\ref{reduction-to-arch-trans} that we may take the shifting element $g$ in the homogeneous Heegner datum $\scD$ to satisfy $g_f=e$. (The Heegner condition, on the other hands, asks that $g_\infty$ conjugates $\bT(\R)$ to $K_\infty$.) Since $K_0(N)$ is of the form \eqref{eq:U-cap-conj}, it follows that $[\bT(\A)g]_{K_0(N)}$ is an aligned Heegner packet, in the sense of Definition \ref{defn:aligned}. The natural projection
\[
[\bT(\A)g]_{K_0(N)}\rightarrow [\bT(\A)g]_K=H_\scD
\]
is therefore bijective, and we may view $H_\scD$ as a subset of $Y_{K_0(N)}=[\bG(\A)]_{K_0(N)}$. It makes sense then to consider, for $u\in C^\infty(Y_{K_0(N)})$, the period integral
\begin{equation}\label{periodintegral}
W_\scD(u)=\int_{[\bT(\A)]} u(tg)\,{\rm d} t=\frac{1}{|H_\scD|}\sum_{t\in H_\scD} u(t),
\end{equation}
which we shall refer to as a \textit{level $K_0(N)$ toric period}.

In particular, the period integral $P_\scD^\Delta(f_1\otimes f_2;s)$, originally defined in \eqref{defn:PEdelta}, may be viewed as a level $K_0(N)$ toric period of the restriction $f_1\otimes f_2|_{Y^\Delta_0(N)}$ by means of the identification  \eqref{eq:joint-period-to-reg} and Lemma \ref{lemma:conversion-to-Hecke-congruence}. As a first step towards Proposition \ref{prop-smallq}, we now spectrally expand the joint Heegner period $P_\scD^\Delta(f_1\otimes f_2;s)$ across any intermediate Hecke set $Y_0^\Delta(N)=Y_K^\Delta(s)$ containing $H^\Delta_\scD(s)$. (Recall from \S\ref{sec:joint-Heegner-packet}, as well as the identification \eqref{eq:iso-Pic}, that $H^\Delta_\scD(s)$ depends only on the class of $s$ in ${\rm Pic}(\mathscr{O}_\scD)$.) In preparation for this, we recall some well-known facts about the spectral theory of the automorphic quotient $[\bG(\A)]$.

\subsection{The spectral decomposition of $L^2([\bG(\A)],m_\bG)$}\label{sec:spec-decomp-K(s)}

The right regular representation of $\bG(\A)$ on $L^2([\bG(\A)],m_\bG)$ decomposes as a Hilbert space direct sum of closed $\bG(\A)$-invariant subspaces
\[
L^2_{\rm disc}([\bG(\A)],m_\bG)\oplus L^2_{\rm cont}([\bG(\A)],m_\bG).
\]
The continuous subspace $L^2_{\rm cont}([\bG(\A)],m_\bG)$, which is non-zero only when $B$ is the split algebra, is spanned by Eisenstein series; we shall describe it in more detail below.

Its orthocomplement --- the discrete subspace --- breaks up further as
\[
L^2_{\rm disc}([\bG(\A)],m_\bG)=L^2_{\rm res}([\bG(\A)],m_\bG)\oplus L^2_{\rm cusp}([\bG(\A)],m_\bG),
\]
with the cuspidal subspace $L^2_{\rm cusp}([\bG(\A)],m_\bG)$ being defined here as the orthocomplement of $L^2_{\rm res}([\bG(\A)],m_\bG)$ in $L^2_{\rm disc}([\bG(\A)],m_\bG)$. The residual subspace is given by
\begin{equation}\label{eq:residual-decomp}
L^2_{\rm res}([\bG(\A)],m_\bG)=\bigoplus_{\eta\in {\rm Nr}_B(\bG(\A))^\vee} \C_{\eta\circ {\rm Nr}_B},
\end{equation}
where ${\rm Nr}_B:\bG\rightarrow\mathbb{G}_m$ is induced by the reduced norm on $\bB^\times$. The map ${\rm Nr}_B\restriction_{\bG(\A)}$ factorizes through the compact abelian group $\bG(\Q)\backslash\bG(\A)/\bG(\A)^+$, where $\bG(\A)^+$ is the image in $\bG(\A)$ of the $\A$-points of the simply connected cover $\bG^{(1)}$ --- the norm one quaternions --- of $\bG$, and induces an isomorphism of $\bG(\Q)\backslash\bG(\A)/\bG(\A)^+$ onto 
\begin{equation} \label{eq:char-group}
\begin{cases}
\Q^\times\backslash\A^\times/{\A^\times}^2,& B_\infty=\Mat(\R);\\
\Q_{>0}\backslash (\R_{>0}\times \A_f^\times)/{\A^\times}^2,& \textrm{else}.
\end{cases}
\end{equation}
For more details, see \cite[\S 2.3]{Kh}. In particular, when $B_\infty=\Mat(\R)$, the characters appearing in $L^2_{\rm res}$ are the quadratic Dirichlet characters composed with the determinant map. 

We now take $\bG=\mathbf{PGL}_2$. Finally, we recall the description of the continuous spectrum of $L^2([\bG(\A)],m_{\bG})$; see \cite[\S 4]{GJ} for more details. Let $\mathbf{P}$ the Borel subgroup consisting of upper triangular matrices mod the center, with unipotent radical $\mathbf{N}$. As usual, we let $K=\mathbf{PGL}_2(\widehat\Z){\rm PSO}(2)$. We have the Iwasawa decomposition $\bG(\A)=\mathbf{P}(\A)K$. For ${\sf s}\in\C$ --- to avoid a notational clash with the adelic shifting parameter, we shall write a complex number in \textit{sans serif} font, as ${\sf s}={\sf \sigma}+it$ --- and a unitary character $\chi$ of $\Q^\times\backslash\A^\times$, we consider the induced representation $\mathbf{U}(\chi,{\sf s})$ consisting of smooth $K$-finite functions $\phi:\bG(\A)\rightarrow \C$ transforming by $\chi(a/d)|a/d|^{\sf s}$ under left-multiplication by $\left[\begin{smallmatrix}a& b\\ 0 & d \end{smallmatrix}\right]\in \mathbf{P}(\A)$ and such that
\begin{equation}\label{big-induced-integral}
\int_K|\phi (k)|^2\, {\rm d} k<\infty.
\end{equation}
Here, ${\rm d} k$ is any Haar measure on $K=\prod_v K_v$, which we henceforth take to be the product of the probability Haar measures of $K_v$ for each $v$. We equip $\mathbf{U}(\chi,{\sf s})$ with the inner product induced by \eqref{big-induced-integral} and let $\mathbf{H}(\chi,{\sf s})$ be its completion. This defines a Hilbert space representation of $\bG(\A)$, which is unitary when ${\rm Re}({\sf s})=1/2$. 

For a right $K$-finite function $f\in\mathbf{H}(\chi,{\sf s})$, the associated Eisenstein series is defined on ${\rm Re}({\sf s})>1$ by
\[
E_{f,\chi}(g,{\sf s})=\sum_{\gamma\in \mathbf{P}(\Q)\backslash\bG(\Q)} f(\gamma g).
\]
Then $E_{f,\chi}(g,{\sf s})$ converges absolutely and uniformly on compacta (modulo the center) as a function of $g\in \bG(\A)$, and continues meromorphically as a function of ${\sf s}$ to all of $\C$, with at most a simple pole (when $\chi^2=1$) at ${\sf s}=1$ on the right half plane ${\rm Re}({\sf s})\geq 1/2$. In particular, $E_{f,\chi}(g,{\sf s})$ is regular on the unitary axis ${\rm Re}({\sf s})=1/2$ and spans $L^2_{{\rm cont}}([\bG(\A)],m_{\bG})$.

For ${\sf s}\in\C$ let $\delta^{\sf s}$ denote the character of $\mathbf{P}(\A)$ sending $\left[\begin{smallmatrix}a& b\\ 0 & d \end{smallmatrix}\right]$ to $|a/d|_\A^{\sf s}$. For each character $\chi$ of $\Q^\times\backslash\A^1$ we make a $K$-equivariant identification of $\mathbf{H}(\chi):=\mathbf{H}(\chi,0)$ with $\mathbf{H}(\chi, {\sf s})$ via the map $f\mapsto f_{\sf s}:=f\delta^{\sf s}$. The $L^2$-spectral expansion states that for a function $F\in L^2([\bG(\A)],m_{\bG})$, we have
\begin{equation}
\begin{aligned}
\label{eq:spectral-expansion}
F=&\sum_{\eta\in ({\rm Nr}_B(\bG(\A))^\vee}\langle F,\eta\circ {\rm Nr}_B\rangle (\eta\circ {\rm Nr}_B)+\sum_{\sigma \textrm{ cuspidal } }\sum_{u\in\mathscr{B}_\sigma}\langle F,u\rangle u  \\
&\qquad+ \delta_{B=\Mat(\Q)}\frac{1}{4\pi i}\int_{{\rm Re}({\sf s})=1/2} \sum_{\chi\in\widehat{\Q^\times\backslash\A^1}} \sum_{f\in \mathscr{B}_\chi}  \langle F, E_{f_{\sf s},\chi}(\cdot, {\sf s})\rangle  E_{f_{\sf s},\chi}(\cdot ,{\sf s}) \, {\rm d} {\sf s},
\end{aligned}
\end{equation}
where $\mathscr{B}_\sigma$ is an orthonormal basis for the cuspidal representation $\sigma\subset L^2_{\rm cusp}([\bG(\A)],m_{\mathbf{PGL}})$ and $\mathscr{B}_\chi$ is an orthonormal basis for the principal series representation $\mathbf{H}(\chi)$ of $\bG(\A)$. The above spectral expansion holds \textit{a priori} only in $L^2([\bG(\A)],m_{\bG})$, but it can be shown to hold pointwise for any $F\in C_c^\infty([\bG(\A)])$, with the sums and integrals on the right-hand side converging absolutely (see \cite[\S2.2.1, Remark]{MV}).

\subsection{Choice of basis for $L^2(Y_{K_0(N)},\mu)$}\label{sec:basis-choice}
Having described the spectral decomposition of $L^2([\bG(\A)],m_\bG)$, we now take $K_0(N)$-invariants to obtain the corresponding spectral decomposition of $L^2(Y_{K_0(N)},\mu)$, where $\mu$ is the Borel probability measure on $Y_{K_0(N)}=[\bG(\A)]_{K_0(N)}$ discussed in \S\ref{sec:Shimura}. In preparation for the spectral expansion of Proposition \ref{lemma-spectral-decomp}, we shall prescribe bases of the discrete and, when $B=\Mat(\Q)$, continuous spectrum of $L^2(Y_{K_0(N)},\mu)$. 

We begin with the residual spectrum. The characters appearing in $L^2_{\rm res}([\bG(\A)],m_\bG)^{K_0(N)}$ are those whose kernels contain the image of $K_0(N)$ under ${\rm Nr}_B$.\footnote{At full level $U=K$ the residual subspace is reduced to the trivial character. This accounts for the simpler form of the limiting measure in Theorem \ref{main-thm} relative to the results in \cite{Kh}, which are formulated on the homogeneous space $[(\bG\times\bG)(\A)]$.} Then $\eta\circ {\rm Nr}_B$, as $\eta$ varies over characters of ${\rm Nr}_B(\bG(\A))$ with ${\rm Nr}_B(K_0(N))\subset\ker\eta$, forms an orthonormal basis of $L^2_{\rm res}(Y_{K_0(N)},\mu)=L^2_{\rm res}([\bG(\A)],m_\bG)^{K_0(N)}$ adapted to the decomposition \eqref{eq:residual-decomp}. We note that the functions $\eta\circ {\rm Nr}_B$ are indeed $L^2$-normalized since $\mu$ is a probability measure on $Y_{K_0(N)}$.

We now pass to the cuspidal spectrum. We have a Hilbert direct sum decomposition
\begin{equation}\label{eq:L2-K(s)-decomp}
L^2_{\rm cusp}([\bG(\A)]_{K_0(N)},\mu)=\bigoplus_{\sigma\subset L^2_{\rm cusp}([\bG(\A)],m_\bG)} \sigma^{K_0(N)}
\end{equation}
into irreducible subrepresentations having non-zero $K_0(N)$-invariants. We define an orthonormal basis $\mathscr{B}_{\rm cusp}(N)=\bigsqcup_\sigma \mathscr{B}_\sigma(N)$ of $L^2_{\rm cusp}(Y_{K_0(N)},\mu)$, adapted to the decomposition \eqref{eq:L2-K(s)-decomp}, in the following way. For every $\sigma$ appearing non-trivially in \eqref{eq:L2-K(s)-decomp}, let $M\mid N$ be the minimal divisor of $N$ for which $\sigma$ has non-zero $K_0(M)$ invariants. We let $\varphi_\sigma^{\rm new}\in \sigma^{K_0(M)}$ be $L^2$-normalized and take
\[
\mathscr{B}_\sigma^{\rm skew}(N)=\left\{a(L)^{-1}.\varphi_\sigma^{\rm new} : L\mid N/M \right\}.
\]

While each member of $\mathscr{B}_\sigma^{\rm skew}(N)$ is $L^2$-normalized, they are are not orthogonal to each other. We may, however, orthonormalize this ``skew'' basis, while controlling the complexity of the coefficients, as the next lemma shows.

\begin{lemma}\label{lemma:skew-basis}
We may find an orthonormal basis $\mathscr{B}_\sigma(N)$ of $\sigma^{K_0(N)}$ consisting of linear combinations of $\mathscr{B}_\sigma^{\rm skew}(N)$ with coefficients bounded by $O(1)$.
\end{lemma}

\begin{proof}
This is proved in \cite[Lemma 9]{BMi} for the case $\bG=\mathbf{PGL}_2$, in the classical language; see \cite{SPY} for a more general version and further references. The argument can be written adelically and be made to apply to any $\bG$. Indeed, since $N$ is prime to $d_B$, one only has to work locally at places for which $\bG(\Q_p)\simeq\mathbf{PGL}_2(\Q_p)$.
\end{proof}

In the case $B=\Mat(\Q)$, we choose an explicit orthonormal basis for the level $N$ Eisenstein series. Fix a prime $p$. For $k\in\Z_{\geq 0}$ let
\begin{equation}\label{def-hecke-cong-sbgps}
\mathcal{K}_{0,p}(p^k)=\left\{\begin{pmatrix} a & b\\ c & d\end{pmatrix}\in \mathbf{GL}_2(\Z_p): c\in p^k\Z_p\right\}
\end{equation}
denote the Hecke congruence subgroup of $\mathbf{GL}_2(\Z_p)$ of exponential level $k$. When $k=0$ we write $\mathcal{K}_p$ instead of $\mathcal{K}_{0,p}(1)$.

We let $\chi_p$ be a unitary character of $\Q_p^\times$. We denote by $\ell\in\Z_{\geq 0}$ the exponential conductor of $\chi$. Let $\bH(\chi_p):=\bH(\chi_p,0)$ denote the induced representation from \S\ref{sec:spec-decomp-K(s)}, with inner product given by integration over $\mathbf{GL}_2(\Z_p)$, equipped with its probability Haar measure.

\begin{lemma}\label{lemma:ONB}
The dimension of $\bH(\chi_p)^{\mathcal{K}_{0,p}(p^k)}$ is $\max\{0,k-2\ell+1\}$. In particular, this space is non-zero precisely when $2\ell\leq k$.

When $k=\ell=0$, the one-dimensional space $\bH(\chi_p)^{\mathcal{K}_p}$ is generated by $f_{\chi_p}^\circ$, the unique function in $\bH(\chi_p)$ whose restriction to $\mathbf{GL}_2(\Z_p)$ is the constant function $1$. Moreover $f_{\chi_p}^\circ$ is of $L^2$-norm 1.

When $k\geq1$ and $2\ell\leq k$, an orthonormal basis of $\bH(\chi_p)^{\mathcal{K}_{0,p}(p^k)}$ is given by $\mathscr{B}_{\chi_p}(p^k)=\{f_{\chi_p}^{(j,k,\ell)}:\ell\leq j\leq k-\ell\}$, with $f_{\chi_p}^{(j,k,\ell)}$ defined as follows:
\begin{enumerate}
\item\label{ell0} Assume $\ell=0$. Then $f_{\chi_p}^{(j,k,0)}$ is the unique function in $\bH(\chi_p)$ sending $\left(\begin{smallmatrix} a & b \\ c & d\end{smallmatrix}\right)\in\mathbf{GL}_2(\Z_p)$ to
\[
\begin{cases}
A_j{\bf 1}_{p^j\Z_p^\times}(c),& 0\leq j\leq k-1;\\
A_k{\bf 1}_{p^k\Z_p-\{0\}}(c),&  j=k.
\end{cases}
\]
\item Assume $\ell\geq 1$. Then $f_{\chi_p}^{(j,k,\ell)}$ is the unique function in $\bH(\chi_p)$ sending $\left(\begin{smallmatrix} a & b \\ c & d\end{smallmatrix}\right)\in\mathbf{GL}_2(\Z_p)$ to $\chi_p(ac^{-1})A_j{\bf 1}_{p^j\Z_p^\times}(c)$.
\end{enumerate}
The normalization constant $A_j$ satisfies $A_j\asymp p^{j/2}$, for $0\leq j\leq k$. 
\end{lemma}

\begin{proof}
The uniqueness property follows from the Iwasawa decomposition, which implies that functions in $\bH(\chi_p)$ are uniquely determined by their restriction to $\mathbf{GL}_2(\Z_p)$. The calculation of the dimension of $\bH(\chi_p)^{\mathcal{K}_{0,p}(p^k)}$ is a classical result of Casselman \cite{Ca}. When $\ell\geq 1$, the family $\mathscr{B}_{\chi_p}(p^k)$ is precisely the orthonormal basis described in \cite[\S 2.6]{BH}. When $\ell=0$, the functions $f_{\chi_p}^{(j,k,0)}$ differ from those given in \textit{loc.\ cit.}, but the proof for $\ell\geq 1$ extends to the $\ell=0$ family in \eqref{ell0}. Indeed, for $\chi_p$ unramified the functions $f_{\chi_p}^{(j,k,0)}$ can easily be shown to lie in $\bH(\chi_p)^{\mathcal{K}_{0,p}(p^k)}$; their supports are disjoint, their normalization is unitary, and $|\mathscr{B}_{\chi_p}(p^k)|=\dim \bH(\chi_p)^{\mathcal{K}_{0,p}(p^k)}$.
\end{proof}

If $N=\prod_{p\mid N} p^{k_p}$ is the prime factorization of $N$ we let $\mathcal{K}_0(N)=\prod_{p\nmid N} {\rm GL}_2(\Bbb{Z}_p) \prod_{p\mid N}\mathcal{K}_{0,p}(p^{k_p})$ be the level $N$ Hecke congruence subgroup of $\mathbf{GL}_2(\widehat{\Z})$, where $\mathcal{K}_{0,p}(p^k)\subset\mathbf{GL}_2(\Z_p)$ is defined in \eqref{def-hecke-cong-sbgps}. Let $\chi$ be a character of $\Q^\times\backslash\A^1$ such that $\bH(\chi)$ has non-zero $\mathcal{K}_0(N)$-invariants. In other words, the conductor $M$ of $\chi$ should satisfy $M^2\mid N$. Let $L\mid N/M^2$. We factorize $M=\prod_p p^{\ell_p}$ and $L=\prod_p p^{j_p-\ell_p}$, so that $\ell_p\leq j_p\leq k_p-\ell_p$. Let
\[
f_\chi^{(L,M,N)}=\prod_{p\nmid N}f_{\chi_p}^\circ\prod_{p\mid N}f_{\chi_p}^{(j_p,k_p,\ell_p)}f_{\chi_\infty}^\circ\in\bH(\chi),
\]
where the local factors at the finite places are defined in Lemma \ref{lemma:ONB} and $f_{\chi_\infty}^\circ$ is a unit vector in the one-dimensional space of $K_\infty$-invariants for $\bH(\chi_\infty)$. Let
\[
\mathscr{B}_\chi(N)=\left\{f_\chi^{(L,M,N)} : L\mid N/M^2 \right\}.
\]
Finally we write $\mathscr{B}_{\rm cont}(N)= \bigsqcup_{M^2\mid N} \bigsqcup_{\chi\!\!\mod M}\mathscr{B}_\chi(N)$. We agree to write $(f,\chi)$ for a general element of $\mathscr{B}_{\rm cont}(N)$, where $f\in\mathscr{B}_\chi(N)$.

\begin{remark}\label{rem-cusp-ONB}
The basis in $\mathscr{B}_\sigma^{\rm skew}(N)$ (resp. $\mathscr{B}_\chi(N)$), will be used in the proof of Lemma \ref{lemma-triple-prod} (resp. Lemma \ref{prop-local-RS-bound1}), to bound the local factors of triple product integrals. On the other hand, a pleasant feature of the proof of Lemma \ref{lemma-L-function-sum} (resp. Lemma \ref{lemma:Weyl-Eis-bound}), which bounds the local factors of the toric integrals, is that for a given $\sigma\subset L^2_{\rm cusp}([\bG(\A)],m_\bG)$ (resp. character $\chi$ of $\Q^\times\backslash\A^1$) any orthonormal basis for $\sigma^{K_0(N)}$ (resp. ${\mathbf H}(\chi)$) will do.
\end{remark}

\subsection{Decomposing the joint Heegner period across intermediate Hecke correspondences}\label{sec:spectral-decomp}
Recall the $L^2$-spectral expansion \eqref{eq:spectral-expansion} and its pointwise version in $C_c^\infty([\bG(\A)])$. By taking $K_0(N)$-invariants, we deduce that for any $\varphi\in C_c^\infty(Y_0(N))$ we have
\begin{equation}
\label{eq:spectral-expansion2}
\begin{aligned}
\varphi(g)=\langle \varphi,{\bf 1}\rangle&+\sum_{u\in\mathscr{B}_{\rm cusp}(N)}\langle \varphi,u\rangle u(g) \\
&+ \delta_{B=\Mat(\Q)}\frac{1}{2\pi i}\int_{{\rm Re}({\sf s})=1/2} \sum_{(f,\chi)\in \mathscr{B}_{\rm cont}(N)}  \langle \varphi, E_{f,\chi}(\cdot, {\sf s})\rangle  E_{f,\chi}(g,{\sf s}) \, {\rm d} {\sf s}.
\end{aligned}
\end{equation}
Here, we have simplified the contribution from the residual spectrum, whose $K_0(N)$-invariants reduce to a sum over $\eta\in ({\rm Nr}_B(\bG(\A))^\vee$ with $\ker\eta\subset {\rm Nr}_B(K_0(N))$. A local computation shows that ${\rm Nr}_B(K_0(N))\equiv\widehat{\Z}^\times \bmod {\A_f^{\times}}^2$. The description of ${\rm Nr}_B(\bG(\A))$ in \eqref{eq:char-group} and the triviality of the class group of $\Q$ imply ${\rm Nr}_B(\bG(\A))\slash {\rm Nr}_B(K_0(N))=1$, so that only the trivial character ${\bf 1}$ contributes.

The above spectral expansion holds \textit{a priori} only in $C_c^\infty(Y_0(N))$. Nevertheless, we claim that a version of it continues to hold in the non-compact case when $B=\Mat(\Q)$, again with absolute convergence in the sums and integrals on the right-hand side. Namely, for $A\geq 1$ large enough, and for functions in the Sobolev-type space
\begin{equation}\label{eq:Sob}
\{\varphi\in C^\infty(Y_0(N)):  \|\Delta^a\varphi\|_1\ll_{a,\varphi}1,\; \Delta^a\varphi \textrm{ decays rapidly at all cusps}, \forall a\leq A\},
\end{equation}
we can argue as follows. For the sum over $u\in\mathscr{B}_{\rm cusp}(N)$ in \eqref{eq:spectral-expansion2}, integration by parts, a global sup-norm bound $\|u\|_\infty\ll_N \lambda_\sigma^B$ for some fixed $B>0$, and the assumed estimates on $\|D\varphi\|_1$ imply that $\langle \varphi,u\rangle \ll_{A,\varphi} \lambda_{\sigma}^{-A+B}$, whence $\langle \varphi,u\rangle u(g)\ll_{A,\varphi} \lambda_{\sigma}^{-A+2B}$ uniformly in $g$. Taking $A$ large enough, the Weyl law for the cuspidal spectrum implies that the sum over $u\in\mathscr{B}_{\rm cusp}(N)$ converges absolutely. For the integral over Eisenstein series in \eqref{eq:spectral-expansion2}, a similar argument applies, but we need to use the rapid decay of the derivatives of $\varphi$, and a sup-norm bound of the form $\|E_{f,\chi}(\cdot,{\sf s})\restriction_\Omega\|_\infty \ll_{\Omega, f,\chi} |{\sf s}|^B$ for all $\Omega\subset Y_0(N)$ compact. With the claim established, the right-hand side of \eqref{eq:spectral-expansion2} for such $\varphi$ is a continuous function, and a pointwise equality can be deduced from the $L^2$-expansion.

Let $s\in \bT(\A)$, $\sk\in\mathcal{I}(\mathscr{O}_\scD)$, and $N=\Nr\sk$ be as in \S\ref{sec:level-N-Hecke-cor}. We are now ready to prove  the spectral expansion of the joint Heegner period $P_\scD^\Delta  (f_1\otimes f_2;s)$ across the Hecke correspondence $Y_0^\Delta(N)$.

\begin{prop}\label{lemma-spectral-decomp}
Let $f_1,f_2\colon Y_K\to \mathbb{C}$ be cuspidal Laplace eigenfunctions. For $N\in\N$ coprime to $d_B$ let $T_{a(1/N)}$ be the Hecke operator from \eqref{defn:Hecke-operator}, and assume that one of $f_1,f_2$  is a $T_N$-eigenfunction. Denote by $\lambda(N)$ its eigenvalue. Let $\scD$ be a homogeneous Heegner datum with $g_f=e$. With $W_\scD$ as in \eqref{periodintegral} we have
\begin{align*}
P_\scD^\Delta  (f_1\otimes f_2;s) &=\lambda(N)\langle f_1,f_2\rangle\\
&+ \sum_{u\in\mathscr{B}_{\rm cusp}(N)} \langle f_1(\cdot) f_2(\cdot\, a(1/N)), u\rangle  W_\scD(u)+ \delta_{B=\Mat(\Q)} {\rm Eis}_N(f_1,f_2),
\end{align*}
where ${\rm Eis}_N(f_1,f_2)$ is defined as
\[
\frac{1}{2\pi i}\int_{{\rm Re}({\sf s})=1/2} \sum_{(f,\chi)\in \mathscr{B}_{\rm cont}(N)}\langle f_1(\cdot )f_2(\cdot a(1/N)), E_{f,\chi}(\cdot, {\sf s})\rangle W_\scD(E_{f,\chi}(\cdot, {\sf s}))\, {\rm d} {\sf s}.
\]
\end{prop}

\begin{remark}
The joint Heegner period $P_\scD^\Delta  (f_1\otimes f_2;s)$ is invariant under $s\mapsto t_\Q s$, where $t_\Q\in\bT(\Q)$. Such a substitution changes $\sk$ to another primitive integral ideal $\mathfrak{n}$ in the same ideal class as $\sk$, and all such ideals arise from an appropriate choice of $t_\Q$. Applying Proposition \ref{lemma-spectral-decomp} with $t_\Q s$ then yields another expansion of $P_\scD^\Delta  (f_1\otimes f_2;s)$, but now relative to the level $\Nr \mathfrak{n}$. In particular, one may choose $t_\Q$ in such a way as to minimize $\Nr \mathfrak{n}$, as in the definition \eqref{eq:defn-q}.
\end{remark}
\begin{proof}
We use \eqref{eq:joint-period-to-reg} to write $P_\scD^\Delta  (f_1\otimes f_2;s)=P_\scD\big(\beta_K\circ(f_1\otimes f_2)\restriction_{Y^0_\Delta(N)}\big)$.
We spectrally expand $\beta_K\circ(f_1\otimes f_2)|_{Y_0^\Delta(N)}$ across an orthonormal basis for $L^2(Y_{K_0(N)},\mu)$. We apply the spectral expansion \eqref{eq:spectral-expansion2} to the function $g\mapsto f_1(g)f_2(gs)$ on $Y_{K_0(N)}$, which lies in \eqref{eq:Sob} since both $f_1,f_2$ are cuspidal, and take the integral over $H_\scD=[\bT(\A)g]$. Absolute convergence in the spectral expansion above  implies that we can exchange the summation over the discrete spectrum and the integral over the continuous spectrum with the torus integral. We obtain the statement of the proposition, upon writing $\langle f_1(\cdot)f_2(\cdot\, a(1/N)),{\bf 1}\rangle$ as
\begin{align*}
 \int_{[\bG(\A)]_{K_0(N)}}f_1(g)f_2(ga(1/N))\,{\rm d}\mu(g)&=\int_{[\bG(\A)]_K}\frac{1}{{\rm deg}(a(1/N))}\sum_{g\in\delta_1^{-1}(x)} f_1(g)f_2(g a(1/N))\,{\rm d}\mu(x)\\
&=\int_{[\bG(\A)]_K}f_1(x)\frac{1}{{\rm deg}(a(1/N))}\sum_{y\in \delta_2(\delta_1^{-1}(x))}f_2(y)\,{\rm d}\mu(x),
\end{align*}
which one recognizes as $\langle f_1, T_{a(1/N)}f_2\rangle=\lambda(N)\langle f_1,f_2\rangle$.\end{proof}

\section{The main estimate for small $q$}\label{sec:main-estimate}
Let $s\in\bT(\A_f)$. As usual we denote by $\sk$ the unique primitive integral ideal representing the image of $s$ under \eqref{eq:s-to-ideal} and write $N=\Nr\sk$. As in \S\ref{sec442}, we may and will assume that $N$ is coprime to $d_B$. Recall from \eqref{eq:defn-q} that $q\leq N$ denotes the minimal norm of an integral ideal representing the class of $\sk$ in ${\rm Pic}(\mathscr{O}_E)$. Let $\scD=(\iota,g)$ be a homogeneous Heegner datum, defined in \S\ref{sec:special}, and assume that $g_f=e$, as in \S\ref{reduction-to-arch-trans}.

The goal of this section is to prove a non-trivial bound on the joint Heegner period $P_\scD^\Delta(\varphi;s)$, for arithmetic test functions $\varphi$ in $L^2_{\rm cusp}(Y_K\times Y_K,\mu\times\mu)$, when $q$ is small relative to $D^{1/2}$.

\begin{prop}\label{prop-smallq}
Let $f_1,f_2\in L^2_{\rm cusp}(Y_K,\mu)$ be $L^2$-normalized Hecke cusp forms. Assume the Lindel\"of Hypothesis and, whenever $B$ is indefinite, the Ramanujan conjecture for $\mathbf{PGL}_2/\Q$. Then, with notation as above,
\[
P_\scD^\Delta(f_1\otimes f_2;s) \ll_\varepsilon N^{-1/2+\varepsilon}+ N^{1/2}D^{-1/4+\varepsilon}.
\]
In particular, for any $\varepsilon>0$, as $q\rightarrow\infty$ subject to $q\leq D^{1/2-\varepsilon}$, we have $P_\scD^\Delta(f_1\otimes f_2;s) =o(1)$.
\end{prop}

We remind the reader that the implied constant in Proposition \ref{prop-smallq}, which provides the first estimate on $h_\varepsilon(q,D))$ in the error term of Theorem \ref{main-thm}, depends on the quaternion algebra $B$. We shall not include this dependence in the notation, since $B$ is considered to be fixed throughout.

The proof follows the basic contours of the proof sketch in \S \ref{sec:bottom-q-sketch}. Indeed, after applying the spectral expansion from Proposition \ref{lemma-spectral-decomp}, it remains to bound the triple  products
\[
\langle f_1(\cdot)f_2(\cdot\, a(1/N)),u\rangle\quad\textrm{and}\quad \langle f_1(\cdot )f_2(\cdot\, a(1/N)), E_{f,\chi}(\cdot, 1/2+it)\rangle,
\]
as well as the level $K_0(N)$ toric periods
\[
W_\scD(u)\quad\textrm{and}\quad W_\scD(E_{f,\chi}(\cdot, 1/2+it)).
\]
We shall do so now by using their relation with various $L$-functions, whose size is controlled by the Lindel\"of Hypothesis. Such period formulae also involve products of local factors at bad primes, which we furthermore have to control. We bound these local factors for the toric periods of cusp forms and Eisenstein series in Lemmata \ref{lemma-L-function-sum} and \ref{lemma:Weyl-Eis-bound}. For the triple product local factors, we do this in \S\ref{appendix:triple-prod}, with the assumption of the Ramanujan conjecture providing optimal exponents. Somewhat weaker bounds can be obtained unconditionally in the same way. 

\subsection{The cuspidal periods}\label{sec:estimate-cusp-cont}  
In this section we bound both $\langle f_1(\cdot)f_2(\cdot\, a(1/N)),u\rangle$ and $W_\scD(u)$ for $u$ cuspidal of level $N$. Our immediate goal will simply be to relate these to $L$-functions using appropriate period formulae, while retaining good control on the local factors.

For $\sigma \subset L^2_{\rm cusp}([\bG(\A)],m_\bG)$ with $\sigma^{K_0(N)}\neq \{0\}$, and $u\in\mathcal{B}_\sigma(N)$, we begin by bounding the cuspidal triple product $\langle f_1(\cdot)f_2(\cdot\, a(1/N)),u\rangle$ (or, rather, its square) by a central $L$-value, using the Ichino--Watson formula, and the bounds on local factors proved in \S\ref{appendix:triple-prod}.

Let $f_1,f_2\in L^2_{\rm cusp}(Y_K)$ be $L^2$-normalized Hecke eigenforms, as in the statement of Proposition \ref{prop-smallq}. Then $f_i$ generates an irreducible cuspidal automorphic representation $\sigma_i$ of $\bG(\A)$. Let $\pi_i=\sigma_i^{\rm JL}$ denote the Jacquet--Langlands lift of $\sigma_i$ to $\mathbf{PGL}_2(\A)$. Since the level structure $U=K$ is taken to be almost maximal, the conductor of $\pi_1$ and $\pi_2$ is the reduced discriminant $d_B$ of $B$.

\begin{lemma}\label{lemma-triple-prod}
Let $N\geq 1$ be an integer coprime with $d_B$. Let $\sigma \subset L^2_{\rm cusp}([\bG(\A)],m_\bG)$ with $\sigma^{K_0(N)}\neq \{0\}$ and let $\pi=\sigma^{\rm JL}$ denote the Jacquet--Langlands lift of $\sigma$ to $\mathbf{PGL}_2(\A)$. Let $M\mid N$ be the conductor of $\sigma$ and let $L\mid N/M$. Let $u^{\rm new}\in\sigma^{K_0(M)}$ be an $L^2$-normalized new vector. Then, assuming the Ramanujan conjecture for $\mathbf{PGL}_2/\Q$,
\begin{displaymath}
\begin{split}
\langle f_1(\cdot)f_2(\cdot\, a(1/N)), &a(1/L).u^{\rm new}\rangle^2\\
&\ll_\varepsilon N^{-1+\varepsilon} L(1/2,\pi_1\otimes \pi_2\otimes \pi) \min\Big(1, \exp[- \pi (|\nu_\pi | - |\nu_{\pi_1}| - |\nu_{\pi_2}|) ]\Big),
\end{split}
\end{displaymath}
where the spectral parameter $\nu_\pi$ of $\pi$ was defined in \S\ref{sec:reduction-to-sparse}. 
\end{lemma}

\begin{proof}
Fix isomorphisms $\pi_i\simeq\otimes_v\pi_{i,v}$ and denote  by $\otimes_v f_{i,v}$ the image of $f_i$. Let $u=a(1/L).u^{\rm new}$ and write $\otimes_v u_v$ for the image of $u$ under a fixed isomorphism $\sigma\simeq\otimes_v \sigma_v$. Thus, if $p^\ell\| M$, $p^k\| N$, and $p^d\| L$, with $\ell+d\leq k$, then $u_p=a(\varpi_p^{-d}).u_p^{\rm new}$. For all but finitely many places $v$ the local components $f_{i,v}, u_v$ are unitary.

We now convert $|\langle f_1(\cdot)f_2(\cdot\, a(1/N)), u\rangle|^2$ to a product of local triple product model integrals and a global triple product $L$-function. We use the local measure conventions of \S\ref{sec:Shimura}, for which $m_\bG=C_\bG^{-1}\prod_v{\rm d}g_v$. We begin by defining the following $L^2$-normalized matrix coefficients for $\pi_{1, v},\pi_{2, v}$ and $\sigma_v$, respectively:
\[
\Phi_{1,v}(g)=\langle g.f_{1,v},f_{1,v} \rangle/\| f_{1, v}\|^2,
\]
\[
\Phi_{2,v}(g)=
\begin{cases}
\langle ga(\varpi_p^{-k}).f_{2,p},a(\varpi_p^{-k}).f_{2,p} \rangle/\| f_{2, p}\|^2,& v=p, \; p^k\|N, k\geq 1;\\
\langle g.f_{2,v},f_{2,v} \rangle/\| f_{2, v}\|^2,& v\nmid N,
\end{cases}
\]
and
\[
\Upsilon_v(g)=
\begin{cases}
\langle g.a(\varpi_p^{-d}).u_p^{\rm new},a(\varpi_p^{-d}).u_p^{\rm new} \rangle/\| u_p\|^2, & v=p,\, p^d\|L,\, d\geq 1;\\
\langle g.u_v,u_v \rangle/\| u_v\|^2,& v\nmid L.
\end{cases}
\]
The local triple product integral
\[
I_v =\int_{\bG(\Q_v)}\Phi_{1,v}(g_v)\Phi_{2,v}(g_v)\Upsilon_v(g_v)\, {\rm d} g_v
\]
converges absolutely \cite[Lemma 2.1]{Ich}. Let $I_S=\prod_{v\in S}I_v$, where $S$ denotes the set of primes dividing $Nd_B$, as well as the archimedean place. The Ichino--Watson formula states \cite{Ich, Wat} that
\[
|\langle f_1(\cdot)f_2(\cdot\, a(1/N)), u\rangle|^2=c I_S\frac{L^S(1/2, \pi_1\otimes \pi_2\otimes \pi)}{L^S(1,\pi_1, {\rm Ad})L^S(1, \pi_2, {\rm Ad})L^S(1,\pi, {\rm Ad})},
\]
for a constant $c>0$ depending only on local measure normalizations and coarsely bounded by $c \ll \exp(O(|S|)) \ll (Nd_B)^\varepsilon$, where the Euler products defining the $L$-functions are restricted to primes $p\notin S$. For the sake of cleanly quoting bounds on $L$-functions, we may insert the local $L$-factors at primes in $S$, at the cost of a multiplicative factor of $O_\varepsilon((Nd_B)^\varepsilon)$. Indeed, this follows from the control on their Dirichlet coefficients granted by the Ramanujan conjecture (power-saving improvements over the trivial bound suffice).

We now apply the Hoffstein--Lockhart bound \cite{HL} on the adjoint $L$-values, and incorporate the dependence on $d_B$ into the implied constant (as is our general convention throughout the paper), to obtain
\[
\langle f_1(\cdot)f_2(\cdot\, a(1/N)), u\rangle^2\ll_\varepsilon N^\varepsilon |I_S| |L(1/2, \pi_1\otimes \pi_2\otimes \pi)|.
\]
It therefore remains to estimate the contribution of the local integrals $I_S$.

For $p\mid d_B$, we bound $I_p$ using H\"older's inequality
\[
|I_p|\leq \|\Phi_{1,p}\|_3 \|\Phi_{2,p}\|_3\|\Upsilon_p\|_3.
\]
Instead of the full strength of the Ramanujan conjecture, we may use the Kim--Sarnak bound \cite{KS} for the $L^3$-integrability of matrix coefficients, to obtain the finiteness of the right-hand side. This yields a $O(1)$ bound depending on $p$; as $p$ divides the fixed $d_B$, we do not indicate this dependence. For $p^k\|N$, $k\geq 1$, we bound $I_p$ using Proposition \ref{prop-local-RS-bound} and the Ramanujan conjecture, which yields $I_p\ll p^{-k}$. Finally, at the real place we invoke Lemma \ref{lemma:arch-3prod}, which in the Maass case simplifies under the hypothesis that $\Re\nu_j=0$, while in the holomorphic case the characteristic function on $|\nu_\pi | \leq  |\nu_{\pi_1}| + |\nu_{\pi_2}|$ is majorized by the exponential.
\end{proof}

Let $\sigma\subset L^2_{\rm cusp}([\bG(\A)])$ and $u\in\mathscr{B}_\sigma(N)$. Using Waldspurger's theorem, along with some elementary upper bounds on local factors, we shall now bound $|W_\scD(u)|^2$ by a central $L$-value. Recall that the global order $\mathscr{O}_\mathscr{D}$ is maximal, so that $D$ is the discriminant of the underlying imaginary quadratic field $E$.

\begin{lemma}\label{lemma-L-function-sum}
We have 
\[
W_\scD(u)\ll_\varepsilon D^{-1/4+\varepsilon} N^\varepsilon L(1/2, \sigma^{\rm JL}\times \theta_E)^{1/2}.
\]
\end{lemma}

\begin{proof}
Write $\otimes_vu_v$ for the image of $u$ under a fixed isomorphism $\sigma\simeq\otimes_v\sigma_v$. Let $S$ denote the set of primes dividing $ND d_B$, along with the real place. Waldspurger's formula \cite[Proposition 7]{Wa} (cf.\ also  \cite[Theorem 1.1]{FMP} and \eqref{rho} for the definition of $\rho_E$) states that
\begin{equation}\label{eq:Walds}
|W_\scD(u)|^2= \frac{c}{\rho_E} I_{S,\scD} \frac{L^S(1/2, \sigma^{\rm JL}\times \theta_E)}{L^S (1,\chi_E)L^S(1,\sigma^{\rm JL},{\rm Ad})},
\end{equation}
for a constant $c \ll \exp(O(|S|))$ as above, where (recalling the assumption that $g_p=e$ for all primes $p$) the local integrals are given by
\[
I_{S,\scD}=I_{\infty, \scD_\infty}\prod_{v\in S} I_v,\quad   I_{\infty,\scD_\infty}=\int_{\bT(\R)} \frac{\langle\sigma_\infty(t g_\infty)u_{j,\infty},u_{j,\infty}\rangle}{\| u_{j, \infty}\|^2} \, {\rm d} t, \quad I_p=\int_{\bT(\Q_p)}\frac{\langle \sigma_p(t)u_{j,p},u_{j,p}\rangle}{\| u_{j, p}\|^2} \, {\rm d} t.
\]
We may add in the $L$-factors dividing $S$, as in the proof of Lemma \ref{lemma-triple-prod}, using $e^{|S|} \ll_{\varepsilon} (ND)^\varepsilon$. 

We may bound $I_p$, for $p\in S$ dividing $D$, by
\[
\|\langle \sigma_p(\cdot )u_p,u_p\rangle/\| u_{j, p}\|^2\|_\infty {\rm vol}(\bT(\Q_p))\leq D_p^{-1/2},
\]
where we used $|\langle \sigma_v(g)u_v,u_v\rangle|/\| u_v \|^2\leq 1$ by the Cauchy--Schwarz inequality and the measure normalizations in \S\ref{sec:periodic-measure}. For primes $p\in S$ not dividing $D$ we argue according to whether $p$ is split in $E$ or not. If $p$ is inert, then the same argument for $p\mid D$ yields
$|I_p|\leq 1$. When $p$ is split in $E$, we may find $k_p\in K_p$ such that $k_p^{-1}\bT(\Q_p)k_p$ is the diagonal subgroup of $\bG(\Q_p)\simeq \mathbf{PGL}_2(\Q_p)$, as in the proof of Lemma \ref{lemma:conversion-to-Hecke-congruence}. It follows that $\langle \sigma_p(t)u_p,u_p\rangle=\langle \sigma_p(a(x))u_p',u_p'\rangle$, with $u_p'=\sigma_p(k_p)u_p$ being invariant by $k_pK_0(p^k)k_p^{-1}$. Replacing $u_p$ by $u_p'$, and invoking the decay of matrix coefficients \cite[(9.1)]{Ven}, we have
\[
\langle \sigma_p(t)u_p ,u_p \rangle/\| u_p \|^2 \ll (1 + |t|_p)^{-\delta} \dim \sigma_p^{K_0(p^k)}
\]
for some $\delta > 0$ (we can take any $\delta < 1/2 - 7/64$, but this is irrelevant here).  Hence $I_p \ll k$. Finally, since $g_\infty^{-1}\bT(\R)g_\infty=K_\infty$, and $u_\infty$ is fixed by $K_\infty$, we have $I_{\infty,\scD}={\rm vol}(\bT(\R))\langle\sigma_\infty(g_\infty)u_\infty,u_\infty\rangle$; by the Cauchy--Schwarz inequality, $I_{\infty,\scD}=O(1)$. The lemma follows then from an application of Siegel's lower bound on $L(1,\chi_E)$ and the Hoffstein--Lockhart lower bound on $L(1,\sigma^{\rm JL},{\rm Ad})$.
\end{proof}

\subsection{The Eisenstein periods}\label{sec:Eis-cont}

It remains to bound the contributions from the Eisenstein series to Proposition \ref{lemma-spectral-decomp}. Throughout the remainder of this section we therefore have $\bG=\mathbf{PGL}_2$ and $K=\mathbf{PGL}_2(\widehat{\Z}){\rm PSO}(2)$.

Recall that $f_1$ and $f_2$ are $L^2$-normalized Hecke--Maass cusp forms on the modular surface $Y_K=[\mathbf{PGL}_2(\A)]_K={\rm PSL}_2(\Z)\backslash\mathbb{H}$, generating level 1 cuspidal automorphic representations $\pi_1$ and $\pi_2$ of $\mathbf{PGL}_2(\A)$. We now bound the triple product of $f_1$ and $f_2$ with the Eisenstein series. Just as for Lemma \ref{lemma-triple-prod}, we postpone until \S \ref{appendix:triple-prod} the relevant estimates of local factors.

\begin{lemma}\label{lemma:3prod-Eis}
Let $M^2\mid N$ and $\chi$ be a character of $\Q^\times\backslash\A^1$ with conductor $M$. Let $f\in\mathscr{B}_N(\chi)$, for the basis $\mathscr{B}_N(\chi)$ defined in \S\ref{sec:basis-choice}. Then, under the Ramanujan conjecture for $\mathbf{PGL}_2/\Q$,
\begin{align*}
\langle f_1(\cdot)f_2(\cdot\, a(1/N &)),E_{f,\chi}(\cdot, 1/2+it)\rangle\\
&\ll_\varepsilon N^{-1/2+\varepsilon} |L(1/2+it, \pi_1\times\pi_2\otimes\chi)|\min\Big(1, \exp[- \pi (|t| - |\nu_{\pi_1}| - |\nu_{\pi_2}|)]\Big).
\end{align*}
 \end{lemma}
\begin{proof}

Similarly to the proof of Lemma \ref{lemma-triple-prod} (now with $\bG=\mathbf{PGL}_2$), we fix isomorphisms $\pi_j\simeq\otimes_v\pi_{j,v}$ for $j=1,2$ and let $\otimes f_{j,v}$ be the image of $f_j$. Fix the standard non-trivial additive character $\psi_v$ of $\Q_v$. Let $\mathcal{W}_{1,v}(\pi_{1,v};\psi_v)$ be the $\psi_v$-Whittaker model of $\pi_{1,v}$ and $\mathcal{W}_{2,v}(\pi_{j,v};\overline{\psi}_v)$ the $\overline{\psi}_v$-Whittaker model of $\pi_{2,v}$. These models are non-zero since both $\pi_1,\pi_2$ are cuspidal (and hence everywhere locally generic). Denote by $W_{1,v}$ the image of $f_{1,v}$ in $\mathcal{W}_{1,v}(\pi_{1,v};\psi_v)$ under the isometric intertwining \eqref{eq:isometric-intertwine} and similarly for $W_{2,v}$.

Let $f_{\chi}(\cdot ,{\sf s})=\otimes f_{\chi_v}(\cdot ,{\sf s}) \in {\bf H}(\chi,{\sf s})=\otimes_v {\bf H}(\chi_v,{\sf s})$ be as in \S\ref{sec:spec-decomp-K(s)}. Let
\[
\alpha_v(W_{1,v},W_{2,v}, f_{\chi_v})=\int_{{\bf N}(\Q_v)\backslash\mathbf{PGL}_2(\Q_v)} W_{1,v}(g)W_{2,v}(g)f_{\chi_v}(g,{\sf s})\, {\rm d} g
\]
be the local Rankin--Selberg integral, as in \eqref{defn-local-RS}. Then for all primes $p\nmid N$, we have
\[
\alpha_p(W_{1,p},W_{2,p}, f_{\chi_p})=\frac{L_p({\sf s},\pi_{1,p}\times\pi_{2,p}\otimes\chi_p)}{L_p(1,{\rm Ad},\pi_{1,p})^{1/2}L_p(1,{\rm Ad},\pi_{2,p})^{1/2}L_p(1,\chi_p^2)^{1/2}}.
\]
Let
\[
\alpha_S=\alpha_\infty(W_{1,\infty}, W_{2,\infty}, f_{\infty}) \prod_{p^k\| N}\alpha_p(W_{1,p}, a(\varpi_p^{-k}).W_{2,p}, f_{\chi_p}).
\]
Let $S$ denote the set of primes dividing $N$, as well as the archimedean place. Then
\[
\langle f_1(\cdot)f_2(\cdot\, a(1/N)),E_{f,\chi}(\cdot, {\sf s})\rangle=\alpha_S\frac{L^S({\sf s},\pi_1\times\pi_2\otimes\chi)}{L^S(1,{\rm Ad},\pi_1)^{1/2}L^S(1,{\rm Ad},\pi_2)^{1/2}L^S(1,\chi^2)^{1/2}}.
\]
As in the proof of Lemma \ref{lemma-triple-prod}, we may add in the local $L$-factors at primes in $S$, while incurring a $O_\varepsilon(N^\varepsilon)$ loss. We then apply the Hoffstein--Lockhart bound \cite{HL} and the Siegel bound on the denominator $L$-values. It remains to estimate the local integrals $\alpha_S$. When ${\rm Re}({\sf s})=1/2$, Proposition \ref{prop-local-RS-bound1} implies $\prod_{p^k\| N}\alpha_p\ll N^{-1/2}$, where we have made use of the Ramanujan bounds. We conclude the proof by estimating $\alpha_\infty$ using Lemma \ref{lemma:first-arch-int}.\end{proof}

We now turn to bounding the level $K_0(N)$ toric period of the Eisenstein series. This bears some resemblance to \cite[Section 10]{DFI}. 

\begin{lemma}\label{lemma:Weyl-Eis-bound}
Let $(f,\chi)\in\mathscr{B}_{\rm cont}(N)$. Then
\[
W_\scD(E_{f,\chi}(\cdot, 1/2+it))\ll_\varepsilon N^\varepsilon D^{-1/4+\varepsilon}|L(1/2+it, \chi\circ \Nr_{E/\Q})|\exp(-\pi |t|/2).
\]
\end{lemma}

\begin{proof}
Let $S$ consist of the archimedean place along with all primes dividing $N$ and $D$. A variant of the Hecke--Wielonsky formula \cite{Wie}, proved by Wu \cite[Prop.\ 2.23]{Wu}, shows that\footnote{A typo appears in \cite[\S 2.9, line 4]{Wu}: the representation $\pi_{s,\xi}$ should be ${\rm Ind}_{\mathbf{B}(\A)}^{\GL_2(\A)}(\xi |\cdot |^{s},\xi^{-1} |\cdot |^{-s})$, so that the Eisenstein series has trivial central character, and not ${\rm Ind}_{\mathbf{B}(\A)}^{\GL_2(\A)}(\xi |\cdot |^{s},\xi |\cdot |^{-s})$. Note, furthermore, that our  normalization for the periodic measure on $[\bT(\A)g]$ differs from that in \cite{Wu}, which accounts for the additional factor of $D(\mathbf{F})^{-1/2}$ in \cite[Prop. 2.23]{Wu}.}
\[
|W_\scD(E_{f,\chi}(\cdot, 1/2+it))|^2 = \frac{c}{\rho_E} I_{S,\mathscr{D}}\frac{|L^S(1/2+it, \chi\circ \Nr_{E/\Q})|^2}{L^S(1,\chi_E)L^S(1+2it,\chi^2)L^S(1-2it,\chi^{-2})},
\]
where $c \ll \exp(O(|S|))$ is as above, and the local factors $I_{S,\mathscr{D}}=I_{\infty,\mathscr{D}}\prod_{p\in S}I_p$ are the same as those appearing in \eqref{eq:Walds}, with the local representation $\sigma_p$ of $\bG(\Q_p)$ replaced by the unitary principal series representation $\mathbf{H}(\chi_p,1/2+it)$ of $\mathbf{PGL}_2(\Q_p)$. The proof of Lemma \ref{lemma-L-function-sum} then applies without changes to bound the local factors. Alternatively, one could apply the Hecke--Wielonsky formula \cite{Wie} and realize the square of the local factors as the Waldspurger integral $I_v$, as was done in \cite[(43)]{ELMV2}, and once again apply the proof of Lemma \ref{lemma-L-function-sum}.
\end{proof}

\subsection{Completing the proof of Proposition \ref{prop-smallq}}

We now bound the cuspidal contribution to Proposition \ref{lemma-spectral-decomp} using standard arguments which combine the results of \S\ref{sec:estimate-cusp-cont} along with the Lindel\"of Hypothesis. 

\begin{lemma}\label{lem:cuspidal-cont}
Under the Ramanujan conjecture for $\mathbf{PGL}_2/\Q$ and the Lindel\"of Hypothesis, 
\[
\sum_{u\in\mathscr{B}_{\rm cusp}(N)} |\langle f_1(\cdot) f_2(\cdot\, a(1/N)), u\rangle  W_\scD(u)|\ll_\varepsilon D^{-1/4+\varepsilon}N^{1/2+\varepsilon}.
\]
\end{lemma}

\begin{proof}
From Lemma \ref{lemma:skew-basis} we may replace the basis $\mathscr{B}_\sigma(N)$ by the skew basis $\mathcal{B}_\sigma^{\rm skew}(N)$. We may therefore assume that $u$ is of the form $a(1/L).u^{\rm new}$, where $L\mid N/M$. Lemmata \ref{lemma-triple-prod} and \ref{lemma-L-function-sum} then together imply that
\[
\sum_{u\in\mathscr{B}_{\rm cusp}(N)} |\langle f_1(\cdot) f_2(\cdot\, a(1/N)), u\rangle  W_\scD(u)|\ll_\varepsilon D^{-1/4+\varepsilon}N^{-1/2+\varepsilon}\mathcal{L}(D,N),
\]
where
\[
\mathcal{L}(D,N)=\sum_{\substack{\pi\subset L^2_{\rm cusp}([\mathbf{PGL}_2(\A)])\\ \pi^{K_0(Nd_B)}\neq \{0\},\; |\nu_\pi|\ll N^\varepsilon}} |L(1/2,\pi_1\otimes \pi_2\otimes \pi)|^{1/2} |L(1/2, \pi\times \theta_E)|^{1/2}.
\]
The truncation to $|\nu_\pi|\ll N^\varepsilon$ is justified by pitting (when $B$ is indefinite) the exponential decay in the spectral parameter against the convexity bound on the $L$-function. By the Weyl law, there are $O(N^{1+\varepsilon})$ terms in $\mathcal{L}(D,N)$, and an application of the Lindel\"of bound to the central $L$-values in each term (i.e.\ each $L$-value is bounded by an $\varepsilon$-power of its conductor) yields result. 
\end{proof}

Putting together the results in \S\ref{sec:Eis-cont}, we obtain the following estimate on the Eisenstein term in Proposition \ref{lemma-spectral-decomp}.

\begin{lemma}\label{cor:Eis-cont}
Under the Lindel\"of Hypothesis, we have ${\rm Eis}_N(f_1,f_2)\ll_\varepsilon N^{1/2+\varepsilon}D^{-1/4+\varepsilon}$.
\end{lemma}
\begin{proof}
Inserting Lemmata \ref{lemma:3prod-Eis} and \ref{lemma:Weyl-Eis-bound}, we obtain a majorant of the form
\[
N^{-1/2+\varepsilon}D^{-1/4+\varepsilon}\sum_{\substack{M^2\mid N\\\chi\!\!\!\mod M}}\int_{{\rm Re}({\sf s})=1/2} |L({\sf s},\pi_1\times \pi_2\otimes\chi)L({\sf s}, \chi\circ \Nr_{E/\Q})|  G({\sf s})\, {\rm d} |{\sf s}|,
\]
where $G(1/2+it)$ decays exponentially with $|t|$. Estimating the contribution of the integral $|{\rm Im}({\sf s})|=|t|\gg (DN)^\varepsilon$ by pitting the exponential decay against the convexity bound for the $L$-functions  truncates the integral to $|t|\ll (DN)^\varepsilon$. The lemma then follows by applying the Lindel\"of Hypothesis and summing over all $(M,\chi)$.
\end{proof}

The second term in the bound of Proposition \ref{prop-smallq} follows from Lemmata \ref{lem:cuspidal-cont} and \ref{cor:Eis-cont}. To conclude the proof of Proposition \ref{prop-smallq}, it remains to bound $\lambda(N)$ coming from the trivial representation in \eqref{lemma-spectral-decomp}. We do so by applying the bound arising from the Ramanujan conjecture, as in Remark \ref{rem:Hecke}.

\begin{remark}\label{rem:unconditional}
For $N \ll D^{1/2}$ we may bound the right-hand side of Lemma \ref{lem:cuspidal-cont} unconditionally by a weaker, but still non-trivial, bound, using best known estimates towards subconvexity for the $L$-functions appearing in the quantity $\mathcal{L}(D,N)$. Note the factorization $L(1/2, \pi\times \theta_E)=L(1/2, \pi\times \chi_E)L(1/2, \pi)$, so that we need to bound
$$\sum_{u \text{ level } N} 1 \cdot L(1/2, u)^{1/2} \cdot L(1/2, u \times \chi_E)^{1/2} \cdot L(1/2, u \times \pi_1 \times \pi_2)^{1/2}$$
where the spectral parameter of $u$ is essentially fixed, and $\pi_1, \pi_2$ are also fixed. We apply H\"older's inequality with exponents $12, 4, 6, 2$ (in this order). For the first term we simply use Weyl's law, and for the second the easy bound $\sum_\pi L(1/2, u)^2 \ll N^{1+\varepsilon}$. For the third we use a version of \cite[Theorem 4.1]{AW} to obtain
$$\sum_{u \text{ level } N}  L(1/2, u \times \chi_E)^{3}  \ll (Nq)^{1+\varepsilon}.$$
Finally, the fourth term  is roughly like a fourth moment of $L(1/2, \pi)$. The technology of \cite{KMV} (which treats the case when $\pi_1$ and $\pi_2$ are Eisenstein) should yield the bound $N^{1+\varepsilon}$ for this sum, so that in total we obtain the estimate 
\[
\sum_{u\in\mathscr{B}_{\rm cusp}(N)} |\langle f_1(\cdot) f_2(\cdot\, a(1/N)), u\rangle  W_\scD(u)|\ll_\varepsilon D^{-1/12+\varepsilon}N^{1/2}
\]
which is non-trivial for $N \leq D^{1/6 - \varepsilon}$. Similar considerations hold for the $L$-function expression in the proof of Lemma \ref{cor:Eis-cont}.
\end{remark}

\section{Bounds on local Rankin--Selberg and triple product integrals}\label{appendix:triple-prod}

Let $v$ be a fixed place of $\Q$. Since we work locally at $v$ throughout this section, we shall often omit the dependence on $v$ from the notation.

Let $\pi_i$, $i=1,2,3$, be irreducible unitary representations of $B_v^\times$, the product of whose central characters is trivial. Let $v_i\in\pi_i$. We write
\[
I(v_1,v_2,v_3)=\int_{\bPB^\times(\Q_v)}\prod_{i=1}^3\langle \pi_i (g) v_i,v_i\rangle\,  {\rm d} g,
\]
for the local triple product integral. 

When $B=\Mat(\Q)$, we furthermore define the local Rankin--Selberg integral as follows. We suppose that one of the three representations of $\mathbf{GL}_2(\Q_v)$, say $\pi_3$, is a principal series representation $\bH(\chi,{\sf s})$. Fix the standard character $\psi$ of $\Q_v$. Let $\mathcal{W}(\pi_1,\psi)$ and $\mathcal{W}(\pi_2,\overline{\psi})$ denote the Whittaker model of $\pi_1$ and $\pi_2$, respectively.  The Jacquet integral provides for an isometric intertwining
\begin{equation}\label{eq:isometric-intertwine}
\pi_1\rightarrow \mathcal{W}(\pi_1;\psi); \qquad \phi\mapsto W(g)=\int_{\Q_v} \phi(wn(x)g)\overline{\psi}(x)\, {\rm d} x,
\end{equation}
and similarly for $\pi_2$, with $\psi$ replaced by $\overline{\psi}$. For $W_1\in \mathcal{W}(\pi_1;\psi)$, $W_2\in \mathcal{W}(\pi_2;\overline{\psi})$, and $f_3\in\bH(\chi,{\sf s})$ we put
\begin{equation}\label{defn-local-RS}
\alpha(W_1, W_2, f_3)=\int_{\mathbf{N}(\Q_v)\backslash\mathbf{PGL}_2(\Q_v)}W_1(g)W_2(g)f_3(g)\, {\rm d} g
\end{equation}
for the local Rankin--Selberg integral. 

Our goal in this section will be to bound $I(v_1,v_2,v_3)$ and $\alpha(W_1, W_2, f_3)$, for representations $\pi_i$ and for test vectors appearing in the proofs of Lemmata \ref{lemma-triple-prod} and \ref{lemma:3prod-Eis}. Similar bounds appear in the literature with varying assumptions, but we are not aware of a reference which covers our precise setting. For example, computations that are similar in spirit appear in \cite[\S\S3-4]{Hu}, but deal with at least two newforms (we have only one).

\subsection{Archimedean case}
We begin with the local Rankin--Selberg integral in the archimedean case; see also \cite[p.\ 40]{Wat}.

\begin{lemma}\label{lemma:first-arch-int}
For $j=1,2$, let $\pi_j$ be irreducible unitary principal series representations of $\mathbf{PGL}_2(\R)$ as in \S\ref{sec:reduction-to-sparse}. Fix the additive character $\psi(x)=e^{2\pi ix}$ of $\R$ and equip $\mathcal{W}(\pi_j,\psi)$ with the inner product given in the Kirillov model by $\Gamma_\R(2)\int_{\R^\times} |W(a(y))|^2{\rm d}^\times y$. Let $\chi_\infty$ be either the trivial character or the sign character of $\R^\times$, and let ${\sf s}\in\C$.

Then we may choose unit vectors $W_1\in\mathcal{W}(\pi_1,\psi)^{K_\infty}$, $W_2\in\mathcal{W}(\pi_2,\overline{\psi})^{K_\infty}$, and $f_3\in \mathbf{H}(\chi_\infty,{\sf s})^{K_\infty}$ such that, for ${\rm Re}({\sf s})>{\rm Re}\,\nu_1+{\rm Re}\,\nu_2$, 
\[
\alpha (W_1,W_2, f_3)=\frac{\Gamma_\R(2{\sf s})^{-1}L_\infty({\sf s},\pi_1\times\pi_2)}{L_\infty(1,{\rm Ad}\, \pi_1)^{1/2}L_\infty(1,{\rm Ad}\, \pi_2)^{1/2}},
\]
where $L_\infty({\sf s},\pi_1\times\pi_2)=\prod_{\pm,\pm} \Gamma_\R({\sf s}\pm \nu_1\pm\nu_2)$ and  $L_\infty(1,{\rm Ad}\, \pi_j)=\prod_{\pm} \Gamma_\R(1\pm 2\nu_j)$. In particular,
\[
\alpha (W_1,W_2, f_3)\ll_{\pi_1,\pi_2} \min\Big(1, \exp[- \pi (|t| - |\nu_{\pi_1}| - |\nu_{\pi_2}|)]\Big).
\]
\end{lemma}

\begin{proof}
Multiplying $f_3$ by a complex number of modulus 1, we may assume that $f_3(1)=1$. From the Iwasawa decomposition and the right $K_\infty$-invariance of $W_1,W_2, f_3$, we have
\[
\alpha (W_1,W_2, f_3)=\int_{\R^\times} W_1(a(y))W_2(a(y))|y|^{{\sf s}-1}\, {\rm d}^\times y.
\]
Let $W_j^{\rm new}\in \mathcal{W}(\pi_j,\psi)$ be the unique vector such that $W_j^{\rm new}(a(y))=2\, {\rm sgn}^{\epsilon_j}(y)\sqrt{|y|}K_{\nu_j}(2\pi |y|)$. We may assume that $\epsilon_1=\epsilon_2$ since otherwise the above integral vanishes. It follows from \cite[p.\ 205]{Iw2} that, for ${\rm Re}({\sf s})>{\rm Re}\,\nu_1+{\rm Re}\,\nu_2$, we have 
\[
\Gamma_\R(2{\sf s})\int_{\R^\times} W_1^{\rm new}(a(y))W_2^{\rm new}(a(y))|y|^{{\sf s}-1}\, {\rm d} y=L_\infty({\sf s},\pi_1\times\pi_2).
\]
It follows from this and the unitarity of $\pi_j$ that $\| W_j^{\rm new}\|^2=L_\infty(1,{\rm Ad}\, \pi_j)$. Multiplying $W_j$ by a complex number of modulus $1$, we may therefore assume that $W_j=L_\infty(1,{\rm Ad}\, \pi_j)^{-1/2}W_j^{\rm new}$. This proves the stated formula $\alpha (W_1,W_2, f_3)$ and the estimate follows from an application of Stirling's formula.
\end{proof}

We now bound the local archimedean triple product formula. We write $\mathcal{H}=\mathbf{B}^{(2,\infty)}(\R)$ for the Hamiltonians (over $\Bbb{R}$), and we denote by  $\left(\begin{smallmatrix} j_1&j_2&j_3\\ m_1&m_2&m_3\end{smallmatrix}\right)$ the Wigner 3-$j$ symbol. We only need to know \cite[Table 2]{Ed}
\begin{equation}\label{wig3j}
\left(\begin{matrix} j_1&j_2&j_3\\ 0&0&0\end{matrix}\right)^2 = \frac{(j_1 + j_2 - j_3)!(j_1 + j_3 - j_2)!(j_2 + j_3 - j_1)! (\frac{j_1 + j_2 + j_3}{2})!^2}{(j_1 + j_2 + j_3 + 1)! (\frac{j_1 + j_2 - j_3}{2})!^2(\frac{j_1 + j_3 - j_2}{2})!^2(\frac{j_2 + j_3 - j_1}{2})!^2}
\end{equation}
which is understood to be zero unless $j_1 + j_2 + j_3$ is even and all parentheses are non-negative, in particular $|j_1 - j_2| \leq j_3 \leq j_1 + j_2$.

\begin{lemma}\label{lemma:arch-3prod}  
${}$ 
\begin{enumerate}
\item\label{lemma:3prod-arch1} Let $\pi_j$, for $j=1,2,3$, be irreducible unitary principal series representations of $\mathbf{PGL}_2(\R)$ with spectral parameters $\nu_j$, as in \S\ref{sec:reduction-to-sparse}. Assume that $\sum_j |\Re\,\nu_j|<1/2$. Then we may choose unit vectors $v_j\in \pi_j^{K_\infty}$ such that
\[  
I(v_1,v_2,v_3)=\frac{L_\infty(1/2,\pi_1\otimes\pi_2\otimes\pi_3)}{\prod_{j=1}^3 L_\infty(1,{\rm Ad}\,\pi_j)}=\frac{\prod_{e_1,e_2,e_3\in \{\pm 1\}}\Gamma_\R(1/2+e_1\nu_1+e_2\nu_2+e_3\nu_3)}{\prod_{j=1}^3 \Gamma_\R(1)\Gamma_\R(1+2\nu_j)\Gamma_\R(1-2\nu_j)}.
\]
In particular, 
\[
I(v_1,v_2,v_3)\ll_{\pi_1,\pi_2} \min\Big(1, \exp[- \pi (|{\rm Im}\,\nu_\pi | - |{\rm Im}\, \nu_{\pi_1}| - |{\rm Im}\,\nu_{\pi_2}|) ]\Big).
\]
\item\label{lemma:3prod-arch2} Let $\sigma_j$, for $j=1,2,3$, be finite dimensional representations of $\R^\times\backslash\mathcal{H}^\times$ of dimensions $d_j=2k_j+1$. Then we may choose unit vectors $v_j\in \pi_j^{K_\infty}$ such that
\[
I(v_1,v_2,v_3)=\sqrt{\frac{d_1d_2d_3}{4\pi}}\begin{pmatrix}k_1 & k_2 & k_3\\ 0 & 0 & 0\end{pmatrix}^2,
\]
where $\left(\begin{smallmatrix} j_1&j_2&j_3\\ m_1&m_2&m_3\end{smallmatrix}\right)$ is the Wigner 3-$j$ symbol. In particular,
\[
I(v_1,v_2,v_3)\ll_{\sigma_1,\sigma_2} {\bf 1}_{[0,1]}\left(\frac{k_3}{k_1+k_2}\right).
\]
\end{enumerate}
\end{lemma}

\begin{proof}
The identity in Part \eqref{lemma:3prod-arch1} is \cite[Theorem 2.5]{Ik} (after taking into account the $L^2$ normalization of the vector $v$), and the estimate follows from this by Stirling's formula. For the second, we may identify the quotient $\R^\times\backslash\mathcal{H}^\times$ with ${\rm SO}(3)$. Then $v_j$ can be realized as the ($L^2$-normalized) zonal spherical harmonic $Y_{k_j,0}$ on $S^2$. In this case, the Gaunt integral relation (see \cite[(3.5.5)]{Ed}) gives
\[
I(v_1,v_2,v_3)=\int_0^{2\pi}\int_0^\pi Y_{k_1,0}(\theta,\varphi)Y_{k_2,0}(\theta,\varphi)Y_{k_3,0}(\theta,\varphi)d\theta d\varphi=\sqrt{\frac{d_1d_2d_3}{4\pi}}\begin{pmatrix}k_1 & k_2 & k_3\\ 0 & 0 & 0\end{pmatrix}^2,
\]
proving the identity in Part \eqref{lemma:3prod-arch2}. The estimate follows easily from \eqref{wig3j}.\end{proof}

\subsection{Non-archimedean case: Rankin--Selberg}
We shall prove the following estimate.

\begin{prop}\label{prop-local-RS-bound1}
Let $\pi_i$, for $i=1,2$, be unramified $\vartheta_i$-tempered principal series representations of ${\bf GL}_2(\Q_p)$, with trivial central character. Here $\vartheta_i\in [0,1/2)$. Assume that $\vartheta_1+\vartheta_2<1/2$. Let $\chi$ be a unitary character of $\Q_p^\times$ of conductor $p^\ell$.

Let $v_i\in \pi_i$, $i=1,2$, be the unique up-to-scalar spherical vector, normalized to satisfy $\|v_i\|=1$. Let $W_1\in \mathcal{W}(\pi_1;\psi)$ and $W_2\in \mathcal{W}(\pi_2;\overline{\psi})$ be the images of $v_1$ and $v_2$ in the Whittaker models of $\pi_1$ and $\pi_2$. Let $k\geq 2\ell$ and fix $j\in\{\ell,\ldots , k-\ell\}$. For ${\rm Re}({\sf s})=1/2$, let $f_3=f_\chi^{(j,k,\ell)}\in\bH(\chi,{\sf s})$ be as in Lemma \ref{lemma:ONB}.

Then $\alpha (W_1,a(\varpi^{-k}).W_2,f_3)\ll p^{-k(1/2-\vartheta_2)+j(\vartheta_1-\vartheta_2)}$.
\end{prop}

\begin{proof}
In the integral \eqref{defn-local-RS} defining $\alpha(W_1,W_2,f_3)$ we may restrict to the image $C_w=\mathbf{A}(\Q_p)w\mathbf{N}(\Q_p)$ in $\mathbf{N}(\Q_p)\backslash {\bf PGL}_2(\Q_p)$ of the open Bruhat cell $\mathbf{P}(\Q_p)w\mathbf{N}(\Q_p)\subset \mathbf{GL}_2(\Q_p)$, where $w=\left(\begin{smallmatrix} & 1\\1 & \end{smallmatrix}\right)$. The right-invariant measure ${\rm d}g$ on $\mathbf{N}(\Q_p)\backslash {\bf PGL}_2(\Q_p)$ factorizes in Bruhat coordinates as
\begin{equation}\label{bruhat}
C_w \ni g = a(y)w n(x)= a(y) n'(x)w  = \left(\begin{matrix}  y & \\ & 1\end{matrix}\right) \left(\begin{matrix}1 & \\ x& 1\end{matrix}\right)w, \quad {\rm d}g = \frac{{\rm d}^{\times}y}{|y|} {\rm d}x,
\end{equation}
where ${\rm d}^\times y={\rm d}_{\Q_p}^\times y$ and ${\rm d}x={\rm d}_{\Q_p}x$ are the local multiplicative and additive Haar measures fixed in \S\ref{sec:periodic-measure}. By right-$K_p$ invariance, the integral $\alpha (W_1,a(\varpi^{-k}).W_2,f_3)$ is
\[
\int_{y\in\Q_p^\times}\int_{x\in\Q_p}W_1(a(y)n'(x))W_2(a(y)n'(x)a(\varpi^{-k})) f_3(a(y)n'(x))) \,{\rm d} x \, |y|^{-1} {\rm d}^\times y. 
\]
We use $n'(x)a(y^{-1})=a(y^{-1})n'(xy^{-1})$ to write the above integral as
\[
\int_{y\in\Q_p^\times}\int_{x\in\Q_p}W_1\left(a(y)n'(x)\right)W_2\left(a(y\varpi^{-k})n'(x\varpi^{-k})\right) f_3\left(a(y)n'(x)\right) {\rm d} x\,  |y|^{-1}{\rm d}^\times y.
\]
Using $f_3(a(y)g)=|y|^{1/2}\chi(y)f(g)$ we get
\[
\int_{y\in\Q_p^\times}\int_{x\in\Q_p}W_1(a(y)n'(x))W_2(a(y\varpi^{-k})n'(\varpi^{-k}x))|y|^{-1/2} \chi(y)  f_3(n'(x)) \,{\rm d} x\, {\rm d}^\times y.
\]
We let $I$ denote the contribution to the above integral from $|x|\leq 1$ and $J$ the contribution from $|x|>1$. Note that
\[
I=\int_{y\in\Q_p^\times}\int_{x\in\Z_p}W_1\left(a(y)\right)W_2\left(a(y\varpi^{-k})n'(\varpi^{-k}x)\right)|y|^{-1/2} \chi(y)  f_3\left(n'(x)\right) {\rm d} x\, {\rm d}^\times y,
\]
from the right $\mathbf{GL}_2(\Z_p)$ invariance of $W_1$.

We begin with the range when $\ell\leq j\leq \min\{k-\ell,k-1\}$; the case $(j,\ell)=(k,0)$ will be treated separately. Inserting the definition of $f_3=f_\chi^{(j,k,\ell)}$ and changing variables, we get
\begin{equation}\label{I-before-cas-subdivision}
I=A_jp^{-j}\int_{y\in\Q_p^\times}W_1\left(a(y)\right) |y|^{-1/2} \chi(y)\left( \int_{x\in \Z_p^\times}W_2\left(a(y\varpi^{-k})n'(\varpi^{j-k}x)\right){\rm d} x\right) {\rm d}^\times y.
\end{equation}
Observe that $j-k\leq -1$. When $|X|>1$, the Iwasawa decomposition of $a(Y)n'(X)$ is
\begin{equation}\label{Iwahori-Iwasawa}
a(Y)n'(X)=\begin{pmatrix} Y& \\ X& 1\end{pmatrix}=\begin{pmatrix} X& \\ & X\end{pmatrix} \begin{pmatrix} 1 & Y/X\\ & 1\end{pmatrix}\begin{pmatrix} -Y/X^2 & \\ & 1\end{pmatrix}\begin{pmatrix} 0& 1\\ 1&1/X \end{pmatrix}.
\end{equation}
Thus, using the right $\mathbf{GL}_2(\Z_p)$ invariance and left $(\mathbf{N}(\Q_p),\bar{\psi})$-equivariance of $W_2$, we have
\[
W_2\left(a(y\varpi^{-k})n'(\varpi^{j-k}x)\right)=\psi(-y/\varpi^jx)W_2(a(-y\varpi^{k-2j})).
\]
We find that
\[
I=A_jp^{-j}\int_{y\in\Q_p^\times}W_1\left(a(y)\right)W_2(a(-y\varpi^{k-2j})) |y|^{-1/2} \chi(y)\left( \int_{x\in \Z_p^\times}\psi(-y/\varpi^jx){\rm d} x\right) {\rm d}^\times y.
\]
In the $x$-integral, we change variables $x\mapsto 1/x$ and evaluate
\[
\int_{x\in \Z_p^\times}\psi(-yx/\varpi^j){\rm d} x=\int_{x\in \Z_p}(\psi(-yx/\varpi^j)-p^{-1}\psi(-yx/\varpi^{j+1})){\rm d} x=
\begin{cases}
(1-1/p),&  y\in p^{j+1}\Z_p,\\
1,&  y\in p^j\Z_p^\times,\\
0,&  \textrm{else}.
\end{cases}
\]
Inserting this and $A_j\asymp p^{-j/2}$, we get
\[
|I|\asymp p^{-j/2} \Big|\int_{y\in p^j\Z_p}W_1\left(a(y)\right)W_2(a(-y\varpi^{k-2j})) |y|^{-1/2} \chi(y)\, {\rm d}^\times y\Big|.
\]
We now insert the bounds $W_i(a(y))\ll |y|^{1/2-\vartheta_i}{\bf 1}_{\Z_p}(y)$, and since $2j-k\leq j$, deduce
\[
p^{-k(1/2-\vartheta_2)} p^{j(1/2-2\vartheta_2)}   \int_{y\in p^j\Z_p} |y|^{1/2-\vartheta_1-\vartheta_2}\,{\rm d}^\times y\asymp p^{-k(1/2-\vartheta_2)+j(\vartheta_1-\vartheta_2)},
\]
the integral converging in light of the assumption that $\vartheta_1+\vartheta_2<1/2$.

To complete the analysis in this case, we observe that when $j\geq\max\{1,\ell\}$ the remaining integral $J$ vanishes, since when $|x|>1$, the Iwasawa decomposition \eqref{Iwahori-Iwasawa} shows that $f_3(n'(x))=|x|^{-1}\chi(-1/x^2)f_3\left(\begin{smallmatrix} 0 & 1\\ 1 & x^{-1}\end{smallmatrix}\right)=0$. When $(j,\ell)=(0,0)$, more work is required. 
In this case \eqref{Iwahori-Iwasawa} implies $f_3(n'(x))=|x|^{-1}\chi(x)^{-2}f_3\left(\begin{smallmatrix} 0 & 1\\ 1 & x^{-1}\end{smallmatrix}\right)=|x|^{-1}\chi(x)^{-2}A_0$. Applying \eqref{Iwahori-Iwasawa} to the arguments of $W_1$ and $W_2$, while using the right-$\mathbf{GL}_2(\Z_p)$-invariance and the left $(\mathbf{N}(\Q_p),\psi)$- and $(\mathbf{N}(\Q_p),\bar\psi))$-equivariance of these, we get
\[
J=A_0\sum_{t=1}^\infty \int_{y\in\Q_p^\times}W_1\left(a(y\varpi^{2t})\right)W_2\left( a(y\varpi^{k+2t})\right)|y|^{-1/2} \chi(y)\left(\int_{x\in \Z_p^\times}
 |x|^{-1}\chi(x)^{-2}  {\rm d} x \right){\rm d}^\times y.
\]
The inner integral is zero if $\chi$ is non-trivial and $1$ otherwise. In the latter case, we obtain
\begin{align*}
J&=A_0\sum_{t=1}^\infty \int_{y\in\Q_p^\times}W_1\left(a(y\varpi^{2t})\right)W_2\left( a(y\varpi^{k+2t})\right)|y|^{-1/2}\, {\rm d}^\times y\\
&\ll p^{-k(1/2-\vartheta_2)} \sum_{t=1}^\infty p^{-2t(1-\vartheta_1-\vartheta_2)}\int_{y\in p^{-2t}\Z_p- \{0\}}|y|^{1/2-\vartheta_1-\vartheta_2} \,{\rm d}^\times y, 
\end{align*}
which is $\asymp p^{-k(1/2-\vartheta_2)} \sum_{t=1}^\infty p^{-t}\asymp p^{-k(1/2-\vartheta_2)-1}$ as desired.

It remains to consider the case when $(j,\ell)=(k,0)$. Here, the $J$ integral again vanishes, and the $I$ integral in \eqref{I-before-cas-subdivision}, using right $\mathbf{GL}_2(\Z_p)$-invariance of $W_2$ and a change of variables, satisfies
\begin{align*}
|I|&  =A_kp^{-k}\Big|\int_{y\in\Q_p^\times}W_1\left(a(y)\right) |y|^{-1/2} \chi(y)\left( \int_{x\in \Z_p-\{0\}}W_2\left(a(y\varpi^{-k})n'(x)\right){\rm d} x\right) {\rm d}^\times y\Big|\\
&\asymp p^{-k/2}\Big|\int_{y\in\Q_p^\times}W_1\left(a(y)\right)W_2\left(a(y\varpi^{-k})\right) |y|^{-1/2} \chi(y) \, {\rm d}^\times y\Big|.
\end{align*}
Inserting $W_i(a(y))\ll |y|^{1/2-\vartheta_i}{\bf 1}_{\Z_p}(y)$ as before yields 
\[
I\ll p^{-k\vartheta_2}\Big|\int_{y\in p^k\Z_p-\{0\}} |y|^{1/2-\vartheta_1-\vartheta_2} \chi(y)\, {\rm d}^\times y \Big|\asymp p^{-k(1/2-\vartheta_1)},
\]
as claimed.\end{proof}

\subsection{Non-archimedean case: triple product}
We shall prove the following estimate on the local triple product integrals. 

\begin{prop}\label{prop-local-RS-bound}
Let $\pi_i$, for $i=1,2$, be unramified $\vartheta_i$-tempered principal series representations. Let $\pi_3$ be $\vartheta_3$-tempered of conductor $p^\ell$. Here $\vartheta_i\in [0,1/2)$. Assume that $\sum_i\vartheta_i<1/2$.

Let $v_i\in \pi_i$, $i=1,2,3$, be the unique up-to-scalar new vector, normalized to satisfy $\|v_i\|=1$. Let $W_1$ and $W_3$ be the image in the Whittaker model of $v_1$ and $v_3$. Let $\chi$ be the inducing character of the principal series $\pi_2$ and write $f_2$ for the spherical vector in the induced model of $\pi_2$ given by $\delta^{1/2}\chi$ as a function on $\Bbb{Q}_p^{\times}$.

Let $d,k\in\N$ be such that $\ell+d\leq k$. Then
\begin{enumerate}
\item the local triple product integral satisfies
\[
I(v_1,a(\varpi^{-k}).v_2,a(\varpi^{-d}).v_3)\ll k^4p^{-k(1-2\vartheta_2)+2d(-\vartheta_2+\vartheta_1)};
\]
\item\label{part2-RSbound} the local Rankin--Selberg integral satisfies
\[
\alpha (W_1,a(\varpi^{-k}).f_2,a(\varpi^{-d}).W_3)\ll k^2p^{-k(1/2-\vartheta_2)+d(-\vartheta_2+\vartheta_1)}.
\]
\end{enumerate}
\end{prop}

\begin{proof}
We apply the linearization identity of \cite[Lemma 3.4.2]{MV2}, valid under the assumption that $\sum_i\vartheta_i<1/2$. We use the form appearing in \cite[Lemma 6]{BJN}. It therefore suffices to prove part \eqref{part2-RSbound} of the proposition.

Using Bruhat coordinates as in \eqref{bruhat}, the integral $\alpha (v_1,a(\varpi^{-k}).f_2,a(\varpi^{-d}).v_3)$ can be seen to be
\[
\int_{y\in\Q_p^\times}\int_{x\in\Q_p}W_1(a(y)n'(x))W_3(a(y)n'(x)a(\varpi^{-d})) f_2(a(y)n'(x)a(\varpi^{-k}))\, {\rm d} x\, |y|^{-1}{\rm d}^\times y.
\]
Similarly to the proof of Lemma \ref{prop-local-RS-bound1}, we use $n'(x)a(y^{-1})=a(y^{-1})n'(xy^{-1})$, and a change of variables $x\mapsto \varpi^k x$, to write this as
\[
p^{-k}\int_{y\in\Q_p^\times}\int_{x\in\Q_p}W_1\left(a(y)n'(\varpi^k x)\right)W_3\left(a(y\varpi^{-d})n'(\varpi^{k-d}x)\right) f_2\left(a(y\varpi^{-k})n'(x)\right) \,{\rm d} x\, |y|^{-1}{\rm d}^\times y.
\]
Using $f_2(a(y)g)=|y|^{1/2}\chi(y)f(g)$ we get
\[
p^{-k/2}\chi(\varpi^{-k})\int_{y\in\Q_p^\times}\int_{x\in\Q_p}W_1(a(y)n'(\varpi^kx))W_3(a(y\varpi^{-d})n'(\varpi^{k-d}x))|y|^{-1/2} \chi(y)  f_2(n'(x))\, {\rm d} x \,{\rm d}^\times y.
\]
Note that the factor $p^{-k/2}\chi(\varpi^{-k})$ is of size $O(p^{-k(1/2-\vartheta_2)})$, so we must show that the remaining integral is $O(k^2p^{d(-\vartheta_2+\vartheta_1)})$. We write said integral as $I+J$, where $I$ is the contribution from $|x|\leq 1$ and $J$ the contribution from $|x|>1$. 

Since $n'(x)\in \mathbf{GL}_2(\Z_p)$ for $|x|\leq 1$ and $f_2(\mathbf{GL}_2(\Z_p))=f_2(e)$, we obtain
\[
I=f_2(e)\int_{y\in\Q_p^\times}\int_{x\in\Z_p}W_1\left(a(y)n'(\varpi^kx)\right)W_3\left(a(y\varpi^{-d})n'(\varpi^{k-d}x)\right)|y|^{-1/2} \chi(y)\, {\rm d} x\, {\rm d}^\times y.
\]
We now use the right $\mathbf{GL}_2(\Z_p)$ invariance of $W_1$ and the right-$\mathcal{K}_{0, p}(p^\ell)$-invariance of $W_3$ to write this as
\[
\int_{y\in\Q_p^\times}W_1\left(a(y)\right)W_3\left(a(y\varpi^{-d})\right) |y|^{-1/2}\chi(y)  \,{\rm d}^\times y\int_{x\in\Z_p}\, {\rm d} x.
\]
The $x$ integral is 1, while for the $y$-integral, we use the estimate $W_i(a(y))\ll |y|^{1/2-\vartheta_i}{\bf 1}_{\Z_p}(y)$. This yields a $y$-integral that is bounded by  
\[
p^{d(1/2-\vartheta_3)}\int_{|y|\leq p^{-d}} |y|^{-1/2+\vartheta_2-\vartheta_1-\vartheta_3} {\rm d} y\ll p^{d(-\vartheta_2+\vartheta_1)},
\]
as desired.

We now treat $J$. Using \eqref{Iwahori-Iwasawa} and the definition of $f_2$, it follows that, for $|x|>1$, $f_2(n'(x))=(\delta^{1/2}\chi)(n'(x))=|1/x^2|^{1/2}\chi(-1/x^2)=|x|^{-1}\chi(-1/x^2)$. Thus
\[
|J| \leq  \int_{y\in\Q_p^\times}\int_{|x|>1}
|W_1\left(a(y)n'(\varpi^k x)\right)W_3\left(a(y\varpi^{-d})n'(\varpi^{k-d}x)\right)| |y|^{-1/2}  |x|^{-1-2\vartheta_2}\,{\rm d} x\, {\rm d}^\times y.
\]
For $1<|x|\leq p^k$ we have $n'(\varpi^k x)\in {\rm GL}_2(\Z_p)$, in which case $W_1\left(a(y)n'(\varpi^k x)\right)=W_1\left(a(y)\right)\ll |y|^{1/2-\vartheta_1}{\bf 1}_{\Z_p}(y)$. Similarly, when $|x|>p^k$, $\big|W_1\left(a(y)n'(\varpi^k x)\right)\big|$ is
\begin{align*}
\bigg|W_1\left(\begin{pmatrix} 1 & y/(\varpi^kx)\\ & 1\end{pmatrix}\begin{pmatrix} -y/(\varpi^k x)^2 & \\ & 1\end{pmatrix}\right)\bigg|
&=\bigg|W_1\left(\begin{pmatrix} -y/(\varpi^k x)^2 & \\ & 1\end{pmatrix}\right)\bigg|\\
&\ll \begin{cases} p^{2k(1/2-\vartheta_1)} |x|^{-1+2\vartheta_1} |y|^{1/2-\vartheta_1},& |y|\leq |x|^2 p^{-2k};\\ 0,& \text{else}.
\end{cases}
\end{align*}
Thus we get
\begin{equation}\label{eq:reduce-to-W3}
\begin{aligned}
J\ll&\int_{1<|x|\leq p^k}\int_{|y|\leq 1}
|W_3(a(y\varpi^{-d})n'(\varpi^{k-d}x))|  |x|^{-1} |y|^{-\vartheta_1} \,{\rm d} x\, {\rm d}^\times y\\
&+ p^{2k(1/2-\vartheta_1)}\int_{|x|>p^k} \int_{|y|\leq |x|^2p^{-2k}}
|W_3\left(a(y\varpi^{-d})n'(\varpi^{k-d}x)\right)| |x|^{-2} |y|^{-\vartheta_1}\,{\rm d} x\, {\rm d}^\times y.
\end{aligned}
\end{equation}
We shall estimate the contribution to the integrals in \eqref{eq:reduce-to-W3} according to the following three ranges of $x$: $1<|x|\leq p^{k-(\ell+d)}$, $|x|\geq p^{k-d}$, and $p^{-\ell +k-d}<|x|\leq p^{k-d}$.

The contribution of the double integral when $1<|x|\leq p^{k-(\ell+d)}$ is rather simple to estimate, since $n'(\varpi^{k-d}x)\in\mathcal{K}_0(p^\ell)$ in that range. We obtain
\[
\int_{1<|x|\leq p^{k-(\ell+d)}}\int_{|y|\leq 1}
|W_3(a(y\varpi^{-d}))|  |x|^{-1} |y|^{-\vartheta_1}\, {\rm d}^\times y\, {\rm d} x.
\]
Using $W_3(a(y))\ll |y|^{1/2-\vartheta_3}{\bf 1}_{\Z_p}(y)$ this is at most 
\[
p^{d(1/2-\vartheta_3)}\int_{1<|x|\leq p^{k-(\ell+d)}}\int_{|y|\leq p^{-d}}
 |y|^{-1/2-\vartheta_3} |x|^{-1}\, {\rm d} y\, {\rm d} x\ll \int_{1<|x|\leq p^{k-(\ell+d)}} |x|^{-1}\, {\rm d} x,
\]
which is $O(k)$.

Next, for the range $|x|\geq p^{k-d}$, we use the following result.

\begin{lemma}\label{lemma:large-x}
For $|x|\geq p^{k-d}$ we have
\[
W_3(a(y\varpi^{-d})n'(\varpi^{k-d}x))\ll \begin{cases}
p^{(2k-d-\ell)(1/2-\vartheta_3)}|y|^{1/2-\vartheta_3}|x|^{-1+2\vartheta_3}, & |y|\leq |x|^2 p^{-2k+d+\ell};\\
0, & \textrm{else}.
\end{cases}
\]
\end{lemma}

\begin{proof} Observe that for the stated range of $x$ we have $\varpi^{d-k}x^{-1}\in\Z_p$, so that $n(\varpi^{d-k}x^{-1})\in K$. Note furthermore that for $X,Y\neq 0$, we have 
\begin{equation}\label{n'-into-Bruhat}
a(Y)n'(X)=z(X) n(Y/X)a(YX^{-2})wn(X^{-1}),
\end{equation}
where $z(X)={\rm diag}(X,X)$. Thus, using the matrix computation \eqref{n'-into-Bruhat}, we have the following equivalence up to right $\mathcal{K}_0(p^\ell)$- and left $ZN$-multiplication:
\[
a(y\varpi^{-d})n'(\varpi^{k-d}x)\equiv a(y\varpi^{d-2k} x^{-2})w\equiv a(y\varpi^{d-2k}x^{-2})w\begin{pmatrix} & 1\\ -\varpi^\ell& \end{pmatrix}.
\]
The latter is $a(y\varpi ^{d-2k}x^{-2}) a(\varpi^\ell)=a(y\varpi^{d-2k+\ell} x^{-2})$. Thus, for $|x|\geq p^{k-d}$ we obtain $W_3(a(y\varpi^{-d})n'(\varpi^{k-d}x))=W_3(a(y\varpi^{d-2k+\ell} x^{-2}))$. The lemma follows from an application of the bound $W_3(a(y))\ll |y|^{1/2-\vartheta_3}{\bf 1}_{\Z_p}(y)$.
\end{proof}

Using Lemma \ref{lemma:large-x} we find that when $p^{k-d}\leq |x|\leq p^k$, we have
\begin{align*}
\int_{p^{k-d}\leq |x|\leq p^k}&\int_{|y|\leq |x|^2 p^{-2k+d+\ell}}
|W_3(a(y\varpi^{-d})n'(\varpi^{k-d}x))| |y|^{-\vartheta_1} \, {\rm d}^\times y\, |x|^{-1} {\rm d} x\\
&\ll p^{(2k-d-\ell)(1/2-\vartheta_3)}\int_{p^{k-d}\leq |x|\leq p^k}\int_{|y|\leq |x|^2 p^{-2k+d+\ell}}
|y|^{-1/2-\vartheta_1-\vartheta_3} \, {\rm d} y\, |x|^{-2+2\vartheta_3} {\rm d} x\\
&\ll p^{(2k-d-\ell)\vartheta_1} \int_{p^{k-d}\leq |x|\leq p^k}|x|^{-1-2\vartheta_1} \, {\rm d} x,
\end{align*}
which is $O(kp^{(d-\ell)\vartheta_1})$ as desired. Similarly, when $|x|>p^k$ we have
\begin{align*}
p^{2k(1/2-\vartheta_1)}\int_{|x|>p^k} &\int_{|y|\leq |x|^2p^{-2k}}
|W_3\left(a(y\varpi^{-d})n'(\varpi^{k-d}x)\right)| |x|^{-2} |y|^{-\vartheta_1}\, {\rm d} x\, {\rm d}^\times y\\
&\ll p^{2k(1-\vartheta_1-\vartheta_3)-(d+\ell)(1/2-\vartheta_3)}\int_{|x|>p^k}\int_{|y|\leq |x|^2p^{-2k}}
|y|^{-1/2-\vartheta_1-\vartheta_3}\, {\rm d} y\,  |x|^{-3+2\vartheta_3} \, {\rm d} x  \\
 &\ll p^{k-(d+\ell)(1/2-\vartheta_3)}\int_{|x|>p^k}
 |x|^{-2-2\vartheta_1} \, {\rm d} x,
 \end{align*}
which is $O(p^{-(d+\ell)(1/2-\vartheta_3)-1-2\vartheta_1(k+1)})=O(p^{-1})$.

For the final range on $x$, we claim that every fixed $x$ such that $p^{-\ell +k-d}<|x|\leq p^{k-d}$ the $y$-integral in \eqref{eq:reduce-to-W3} is $O(k)$, uniformly in $x$. This will be enough since
\begin{align*}
\int_{p^{-\ell +k-d}<|x|\leq p^{k-d}}&\int_{|y|\leq 1}
|W_3(a(y\varpi^{-d})n'(\varpi^{k-d}x))| |y|^{-\vartheta_1}\, {\rm d}^\times y\, |x|^{-1} {\rm d} x\\
&\ll  k\int_{p^{-\ell +k-d}<|x|\leq p^{k-d}}  |x|^{-1}\, {\rm d} x= k\sum_{j=-\ell +k-d+1}^{k-d} p^{-j} (p^j(1-1/p))=O(k^2).
\end{align*}

To establish this claim, the following result will treat the $y$-integral close to $0$.

\begin{lemma}[Michel--Venkatesh]\label{lemma-MV}
There is $A>0$ such that, for $1\leq |x|\leq p^k$, we have 
\[
W_3\left(a(y\varpi^{-d})n'(\varpi^{k-d}x)\right)\ll_N p^{A k}\min (|y|^{1/2-\vartheta_3}, |y|^{-N}).
\]
\end{lemma}

\begin{proof}
Let $\mathcal{S}_d$ be the degree $d$ Sobolev norm on the Kirillov model of $\pi_3$, defined as in \cite[Proposition 2.3.2]{MV2}. We apply \cite[\S 3.2.3]{MV2} and the distortion property in \cite[ \S 2.4.1]{MV2} to get
\begin{align*}
W_3\left(a(y\varpi^{-d})n'(\varpi^{k-d}x)\right)&\ll_N\mathcal{S}_d(n'(\varpi^{k-d}x).W_3)\min (|y\varpi^{-d}|^{1/2-\vartheta_3}, |y\varpi^{-d}|^{-N})\\
&\ll p^{d(1/2-\vartheta_3)}\|{\rm Ad}(n'(\varpi^kx))\|\mathcal{S}_d(W_3)\min (|y|^{1/2-\vartheta_3}, |y|^{-N}).
\end{align*}
Now $\|{\rm Ad}(n'(\varpi^{k-d}x))\|\ll p^{kA_1}$, for some $A_1>0$, since $p^{-k}\leq p^{-k+d}\leq |x\varpi^{k-d}|\leq p^d\leq p^k$. We then have $\mathcal{S}_d(W_3)\ll_d p^{\ell A_2}$, for an $A_2>0$ depending on $d$, by the $\mathcal{K}_0(p^\ell)$-invariance of $W_3$ (see \cite[Remark 2.9]{Ne}).\end{proof}

Indeed, inserting Lemma \ref{lemma-MV} we obtain
\begin{align*}
\int_{|y|\leq p^{-6Ak}}
|W_3\left(a(y\varpi^{-d})n'(\varpi^{k-d}x)\right)| |y|^{-\vartheta_1}\, {\rm d}^\times y&\ll p^{Ak}\int_{|y|\leq p^{-6Ak}}|y|^{-1/2-\vartheta_1-\vartheta_3}\, {\rm d} y,
\end{align*}
which is $O(p^{-2Ak(1-\vartheta_1-\vartheta_3)})=O(1)$ as desired. For the remaining range on $y$, we apply the Cauchy--Schwarz inequality and \cite[Proposition 2.10]{Sa} to the remaining $y$-integral, getting
\begin{align*}
\int_{p^{-6Ak}< |y|\leq 1}
&|W_3\left(a(y\varpi^{-d})n'(\varpi^{k-d}x)\right)| |y|^{-\vartheta_1}\, {\rm d}^\times y\\[-0.2cm]
&\leq \sum_{j=-6Ak}^0p^{-\vartheta_1 j}\left(\int_{|y|=p^j}
|W_3\left(a(y\varpi^{-d})n'(\varpi^{k-d}x)\right)|^2\,  {\rm d}^\times y\right)^{1/2}\ll k.
\end{align*}
Putting everything together finishes the proof of part \eqref{part2-RSbound} of Proposition \ref{prop-local-RS-bound}.\end{proof}

\end{document}